\documentclass[12pt]{article}
\usepackage{amssymb,amsfonts,amsmath}
\usepackage{fullpage, mystyle}
\usepackage{mathtools, tensor}
\usepackage[alphabetic,y2k,lite]{amsrefs}
\usepackage{tikz-cd,MnSymbol}

\usepackage{graphicx}

\newcommand{\Zcy}{{Z_{\textrm{CY}}}}

\newcommand{\ZC}{{\cZ(\cC)}}
\newcommand{\ZZC}{{\cZ^\textrm{el}(\cC)}}
\newcommand{\cR}{\mathcal{R}}
\newcommand{\cX}{\mathcal{X}}
\newcommand{\cY}{\mathcal{Y}}
\newcommand{\cD}{\mathcal{D}}
\newcommand{\cG}{\mathcal{G}}
\newcommand{\wdtld}[1]{\widetilde{#1}}
\newcommand{\op}{\textrm{op}}
\newcommand{\cop}{\textrm{cop}}
\newcommand{\induct}[1]{{X_j {#1} X_j^*}}
\newcommand{\iinduct}[1]{{X_i X_j {#1} X_j^* X_i^*}}

\newcommand{\rtdual}[1]{\tensor[^*]{#1}{}}
\newcommand{\dcop}{{*,\cop}}
\newcommand{\SLZ}{{\textrm{SL}_2(\mathbb{Z})}}
\newcommand{\torus}{{\mathbb{T}^2}}
\newcommand{\punctorus}{{\mathbb{T}_1^2}}
\newcommand{\lmb}{\lambda}
\newcommand{\tnsrbar}{\overline{\tnsr}}
\newcommand{\tnsrproj}[2]{\tensor[_{#1}]\tnsrbar{_{#2}}}
\newcommand{\picmid}[1]{\begin{matrix}{#1}\end{matrix}}
\newcommand{\cII}{{\cI^\textrm{el}}}
\newcommand{\cFF}{{\cF^\textrm{el}}}
\newcommand{\cDD}{{\cD^\textrm{el}}}
\newcommand{\oneel}{{\one^\textrm{el}}}
\newcommand{\Cat}{{\textrm{Cat}}}
\newcommand{\modl}{{\textrm{mod}}}

\newcommand{\Vect}{\mathcal{V}ec}  

\newtheorem{defn-prop}{Definition-Proposition}[section]

\begin{document}

\title{The Elliptic Drinfeld Center of a Premodular Category}
\author{Ying Hong Tham}
\date{}
\maketitle

\abstract{
  Given a tensor category $\cC$, one constructs its Drinfeld center
  $\ZC$ which is a braided tensor category,
  having as objects pairs $(X,\lmb)$,
  where $X \in \Obj(\cC)$ and $\lmb$ is a half-braiding.
  For a \emph{premodular category} $\cC$,
  we construct a new category $\ZZC$
  which we call the \emph{Elliptic Drinfeld Center},
  which has objects $(X, \lmb^1, \lmb^2)$, where the $\lmb^i$'s
  are half-braidings that satisfy some compatibility conditions.
  We discuss an $\SLZ$-action
  on $\ZZC$ that is related to the anomaly appearing in
  Reshetikhin-Turaev theory.
  This construction is motivated from the study of the
  extended Crane-Yetter TQFT,
  in particular the category associated to the once punctured torus.
}

\section{Introduction and Preliminaries}

In \ocite{CY}, Crane and Yetter define a
4d TQFT using a state-sum involving 15j symbols,
based on a sketch by Ooguri \ocite{Oo}.
The state-sum begins with a coloring of the
2- and 3-simplices of a triangulation of the four manifold
by integers from $0,1,\ldots,r$.
These 15j symbols then arise as the evaluation
of a ribbon graph living on the boundary of a 4-simplex.
The labels $0,1,\ldots,r$ correspond to simple objects of
the Verlinde modular category,
the semi-simple subquotient of the category of finite
dimensional representations of the quantum group
$U_q \mathfrak{sl}_2$ at $q = e^{\pi i/r+2}$
as defined in \ocite{AP}.
Later Crane, Kauffman, and Yetter \ocite{CKY}
extend this definition to colorings with objects from
a premodular category (i.e. artinian semisimple tortile/ribbon).\\

The invariant for closed 4-manifolds that one obtains 
from the Crane-Yetter (CY) state-sum
essentially boils down to the signature and Euler characteristic
of the manifold, though it is still interesting because
it expresses the signature of a 4-manifold in terms of local
combinatorial data \ocite{CKY_evaluate}.
It is believed that the CY TQFT, with a modular category
as input, is a boundary theory,
in that $\Zcy(M^4)$ for a 4-manifold with boundary
is determined by its boundary $\del M^4$
and classical invariants of $M^4$ like the signature
and Euler characteristic.\\

In \ocite{CKY}, the authors speculate that their theory,
when extended to include insertions at surfaces and points,
could be related to Donaldson-Floer (DF) theory \ocite{D}.
Attempts have been made (e.g. \ocite{Y},\ocite{Ro})
to modify the state-sum in \ocite{CKY}
in the presence of insertions on surfaces.
In principle, these insertions are labellings of a
codimension-2 submanifold by objects in a
category associated to the abstract homeomorphism class
of that submanifold. These categories should be related
to each other via some gluing axioms.\\

The construction presented in this paper arose out of
studying such categories.
Namely, starting with a fixed premodular category $\cC$
(which would be used to produce the CY TQFT),
we have an abstract schema of producing a category
$\Zcy(\Sigma)$ for each surface $\Sigma$
(possibly with punctures and boundaries).
In brief,
the basic objects in $\Zcy(\Sigma)$ are configurations of finitely many
points, each labelled with an object in $\cC$;
morphisms are skein modules with appropriate boundary conditions.
One then completes the category by considering the Karoubian closure.\\

In \ocite{BBJ1}, \ocite{BBJ2}, Ben-Zvi, Brochier, and Jordan 
use factorization homology to construct such categories,
integrating certain algebras over surfaces.
We expect that our constructions agree.\\

Although we have abstractly defined these categories $\Zcy(\Sigma)$
from skeins,
the goal of our studies is to relate them to the input
premodular category $\cC$.
For example, we have that
\begin{align*}
  \Zcy(\mathbf{D}^2)
    &= \cC \textrm{ (by default)}\\
  \Zcy(S^2)
    &= \cZ_{\textrm{M\"{u}}}(\cC), \textrm{ the M\"{u}ger center}\\
  \Zcy(S^1 \times [0,1])
    &= \cZ(\cC), \textrm{ the Drinfeld center}\\
  \Zcy(\punctorus)
    &= \ZZC\\
  \Zcy(\mathbb{T}^2)
    &= \mathcal{V}ec \textrm{ for } \cC \textrm{ modular}
\end{align*}
where $\ZZC$ is the category we construct in this paper,
which we call the \emph{Elliptic Drinfeld center}.
We will establish these results in future work,
as our main goal in this paper is to define and study
$\ZZC$.\\

Briefly, $\ZZC$ has as objects $(X,\lambda^1, \lambda^2)$,
where $\lambda^1, \lambda^2$ are half-braidings on $X$.
$\lambda^1$ and $\lambda^2$ are required to satisfy
certain commutativity relations involving the braiding on $\cC$.
Thus we stress that while the Drinfeld center can be
defined for any monoidal category, our elliptic Drinfeld center
requires that $\cC$ be \emph{braided}.
The other conditions (fusion, ribbon) are not essential for
the definition but are needed to define the (extended) TQFT.
They also lead $\ZZC$ to have nice properties.\\

Choose a pair of oriented simple closed curves
on $\punctorus$ so that $\punctorus$ deformation retracts onto
their union
(see \rmkref{rmk:pop} for a picture).
Then there is a functor
\[
  \ZZC \xrightarrow[]{\sim} \Zcy(\punctorus)
\]
that sends $(X,\lmb,^1,\lmb^2)$ to the image of
a projection on the configuration
with one marked point labelled by $X$.
The projection is built out of $\lmb^1$ and $\lmb^2$,
somehow assigning them to the two chosen curves.
In hand-wavy terms,
$\ZZC$ is a ``coordinate representation" of
$\Zcy(\punctorus)$, in the sense that we have picked
a marked point and a pair of such curves
in order to express our objects with ``coefficients" in $\cC$,
and this ``coordinate representation" changes when
we change these choices.
Further discussions of this can be seen in e.g.
\rmkref{rmk:sl2z_induct_forget}, \secref{sec:sl2z},
\rmkref{rmk:elliptic_slz}).\\

Let us give a brief outline of the paper.
In \secref{sec:drinfeld_center},
we first recall some properties of the usual
Drinfeld center $\ZC$.
In \secref{sec:elliptic},
we then discuss various properties of $\ZZC$
in parallel with those of $\ZC$ laid out in
\secref{sec:drinfeld_center}.
We show that $\ZZC$
is monoidal (see Definition-Proposition \ref{def:ZZC_tnsr}).
Being the category associated to $\punctorus$,
it naturally carries an action of $\SLZ$
(see \thmref{thm:slz_action}).
However, some of the arguments are of a topological nature,
and is more naturally understood in the context of
the extended Crane-Yetter TQFT,
hence to limit the scope of the paper,
we postpone full proofs to future work.\\

As mentioned above, when $\cC$ is modular,
it is expected that $\Zcy$ is a boundary theory.
Since $\Zcy(\mathbf{D}^2) = \cC$,
one expects $\Zcy(\punctorus) \cong \cC$ as well.
To this end, we prove in \secref{sec:modular_case} that:\\

\noindent 
\textbf{\thmref{thm:modular}.}
\emph{
  If $\cC$ is modular, then the composition
  \[
    i = \cI_1 \circ \iota : \cC \to \ZZC
  \]
  is an equivalence of abelian categories,
  where $\iota: \cC \to \ZC$ is the functor $X \mapsto (X,c_{-,X})$,
  and $\cI_1$ is the intermediate induction functor
  defined in \prpref{prp:intermediate_adjoint_1}.
}\\

In \secref{sec:modularRT}, we discuss
the connection of the $\SLZ$-action on $\ZZC$
with the anomaly in Chern-Simons/Reshetikhin-Turaev theory
via \thmref{thm:modular}.
However, in part due to the reliance of this
action on the $\SLZ$-action on $\ZZC$,
we've decided to omit some details and proofs
and once again relegate them to future work.\\

In \secref{sec:C_Hmod},
we consider $\cC = H-\modl$, where $H$ is a Hopf algebra.
In this case,
the Drinfeld center is equivalent to the category of modules over
Drinfeld's quantum double, $\cD(H)$.
In the same spirit, when $H$ is \emph{braided}, we construct an algebra
$\cDD(H)$, which we call the \emph{Elliptic Drinfeld double},
such that\\

\noindent
\textbf{\thmref{thm:ZZRep_alg}.}
\emph{
  For $\cC = H-\modl$, $\ZZC \cong \cDD(H)-\modl$
  as abelian categories.
}\\

Brochier and Jordan \ocite{BJ} defined an algebra
which they also call the elliptic double,
and also arising from studying the category associated to
the once-punctured torus.
These algebras are not isomorphic, but we expect them
to be Morita equivalent.
In \ocite{BJ}, they also obtain an $\wdtld{\SLZ}$-action
on their elliptic double; we touch on this briefly in
\rmkref{rmk:elliptic_slz}.\\

When $\cC$ is symmetric, there is a tensor product on $\ZZC$
different from the one defined in \secref{sec:tensor_product}.
When $H$ is cocommutative, $\cDD(H)$ has a ribbon Hopf structure,
thus $\cDD(H)-\modl$ is a tensor category.
Then with respect to these monoidal structures,
the equivalence in \thmref{thm:ZZRep_alg} is one of tensor categories.\\

In \secref{sec:conclusion},
we discuss a generalization of our construction of $\ZZC$,
corresponding to considering surfaces other than $\punctorus$.
Finally, the last section is an Appendix,
with some useful lemma that are frequently used in
computing with string diagrams,
and a discussion of group actions on categories
given by generators and relations.\\

To conclude this section, let us compare the structures on the
elliptic Drinfeld center with the usual Drinfeld center.
Beginning with a monoidal category $\cC$,
the Drinfeld center $\ZC$ is a braided monoidal category.
On the other hand, beginning with a braided monoidal category $\cC$,
the elliptic Drinfeld center $\ZZC$ is a monoidal category
but not braided (this is discussed in \secref{sec:tensor_product},
but full proofs will be given in future work).
In addition, $\ZZC$ carries an action of $\SLZ$.\\

This difference is a feature of the topology of surfaces:
as mentioned above, the Drinfeld center is associated to the
annulus, while the elliptic Drinfeld center is associated to
the once-punctured torus.
In both cases, the monoidal structure arises from
a generalized pair of pants\footnotemark
(see \rmkref{rmk:pop}).
For the annulus, this generalized pair of pants
is just a thicked pair of pants,
so has a homeomorphism swapping the two inputs,
making the monoidal structure a braided one,
while for the once-punctured torus,
this generalized pair of pants does not admit
such a swapping operation.\\

Note that in both cases, the (braided) monoidal structure
differs from that of $\cC$:
thinking of $\cC$ as an $E_2$ algebra,
the monoidal structure manifests as
inclusion of little disks in the little disks operad,
so in a very loose sense governs ``local behaviour".
However, the monoidal structures on $\ZC$ and $\ZZC$
are an artefact of global topology of the relevant surfaces.
For example, the braided structure on $\ZC$
is quite different from that of $\cC$
- it is constructed using only the monoidal structure of $\cC$.
The takeaway is that we should not think of ``gaining"
or ``losing" structures from $\cC$
when we construct $\ZC$ and $\ZZC$,
but rather observe that they merely reflect the topology of surfaces.
\\

\footnotetext{This is not the generalized pair of pants
in the sense of Floer theory in symplectic geometry.}

\emph{Acknowledgements.} 
The author thanks Alexander Kirillov and Jin-Cheng Guu
for many helpful conversations.

\subsection{Notation and Conventions}

Throughout, let us fix an algebraically closed field $\kk$
of characteristic 0.\\

Let $\cC$ be a $\kk$-linear premodular category, that is,
a ribbon fusion category. Some assumptions and notations:
\begin{itemize}
  \item For simplicity of exposition and minimality of parentheses,
    we suppress applications of the associativity constraint
    unless it leads to confusion.
  \item We implicitly identify $V^{**}$ with $V$
    via the pivotal structure $\delta_V : V \to V^{**}$.
  \item We denote the braiding by $c_{A,B}: A\tnsr B \to B\tnsr A$.
  \item $\cC^\bop$ is the same underlying category as $\cC$
    but with \textbf{op}posite \textbf{b}raiding.
  \item The set of isomorphism classes of simples is denoted by $J$,
    and we fix a representative $X_j$ for each $j\in J$.
    $0\in J$ will index the unit object, $X_0 = \one$.
  \item We fix isomorphisms $\vphi_i : X_j^* \to X_{j^*}$ compatible
    with the pivotal structure, i.e.
    $\delta_{X_j} = \vphi_i^* \circ \vphi_{i^*}^\inv : X_i \to X_i^{**}$.
  \item Evaluation, coevaluation maps are
  \begin{align*}
    \ev_X : X^* \tnsr X \to \one\\
    \coev_X : \one \to X \tnsr X^*\\
    \wdtld{\ev}_X = \ev_{X^*} : X \tnsr X^* \to \one\\
    \wdtld{\coev}_X = \circ \coev_{X^*} : \one \to X^* \tnsr X
  \end{align*}
    (where in the third and fourth line we suppressed the pivotal map
    $\delta_X: X \to X^{**}$)
  \item The categorical dimension of $X_j$ is denoted
    $d_j = \dim_\cC X_j \in \End(\one)$.
    For each $j$, we fix a square root $\sqrt{d_j}$.
    The dimension of $\cC$ is denoted
    $\cD = \sum_j d_j^2$,
    and we will assume that $\cD \neq 0$.
    We also fix a square root $\sqrt{\cD}$.
  \item Quite often we will omit the symbol $\tnsr$,
    so that concatenation of objects denote tensor products,
    e.g. $c_{A,B} : AB \to BA$.
  \item We use an ``Einstein convention":
    when latin lowercase alphabet appear in dual pairs,
    they will be summed over the set of simple objects $J$.
    For example,
    $\induct{ }$ is short for $\dirsum \induct{ }$.
\end{itemize}

We will describe morphisms using graphical calculus
(see for example \ocite{BakK}, \ocite{K}).
Here are some conventions:
\begin{itemize}
  \item All diagrams represent morphisms in $\cC$;
    morphisms in the other categories that show up, $\ZC$ and $\ZZC$,
    are subspaces of morphisms in $\cC$.
  \item Our convention will be that morphisms go from the bottom
    object to the top.
  \item If a string is shown without orientation, it is going up
    by default.
  \item In string diagrams, some strings with be labelled with
    a lowercase latin alphabet.
    This means we are meant to sum the diagram over $J$.
    There is a similar notion for greek letters (see Appendix).
  \item Dashed lines will stand for the sum over all colorings of an edge/loop
    by simple objects $j$, each taken with coefficient $d_j$:
    \[
      \includegraphics[height=60pt]{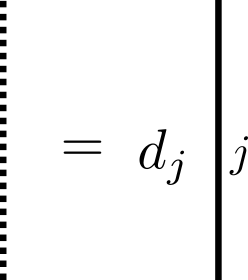}
    \]
    (note this is to be summed over $j\in J$, as mentioned above)
  \item A pair of morphisms labelled with the same greek letter
    (sometimes with an overline)
    will denote a sum over a pair of dual bases with respect to
    a certain pairing -
    see appendix for details.
\end{itemize}

We refer the reader to the appendix for examples, useful identities,
and further clarification.\\

\begin{remark}
All of our constructions are purely algebraic,
but we try to explain their topological underpinnings.
Thus, topological discussion will be a little sloppy;
in particular, we will confuse boundaries and punctures
on a surface unless they lead to confusion.
\end{remark}

\section{The Drinfeld Center}
\label{sec:drinfeld_center}

Let us recall the construction and properties of the Drinfeld center.
There is nothing new here, so the expert may skip to \secref{sec:elliptic};
we include this so as to make the constructions and proofs
for the elliptic Drinfeld center more transparent
and to set some notation.\\

The following construction is due to Drinfeld (unpublished),
and appears in \ocite{Maj},\ocite{JS}:

\begin{definition}
  The \emph{Drinfeld center} $\ZC$ of a monoidal category $\cC$
  is a category consisting of the following:\\

  An object of $\ZC$ is a pair $(X,\lmb)$,
  where $X$ is an object of $\cC$ and
  $\lmb$ is a half-braiding on $X$, i.e. a natural transformation
  $\lmb: - \tnsr X \to X \tnsr -$ that respects tensor products,
  i.e. satisfies the equation on the left below:
  \[
    \picmid{
      \includegraphics[height=80pt]{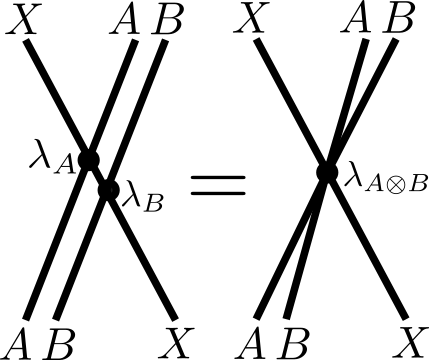}
    }
    \hspace{20pt}
    ;
    \hspace{20pt}
    \picmid{
      \includegraphics[height=80pt]{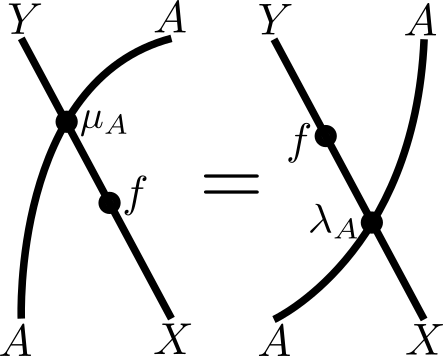}
    }
  \]
  The morphisms $\Hom_{\ZC}((X,\lmb), (Y, \mu))$
  are the subspace of those morphisms in $\Hom_\cC(X, Y)$
  that intertwine the half-braidings $\lmb,\mu$
  (equation on the right above).
\hfill $\triangle$
\end{definition}

A more concise way to simultaneously state the naturality of $\lmb$
and the above condition on $\lmb$ is the following,
which will be used frequently to manipulate diagrams and prove equations:
\[
  \includegraphics[height=80pt]{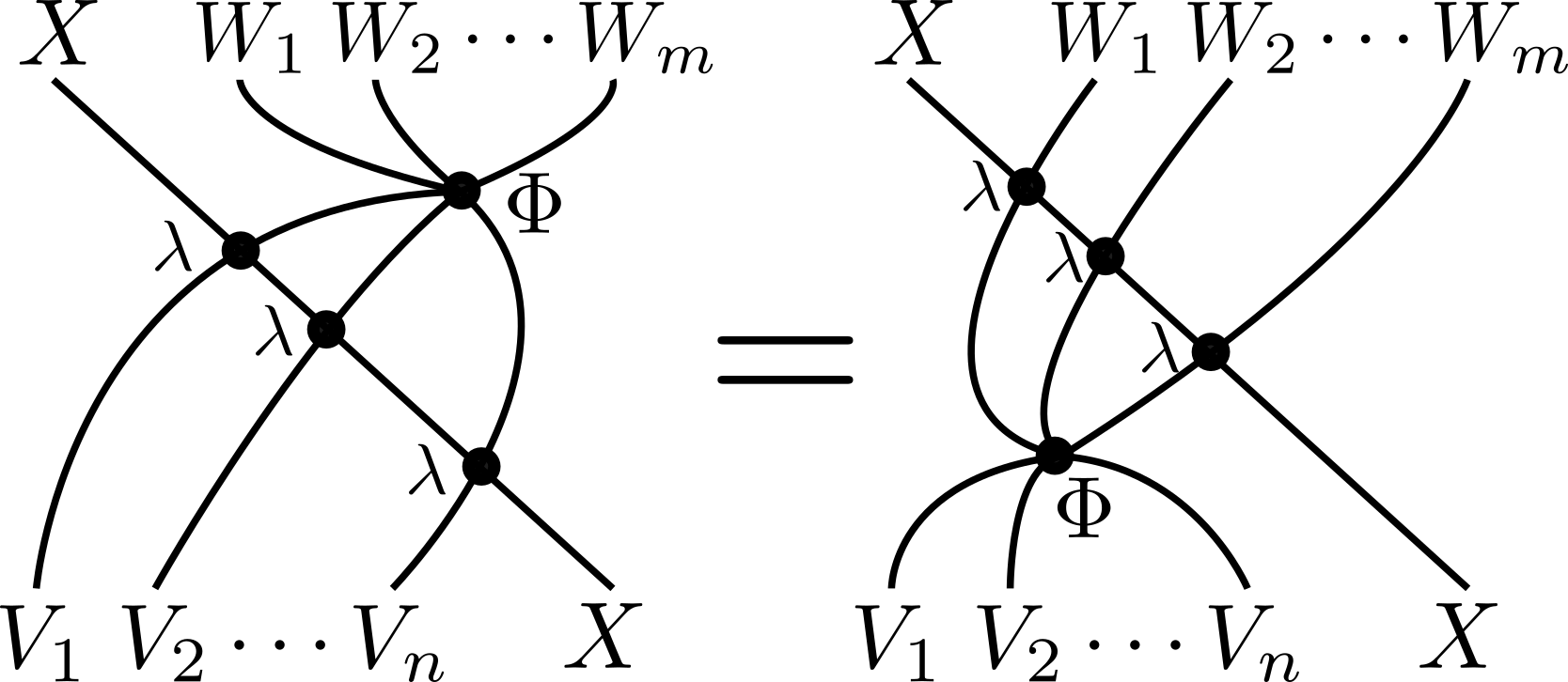}
\]

When $\cC$ is spherical fusion,
a useful alternative description of $\Hom_{\ZC}((X,\lmb),(Y,\mu))$
is as the image of a projection,
which will be very useful for checking that
a certain morphism is actually in $\ZC$:

\begin{lemma}
\label{lem:projection}
  Let $(X, \lmb),(Y,\mu) \in \Obj \ZC$.
  Define the operator
  \[
    P_{\lmb,\mu} \lcirclearrowright \Hom_\cC(X,Y)
  \]
  as follows:
  \[
    \includegraphics[height=60pt]{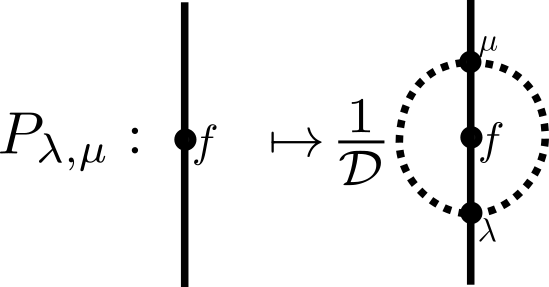}
  \]
  Then $P_{\lmb,\mu}$ is a projector onto the subspace
  $\Hom_\ZC((X,\lmb), (Y,\mu)) \subseteq \Hom_\cC(X,Y)$.

\end{lemma}

\begin{proof}
See e.g. \ocite{BalK}*{Lemma 2.2}
\end{proof}

\subsection{Properties of $\ZC$}

Let us recall some well-known facts about $\ZC$.
In this section we will work with spherical fusion categories $\cC$.

\begin{proposition}{\ocite{muger2}}
\label{prp:ZC_modular}
  Let $\cC$ be a spherical fusion category. Then $\ZC$ is modular.
\end{proposition}

We give a sketch of a proof
and relevant constructions since we will be using
similar techniques for the new category.
Our proof differs slightly from \ocite{muger2},
particularly proof of semisimplicity and finiteness.
The expert may wish to skip to the next section
and refer back later.

\begin{proofsketch}
We first show it is \textbf{abelian}. It is clearly additive.
The kernel, cokernel, and image of a morphism
$f:(X,\lmb) \to (Y,\mu)$ is obtained from the kernel, cokernel, and image
of $f$ thought of as a morphism in $\cC$, and the object inherits
a half-braiding from $X$ or $Y$.
We illustrate this in more detail for the kernel,
since we will repeat the construction for the elliptic Drinfeld center.
The cokernel and image follow a similar pattern.\\

For an exact sequence $0 \to K \to X \to Y$,
we have the following commutative diagram:
\[
\begin{tikzcd}
  0 \ar[r]
    & A \tnsr K \ar[r, "\id_A \tnsr \iota"] \ar[d]
    & A \tnsr X \ar[r, "\id_A \tnsr f"] \ar[d, "\lmb_A"]
    & A \tnsr Y \ar[d, "\mu_A"]\\
  0 \ar[r]
    & K \tnsr A \ar[r, "\iota \tnsr \id_A"]
    & X \tnsr A \ar[r, "f \tnsr \id_A"]
    & Y \tnsr A
\end{tikzcd}
\]
The top and bottom rows are exact by the exactness of
$A \tnsr -$ and $- \tnsr A$.
The leftmost vertical arrow exists and is unique by universal property
of kernels.
The half-braiding condition
is automatically satisfied by the uniqueness of this arrow.\\

Denote by $\lmb|_K$ the half-braiding on $K$ constructed above
(i.e. the vertical arrow in the commutative diagram above),
so that the candidate kernel constructed above is
$(K,\lmb|_{\iota}$, or simply $(K, \lmb|_K)$
when there is no confusion.
We still need to show that this object satisfies 
the universal property of kernels.
Consider the diagram $(W,\zeta) \overset{g}{\to} (X,\lmb) \to (Y, \mu)$
which composes to $gf = 0$.
As morphisms in $\ZC$ are subsets of morphisms in $\cC$,
$gf = 0$ in $\ZC$ implies 
$gf = 0$ as morphisms in $\cC$,
so $g$ must factor uniquely in $\cC$ through $\iota$,
so that there exists a unique $\cC$-morphism $h: W \to K$
such that $g = \iota h$.\\

To see that $h$ is a morphism in $\ZC$, i.e. intertwines half-braidings,
consider the following diagram:
\[
\begin{tikzcd}
  A \tnsr W \ar[drr, "\id_A \tnsr g"] \ar[dr, "\id_A \tnsr h" description]
              \ar[d, "\zeta_A"']
  \\
  W \tnsr A \ar[drr, "g \tnsr \id_A" near end] \ar[dr, "h \tnsr \id_A"']
  & A \tnsr K \ar[r, "\id_A \tnsr \iota"] \ar[d, "(\lmb|_K)_A" near start]
  & A \tnsr X \ar[d, "\lmb_A"]
  \\
  {}
  & K \tnsr A \ar[r, "\iota \tnsr \id_A"']
  & X \tnsr A
\end{tikzcd}
\]
We need to show that the front left parallelogram commutes.
This follows from the facts that: (1) all other faces commute,
(2) so composing the parallelogram in question with the bottommost arrow 
$K\tnsr A \overset{\iota\tnsr\id_A}{\to} A\tnsr K$ commutes,
and (3) this arrow ($\iota \tnsr \id_A$) is monic.\\

For the cokernel of $f$, there is a similar construction of half-braiding
on $C = \coker_\cC(f)$ using the universal property of cokernels,
and we will denote the half-braiding 
inherited from $(Y,\mu)$ by $\overline{\mu}^C$.
For the image, one simply notes that the two ways of constructing
a half-braiding on $I = \im f$ \big(from $\ker(\coker(f))$
and $\coker(\ker(f))$\big)
agree, so that $\overline{\lmb}^I = \mu|_I$.\\

\textbf{Semisimplicity} of $\ZC$:
Let $(K,\lmb|_K) \overset{\iota}{\hookrightarrow} (X,\lmb)$ be a subobject
(any other subobject
$(K, \mu)\overset{\iota'}{\hookrightarrow} (X,\lmb)$
is isomorphic to $(K,\lmb|_K)$
since by monicity, $\iota' = \ker(\coker(\iota'))$ in $\ZC$).

In particular, $K \overset{\iota}{\hookrightarrow} X$ is a subobject in $\cC$,
so by semisimplicity of $\cC$, there is a $\cC$-morphism $X \to K$
such that $p \circ \iota = \id_K$.
Consider $\bar{p} = P_{\lmb, \lmb|_K} (p)$.
Since $\iota$ is a $\ZC$-morphism, we have
$P_{\lmb, \lmb|_K} (p) \circ \iota
= P_{\lmb|_K, \lmb|_K} (p \circ \iota)
= P_{\lmb|_K, \lmb|_K} (\id_K)
= \id_K$,
thus $\iota, \bar{p}$ are an inclusion-projection pair
that exhibites $(K, \lmb|_K)$ as a direct summand of $(X,\lmb)$.\\

$\ZC$ has \textbf{finitely many simple objects}:
this is a consequence of \prpref{prp:direct_summand},
which asserts that $(Y,\lmb) \in \ZC$
is a direct summand of $\cI Y$.
So if $(X,\lmb)$ is a simple object,
with $X = \dirsum_j X_j^{\oplus n_j}$,
then $(X,\lmb) \subseteq \dirsum (\cI X_j)^{\oplus n_j}$,
hence it must be a subobject of some $\cI X_j$.
Since $\End_{\ZC} (\cI X_j)$ is finite dimensional,
there can only be finitely many simple subobjects
of $\cI X_j$.
Finally, there are only finitely many simples $X_j$ in $\cC$.\\

\textbf{Tensor structure:}
The tensor product of two objects $(X,\lmb)$ and $(Y,\mu)$
is given by
\[
  (X,\lmb) \tnsr (Y,\mu) := (X\tnsr Y, \lmb \tnsr \mu)
\]
where the tensor product $X \tnsr Y$ is from $\cC$, and
\[
  (\lmb \tnsr \mu)_A
    := (\id_X \tnsr \mu_A) \tnsr (\lmb_A \tnsr \id_Y)
    = \picmid{\includegraphics[height=30pt]{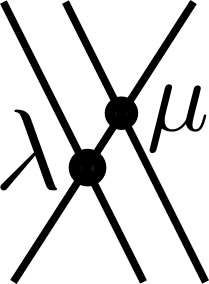}}
\]
The associativity constraint is given by the one from $\cC$,
and easily seen to respect the half-braiding.

The \textbf{unit} object is $(\one, \id_{-})$
(with the left/right unit constraints from $\cC$).
It has endomorphism ring $\kk$, so is simple.\\

\textbf{Rigidity:}
$(X,\lmb)$ has left dual $(X^*, \lmb^*)$,
where $(\lmb^*)_A = (\lmb_{\rtdual{A}})^*$.
Similarly, the right dual is $(\rtdual{X}, \rtdual{\lmb})$,
where $\rtdual{\lmb}_A = \rtdual{(\lmb_{A^*})}$.
A simple computation shows that the (co)evaluation maps
on $X$ are morphisms $\one \to (X,\lmb) \tnsr (X^*,\lmb^*)$ etc.
The pivotal structure on $\cC$ naturally induces one on $\ZC$,
because $\delta_X : X \to X^{**}$ is always a morphism
$\delta_X : (X,\lmb) \to (X^{**}, \lmb^{**})$
(see e.g. \ocite{EGNO}*{Exercise 7.13.6}).
It is clearly still spherical on $\ZC$.

The \textbf{braiding} is given by the half-braiding of the second factor:
\[
  \tilde{c}_{(X,\lmb),(Y,\mu)} = \mu_X
    : (X,\lmb) \tnsr (Y,\mu) \to (Y,\mu) \tnsr (X,\lmb)
\]
We do not prove modularity here as it will not be needed later,
referring the reader to \ocite{muger2}, \ocite{EGNO}*{Corollary 8.20.14}.

\end{proofsketch}

\begin{proposition}
\label{prp:adjoint}
The forgetful functor $\cF: \ZC \to \cC$
has a two-sided ``induction" adjoint functor
$\cI: \cC \to \ZC$,
where on objects, $\cI$ sends
\[
  X \mapsto (\induct{X}, \Gamma),
\]
where
\[
  \includegraphics[height=60pt]{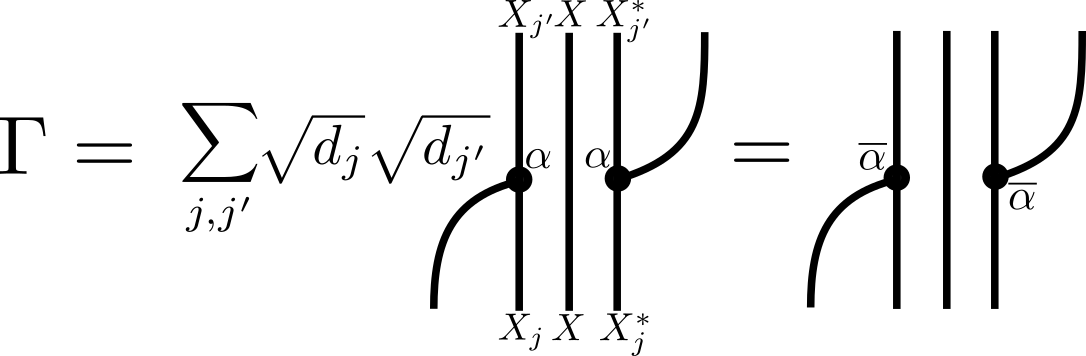}
\]
(where $\alpha$, $\overline{\alpha}$ are defined in the appendix)
and on morphisms, $f\in \Hom_\cC(X,Y)$,
\[
  f \mapsto \sum_i \id_{X_i} \tnsr f \tnsr \id_{X_i^*}
\]
The adjunction is given by the functorial isomorphisms
\begin{align*}
  \Hom_\cC(X,Y) \cong \Hom_\ZC(\cI X, (Y,\mu))
  \hspace{10pt} & \hspace{10pt}
  \Hom_\cC(X,Y) \cong \Hom_\ZC((X,\lmb), \cI Y) \\
  \includegraphics[height=150pt]{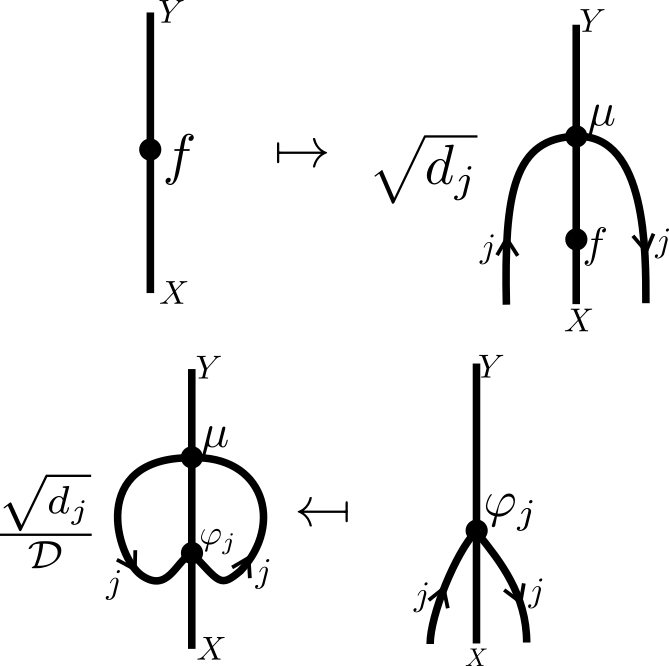}
  \hspace{10pt} & \hspace{10pt}
  \includegraphics[height=150pt]{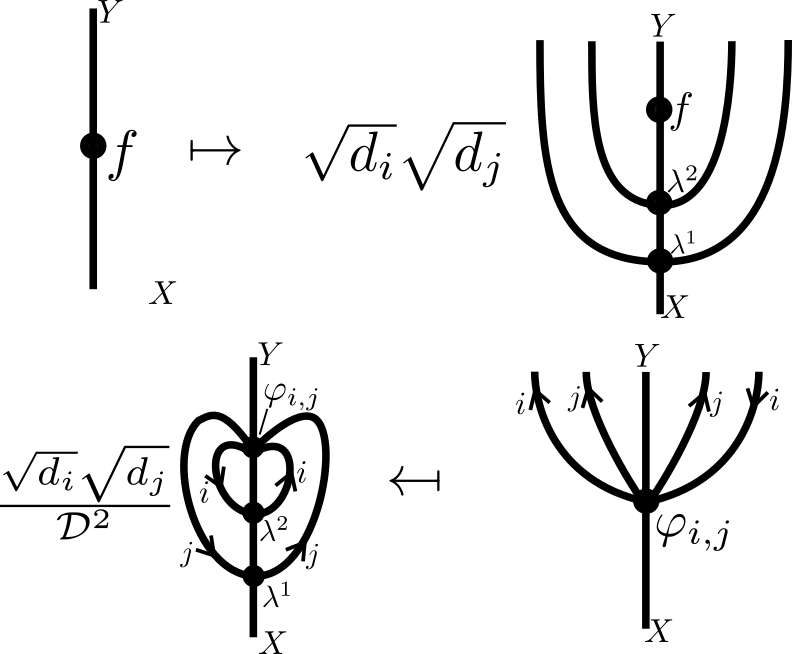}
\end{align*}

\end{proposition}

\begin{proof} 
By \lemref{lem:Gamma_hf_brd}, $\Gamma$ is a half-braiding.
It is also easy to check that $\cI$ is a well-defined functor,
and that the maps between $\Hom$ spaces are indeed isomorphisms,
natural in each each variable.
We refer the reader to \ocite{K}*{Theorem 8.2} and
\ocite{BalK}*{Theorem 2.3}
for more details.
\end{proof}

Note that $\cI$ is not a monoidal functor,
while $\cF$ is naturally a tensor functor,
but is not braided tensor.\\

A useful consequence of this adjunction is a description of
morphisms $\Hom_\ZC(\cI X, \cI Y)$:

\begin{corollary}
\label{cor:hom_IX_IY}
$\Hom_\ZC(\cI X, \cI Y)$ consists of morphisms of the form
\[
  \includegraphics[height=70pt]{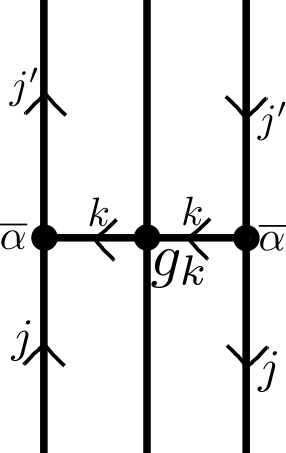}
\] where $g_k \in \Hom_\cC(X, X_k Y X_k^*)$.
\end{corollary}

The point here is that the ``struts" in the middle are
labelled by the same simple object, and is independent of $j,j'$
(up to the $\sqrt{d_j}\sqrt{d_{j'}}$ factor hidden in $\overline{\alpha}$);
these are not true for a general morphism
$f \in \Hom_\cC(\induct{X}, X_{j'}YX_{j'}^*)$. \\

\begin{proposition}
  \label{prp:direct_summand}
$(Y,\mu) \in \ZC$ is a direct summand of $\cI Y$.
In particular,
$\ZC$ is the Karoubian completion of the full image of $\cI$.
\end{proposition}
\begin{proof}
It is easy to see that the morphism defined below (on the left)
projects $\cI Y$ onto a direct summand that is isomorphic to $(Y,\mu)$.
For example, the first half of this projection is a morphism
from $\cI$ to $(Y,\mu)$ in $\ZC$:
\[
  \includegraphics[height=60pt]{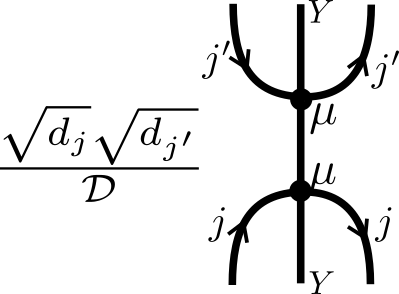}
  \hspace{10pt};\hspace{10pt}
  \includegraphics[height=60pt]{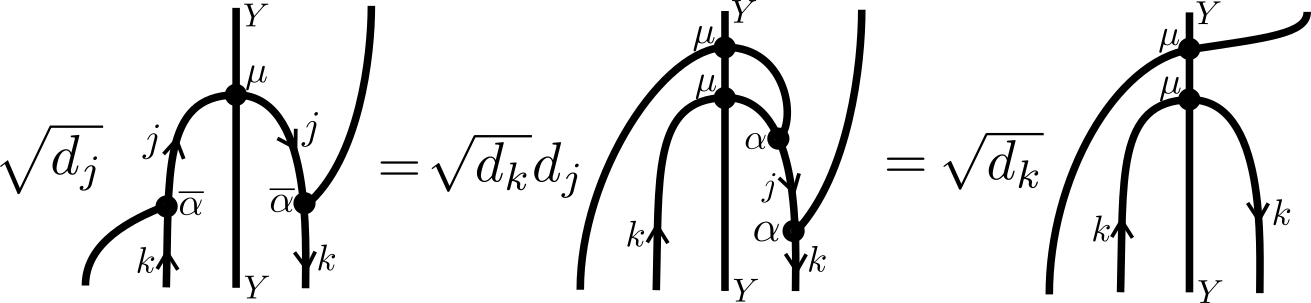}
\]
(note the change from $\alpha$ to $\overline{\alpha}$ in the first equality;
in the second equality, we use \lemref{lem:identity_combine}.)
\end{proof}

\subsection{Special $\cC$'s}
In this section we study what happens when we take special $\cC$'s,
in particular when $\cC$ is modular and when $\cC$ is given
as the category of finite-dimensional representations of a
Hopf algebra $H$. We will be considering analogs of these
results for the elliptic Drinfeld center in
Sections \ref{sec:modular_case} and \ref{sec:C_Hmod}.

\subsubsection{$\cC$ Modular}
Since $\ZC$ is modular, one may expect interesting things to happen
when $\cC$ itself is modular. Indeed, one has the following:

\begin{proposition}{\ocite{ENO1}}
\label{prp:modular_factorizable}
If $\cC$ is modular, then
$\ZC \simeq_{\tnsr, \text{br}} \cC \boxtimes \cC^\bop$,
where $\cC^\bop$ is the same underlying fusion category
with the opposite braiding,
and $\boxtimes$ is Deligne's tensor product \ocite{De1}.
\end{proposition}

\begin{proofsketch}
There are braided tensor functors
$\cC \to \ZC$ and $\cC^\bop \to \ZC$
given by $X \mapsto (X, c_{-,X})$ and $X \mapsto (X, c_{X,-}^\inv)$
respectively, where recall that $c_{-,-}$ is the braiding on $\cC$.
These fit together into a braided tensor functor
$\cC \boxtimes \cC^\bop \to_{\tnsr, \text{br}} \ZC$,
and the modularity of $\cC$ ensures that this is fully faithful.
To show essential surjectiveness,
one checks that this functor hits all the simple objects of $\ZC$
by counting dimensions of endomorphism algebras $\End_\ZC(\induct{X_k})$,
or one checks that their Frobenius-Perron dimensions are the same,
as in \ocite{EGNO}*{Prop 8.20.12}.
\end{proofsketch}

We will use this in studying the elliptic Drinfeld center $\ZZC$
when $\cC$ is modular in \secref{sec:modular_case},
in particular it will be the key fact in proving 
\lemref{lem:I1_kills_braiding}.\\

\subsubsection{$\cC = H-\modl$}
\label{sec:drfld_symmetric}
Next we consider when $\cC = H-\modl$,
the category of finite-dimensional modules over
a finite-dimensional spherical Hopf algebra $H$.
We outline the construction of $\cD(H)$,
Drinfeld's quantum double of $H$, defined in \ocite{Dr},
which is a ribbon Hopf algebra (in the sense of \ocite{RT}),
and show that $\ZC \simeq \cD(H)-\modl$.
Since we work with semisimple $\cC$,
we implicitly assume that $H$ is semisimple,
even though the construction of $\cD(H)$ does not use semisimplicity.
Most of this is follows \ocite{EGNO}*{Section 7.14};
see also \ocite{Kas}*{Section XIII.5}.\\

\begin{defn-prop}
\label{def:drfld_double_alg}
Let $(H, m, 1, \Delta, \veps, S, v)$ be a finite-dimensional
spherical Hopf algebra.
The \emph{Drinfeld double} of $H$, denoted $\cD(H)$,
is a ribbon Hopf algebra defined as folows:
\begin{itemize}

\item As a coalgebra, it is $H \tnsr H^{*,\cop}$,
where $H^{*,\cop} = (H, \Delta^*, \veps, (m^*)^\cop, 1, - \circ S^\inv)$
is the dual Hopf algebra with opposite comultiplication.

\item As an algebra, the obvious inclusions
$H \cong H \tnsr 1 \hookrightarrow H \tnsr H^{*,\cop}$ and
$H^{*,\cop} \cong 1 \tnsr H^{*,\cop}
  \hookrightarrow H \tnsr H^{*,\cop}$
are algebra maps, and the commutation relation
is given by
\[
  fh =
  \eval{f_3,S^\inv(h_1)}\eval{f_1,h_3}
    h_2 f_2
\]
where we use Sweedler's notation
$\Delta^2(h) = h_1 \tnsr h_2 \tnsr h_3$ and
$\Delta^2(f) = ((m^*)^\cop)^2(f) = f_3 \tnsr f_2 \tnsr f_1$
(note the opposite numbering is used in \ocite{EGNO}).

\item The antipode is given componentwise,
i.e. $S(hf) = f(S^\inv(\cdot)) S(h)$

\item $v \in H \hookrightarrow \cD(H)$ is the pivotal element.

\item The $R$-matrix is
$\sum h_i \tnsr h_i^* \in \cD(H) \tnsr \cD(H)$,
where $\{h_i\}$ is a basis of $H$,
and $\{h_i^*\}$ the dual basis of $H^{*,\cop}$.
\hfill $\triangle$
\end{itemize}

\end{defn-prop}

\begin{proof}
Straightforward elementary computations,
e.g. see \ocite{EGNO}*{Section 7.14}.
\end{proof}

\begin{example}[Group Algebra]
\label{xmp:group_algebra}
For $H=\kk[G]$, where $G$ is a finite group,
$H^\dcop \cong F(G)^\cop$, the Hopf algebra of functions on $G$
with the opposite comultiplication
$m^\dcop(\delta_g) = \sum_{g_1g_2 = g} \delta_{g_2} \tnsr \delta_{g_1}$.
By definition, $\{g\}_{g\in G}$ serves as a basis for $\kk[G]$;
let $\{\delta_g\}_{g\in G}$ be the corresponding dual basis
of $F(G)^\cop$.
Then in these bases, for $h \in G$ and $\delta_g \in F(G)^\cop$,
the commutation relations between 
$\kk[G]$ and $F(G)^\cop$ is simply
\[
  \delta_g h = h \delta_{h^\inv g h}
\]

Denote $\cD(G) := \cD(\kk[G])$.
Using this explicit description of $\cD(G)$,
we can interpret representations of $\cD(G)$ as
$G$-equivariant bundles over $G$, where $G$ acts on itself
by conjugation.
Briefly, the $\delta_g$ are projections,
giving us a (vector space) decomposition of a representation
$V$ of $\cD(G)$
into $\dirsum_{g \in G} V_g$, where $V_g = \delta_g V$.
Then for $v_g = \delta_g v \in V_g$,
\[
  h \cdot v_g
  = h \delta_g \cdot v
  = \delta_{hgh^\inv} h\cdot v
  \in V_{hgh^\inv}
\]
thus the bundle with $V_g$ sitting over $g\in G$
is $G$-equivariant.\\

For each conjugacy class $\bar{g} \in \overline{G}$,
the sum $\delta_{\bar{g}} = \sum_{g\in \bar{g}} \delta_g$ is
a central idempotent,
and the collection of such $\delta_{\bar{g}}$ is pairwise
orthogonal and sum to 1.
So the category of finite-dimensional representations
$\cD(G)-\modl$ is semisimple with simple objects
$V_{\bar{g}, \pi}$ labelled by pairs $(\bar{g},\pi)$,
where $\bar{g}\in \bar{G}$ is a conjugacy class of $G$
and $\pi \in \widehat{Z(g)}$ is an isomorphism class
of irreducible representations of the centralizer $Z(g)$
of some representative $g\in \bar{g}$.
We refer the reader to \ocite{BakK}*{Section 3.2} for
further details.
(Note that in terms of our set up in \defref{def:drfld_double_alg},
they are working with $\cD(F(G)) \cong \cD(\kk[G^\op])^\cop$.)\\

Let us also note that $\cD(\kk[G]) \cong \cD(F(G))$,
but this will not quite hold true for the elliptic double
$\cDD(H)$ defined later in \defref{def:elliptic_double}
- there the input Hopf algebra $H$ must at least be ribbon,
so $\cDD(F(G))$ is not even defined,
unless $G$ is abelian.
We will discuss the elliptic analog of this example in
\xmpref{xmp:group_algebra_elliptic}.
\hfill $\triangle$
\end{example}

\begin{proposition}
\label{prp:ZRep}
For a finite-dimensional Hopf algebra $H$, 
let $\cC = H-\modl$,
the category of finite-dimensional left $H$-modules.
Then $\ZC \simeq_{\tnsr, \text{br}} \cD(H)-\modl$.
\end{proposition}

\begin{proofsketch}

We essentially follow \ocite{EGNO}*{Section 7.14},
referring the reader to it for more details.\\

The functor $\cD(H)-\modl \to \ZC$ is constructed as follows.
Let $X$ be a left $\cD(H)$-module. It is in particular
an $H$-module, i.e. an object in $\cC$.
As an object of $\cC$, it has a natural half-braiding,
given by the $R$-matrix of $\cD(H)$:
for $A \in \cC$ some $H$-module, define
\[
  \lmb_A : A \tnsr X \to X \tnsr A
\]
to be the linear map given by first acting
by $R = \sum h_i \tnsr h_i^*$
(the action is defined on $A \tnsr X$
because the first factors appearing in $R$ are in $H$),
and then swapping the factors.
This is a half-braiding on $X$ because
$(\Delta \tnsr \id)(R) = R^{13}R^{23} \in H \tnsr H \tnsr H^\dcop$.
Thus, we have defined a functor
$\cD(H)-\modl \to \ZC$.\\

For the other way,
let $(X, \lmb) \in \ZC$.
By functoriality of $\lmb$ and finite-dimensionality,
$\lmb$ is completely determined by
$\lmb_H: H \tnsr X \to X \tnsr H$.
Then we define for $f\in H^*$,
\[
  f \cdot x = \eval{\id_X \tnsr f, \lmb_H (1 \tnsr x)}
\]
Said otherwise, it is the action of $H^*$ on $X$
such that $\lmb = P \cdot R$.
There are commutation relations between these actions,
which can be derived from looking at the
Yang-Baxter equation on $H \tnsr X \tnsr H^*$,
and one sees that they are precisely those
as imposed on $\cD(H)$.\\

Since morphisms of $\ZC$ are precisely those that intertwine
half-braidings, it is easy to see that they are
also precisely those that intertwine the $H^*$-actions,
so that we have a fully faithful functor back $\ZC \to \cD(H)-\modl$,
and it's not hard to see that it is an equivalence.\\

So far we haven't used the coalgebra structure of $\cD(H)$.
The monoidal structure on $\ZC$ should carry over to
$\cD(H)-\modl$, and the claim is that it agrees
with that which arise from the coalgebra structure on $\cD(H)$.
We can see this as follows:
The tensor product of $(X,\lmb)$ and $(Y,\mu)$
is $(X \tnsr Y, \lmb \tnsr \mu)$,
where recall $\lmb \tnsr \mu$ is just braiding by $\lmb$
and then $\mu$, so
\[
  (\lmb \tnsr \mu)_H : 1 \tnsr x \tnsr y
    \mapsto h_i^*\cdot x \tnsr h_j^*\cdot y \tnsr h_jh_i
    = (m^*)^\cop(h_k^*) \cdot (x \tnsr y) \tnsr h_k
\]
so the action of $f$ is
\[
  f \cdot (x \tnsr y)
  = \eval{1\tnsr f, (\lmb \tnsr \mu)_H(1 \tnsr (x \tnsr y)}
  = \eval{f,h_k} (m^*)^\cop(h_k^*) \cdot (x \tnsr y)
  = (m^*)^\cop(f) \cdot (x \tnsr y)
\]

So the functors we have defined here are tensor functors,
and in fact braided.
The pivotal structure on $\cC$ is naturally a pivotal
structure on $\ZC$,
and these functors clearly respect the pivotal structures.
Once again we refer the reader to \ocite{EGNO}*{Section 7.14}
for more details.
\end{proofsketch}

\begin{corollary}
  If $H$ is semisimple, then so is $\cD(H)$.
\end{corollary}

\section{The Elliptic Drinfeld Center}
\label{sec:elliptic}

In this section we define the \emph{Elliptic Drinfeld Center}
$\ZZC$ of a \emph{braided} monoidal category $\cC$.
It is analogous to the Drinfeld center $\ZC$,
in that objects consist of an object of $\cC$
and, not one, but two half-braidings which are related by
some equation involving the braiding on $\cC$.
As mentioned in the introduction, the motivation for constructing
$\ZZC$ comes from studying the value of the extended
Crane-Yetter TQFT on a once-punctured torus.
We discuss some of the properties of the
elliptic Drinfeld center parallel to those of
the Drinfeld center as discussed in \secref{sec:drinfeld_center}.
We put a monoidal structure on $\ZZC$
in \secref{sec:tensor_product},
and discuss an $\SLZ$-action on $\ZZC$ in \secref{sec:sl2z}.
Just as with $\ZC$, $\ZZC$ has particularly nice descriptions
when $\cC$ is modular and when $\cC = H-\modl$ for a
\emph{quasi-triangular} Hopf algebra $H$,
which we discuss in Sections \ref{sec:modular_case}
and \ref{sec:C_Hmod} respectively.

\subsection{Definition ad Properties of $\ZZC$}

\begin{definition}
\label{def:ZZC}
Given a braided monoidal category $\cC$, the \emph{Elliptic Drinfeld Center}
of $\cC$, denoted $\ZZC$, is the category with objects of the form
$(X, \lmb^1, \lmb^2)$,
where $X$ is an object in $\cC$,
and $\lmb^1,\lmb^2: - \tnsr X \to X \tnsr -$
are half-braidings on $X$ satisfying
the following compatibility condition
which we refer to as \textbf{COMM}:
\[
  \textrm{\textbf{COMM:}}\hspace{5pt}
  \picmid{\includegraphics[height=70pt]{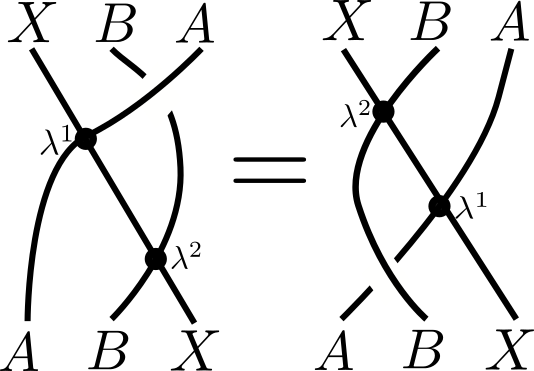}}
\]

We point out that on the l.h.s
one has the universal braiding $c_{-,-}$,
while on the r.h.s. it is the \emph{reverse} $c_{-,-}^\inv$.
We also note that COMM is a condition on an \emph{ordered} pair
of half-braidings, in that if $(\lmb^1,\lmb^2)$ satisfies COMM,
it does not imply that $(\lmb^2,\lmb^1)$ satisfies COMM;
however in some sense $\lmb^1,\lmb^2$ should be treated
on equal footing - see \rmkref{rmk:first_second_hfbrd}.\\

The morphisms from $(X,\lmb^1,\lmb^2)$ to $(Y, \mu^1, \mu^2)$
are given by those in $\Hom_\cC(X,Y)$ that intertwine
both half-braidings, i.e.
\[
  \Hom_\ZZC((X,\lmb^1,\lmb^2), (Y,\mu^1,\mu^2))
  := \Hom_\ZC((X,\lmb^1),(Y,\mu^1)) \bigcap \Hom_\ZC((X,\lmb^2),(Y,\mu^2))
\]
\hfill $\triangle$
\end{definition}

\begin{remark}
\label{rmk:COMM_strands}
As a visual aid, it is helpful to think of COMM as allowing
the strands labelled $A$ and $B$ to pass through each other
- at the moment they meet, they should be transverse to each other,
so that if the $A$ strand was above, then after meeting,
the $A$ strand would be below.\\
Morally, the $A$ strand goes around a meridian in the once-punctured
torus, while the $B$ strand goes around a longitude,
and a pair of meridian and longitude can be isotoped to
intersect once and intersect transversally.\\
Later in \secref{sec:conclusion}, we discuss a variant of COMM,
where instead of being transverse,
the strands become tangent at the moment of meeting.
This reflects the fact that two embedded closed curves
in a thrice-punctured sphere can be isotoped to not intersect
each other.
\hfill $\triangle$\\
\end{remark}

For the rest of the section, we will work with premodular $\cC$.
A simple consequence of COMM is the following
analog of \lemref{lem:projection}:

\begin{lemma}
\label{lem:projection_elliptic}
  Let $(X,\lmb^1,\lmb^2), (Y,\mu^1,\mu^2)$ be objects in $\ZZC$.
  The projections
  $P_{\lmb^1,\mu^1}, P_{\lmb^2,\mu^2} \lcirclearrowright \Hom_\cC(X,Y)$
  defined in \lemref{lem:projection} commute,
  and hence we have
  \[
    \Hom_\ZZC((X,\lmb^1,\lmb^2), (Y,\mu^1,\mu^2))
    = \im(P_{\lmb^1,\mu^1} \circ P_{\lmb^2,\mu^2})
    \subseteq \Hom_\cC(X,Y).
  \]
\end{lemma}

\begin{proof}
\[
  \includegraphics[width=0.9\textwidth]{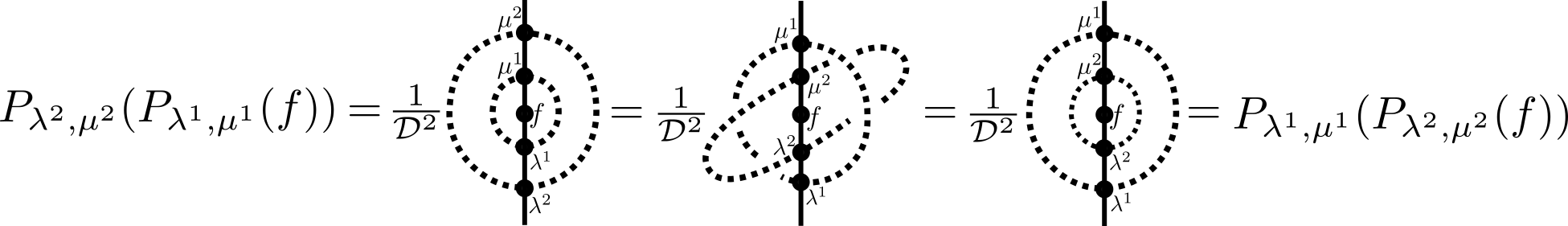}
\]
The second equality uses COMM.
\end{proof}

\begin{proposition}
  Let $\cC$ be a premodualr category.
  Then $\ZZC$ is abelian and semisimple,
  and has finitely many simple objects.
\end{proposition}

\begin{proof}

The proof is similar to that of \prpref{prp:ZC_modular},
and we will use some notation established there.\\

$\ZZC$ is clearly additive.
The proof that $\ZZC$ is \textbf{abelian}
is pretty much the same as for $\ZC$.
For example, to get the kernel of a morphism
$f: (X, \lmb^1, \lmb^2) \to (Y, \mu^1, \mu^2)$,
first find the kernel of $f$ as a morphism in $\cC$,
say $\iota: K \hookrightarrow X$.
$K$ inherits two half-braidings $\lmb^1|_K,\lmb^2|_K$
from $X$ by restricting along $\iota$
(see proof sketch of \prpref{prp:ZC_modular}).
These two half-braidings satisfy COMM
naturally from the universal property of kernels:
\[
\begin{tikzcd}
  KBA \ar[r, hookrightarrow] \ar[rrr, equal, bend left]
    & XBA \ar[r, equal]
    & XBA \ar[r, hookleftarrow]
    & KBA
  \\
  KAB \ar[u] \ar[r, hookrightarrow]
    & XAB \ar[u, "\id_X \tnsr c_{A,B}"]
    & BXA \ar[u, "\lmb_B^2 \tnsr \id_A"'] \ar[r, hookleftarrow]
    & BKA \ar[u]
  \\
  AKB \ar[u] \ar[r, hookrightarrow]
    & AXB \ar[u, "\lmb_A^1 \tnsr \id_B"]
    & BAX \ar[u, "\id_B \tnsr \lmb_A^1"'] \ar[r, hookleftarrow]
    & BAK \ar[u]
  \\
  ABK \ar[u] \ar[r, hookrightarrow] \ar[rrr, equal, bend right]
    & ABX \ar[u, "\id_A \tnsr \lmb_B^2"] \ar[r, equal]
    & ABX \ar[u, "c_{B,A}^\inv \tnsr \id_X"'] \ar[r, hookleftarrow]
    & ABK \ar[u]
\end{tikzcd}
\]
(The inner octagon commutes because
$\lmb^1,\lmb^2$ satisfy COMM by definition,
so the outer octagon commutes,
hence $\lmb^1|_K, \lmb^2|_K$ satisfy COMM.
The equal signs should really be some associativity constraints,
but we suppress them.)\\

So $\ker_\ZZC(f) = (K, \lmb^1|_K, \lmb^2|_K)$.
Likewise, $\coker_\ZZC(f) = (C, \overline{\mu^1}^C, \overline{\mu^2}^C)$,
and $\im (f) = (I, \overline{\lmb^1}^I = \mu^1|_I,
                   \overline{\lmb^2}^I = \mu^2|_I)$,
where $C,I$ are respectively the cokernel, image in $\cC$.\\

\textbf{Semisimplicity} of $\ZZC$ follows from the same argument for $\ZC$
almost verbatim - we again start with
$(K,\lmb^1|_K,\lmb^2|_K) \overset{\iota}{\hookrightarrow} (X,\lmb^1,\lmb^2)$,
noting that abelian-ness of $\ZZC$
implies that any subobject of $(X,\lmb^1,\lmb^2)$ is of this form.
In particular, $K \overset{\iota}{\hookrightarrow} X$ is a subobject in $\cC$,
so by semisimplicity of $\cC$, there is a $\cC$-morphism $X \to K$
such that $p \circ \iota = \id_K$.
Consider $\bar{p} = P_{\lmb^2, \lmb^2|_K}P_{\lmb^1, \lmb^1|_K} (p)$.
Since $\iota$ is a $\ZZC$-morphism, we have
$P_{\lmb^2, \lmb^2|_K}P_{\lmb^1, \lmb^1|_K} (p) \circ \iota
= P_{\lmb^2, \lmb^2|_K}P_{\lmb^1, \lmb^1|_K} (p \circ \iota)
= P_{\lmb^2, \lmb^2|_K}P_{\lmb^1, \lmb^1|_K} (\id_K)
= \id_K$,
thus $\iota, \bar{p}$ are an inclusion-projection pair
that exhibites $(K, \lmb^1|_K, \lmb^2|_K)$ as a direct summand of
$(X,\lmb^1,\lmb^2)$.\\

$\ZZC$ has \textbf{finitely many simple objects}:
similarly to the proof of \prpref{prp:ZC_modular},
this is a consequence of \prpref{prp:elliptic_direct_summand},
which asserts that $(Y,\mu^1,\mu^2)$ is a direct summand of
$\cII Y$. Then if $(X,\lmb^1,\lmb^2)$ is simple
with $X = \dirsum_j X_j^{\oplus n_j}$,
then $(X,\lmb^1,\lmb^2) \subseteq \dirsum_j (\cII X_j)^{\oplus n_j}$,
hence it must be a subobject of some $\cII X_j$.
Since $\End_{\ZZC} (\cII X_j)$ is finite dimensional,
there can only be finitely many simple subobjects
of $\cI X_j$.
Finally, there are only finitely many simples $X_j$ in $\cC$.\\

\end{proof}

The following is the analog of \prpref{prp:adjoint}:

\begin{proposition}
\label{prp:elliptic_adjoint}
The forgetful functor $\cFF : \ZZC \to \cC$ has a two-sided
adjoint $\cII : \cC \to \ZZC$,
where on objects, $\cII$ sends
\[
  X \mapsto (\dirsum_{i,j} X_iX_jXX_j^*X_i^*, \Gamma^1, \Gamma^2)
\]
where
\begin{align*}
  \includegraphics[height=40pt]{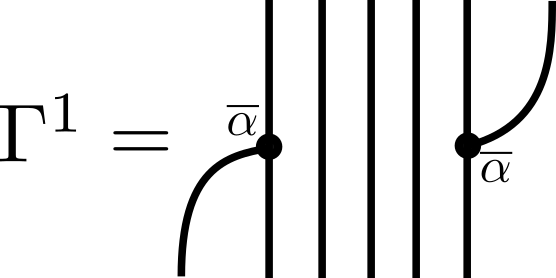},
  \hspace{10pt}
  \includegraphics[height=40pt]{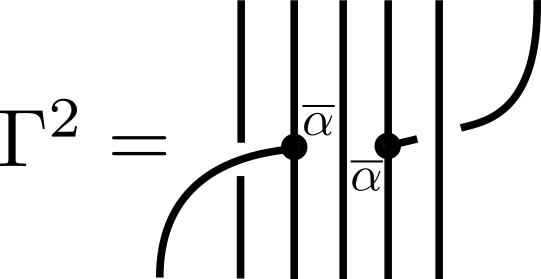},
\end{align*}
(where $\overline{\alpha}$ is defined in the appendix)
and on morphisms, $f\in \Hom_\cC(X,Y)$,
\[
  \cII(f) = \dirsum_{i,j} \id_{X_i X_j} \tnsr f \tnsr \id_{X_j^* X_i^*}
\]

The adjunction is given by the functorial isomorphisms
\begin{align*}
  \Hom_\cC(X,Y) \cong \Hom_\ZZC(\cII X, (Y,\mu^1,\mu^2))
  \hspace{10pt} & \hspace{10pt}
  \Hom_\cC(X,Y) \cong \Hom_\ZZC((X,\lmb^1,\lmb^2), \cII Y) \\
  \includegraphics[height=150pt]{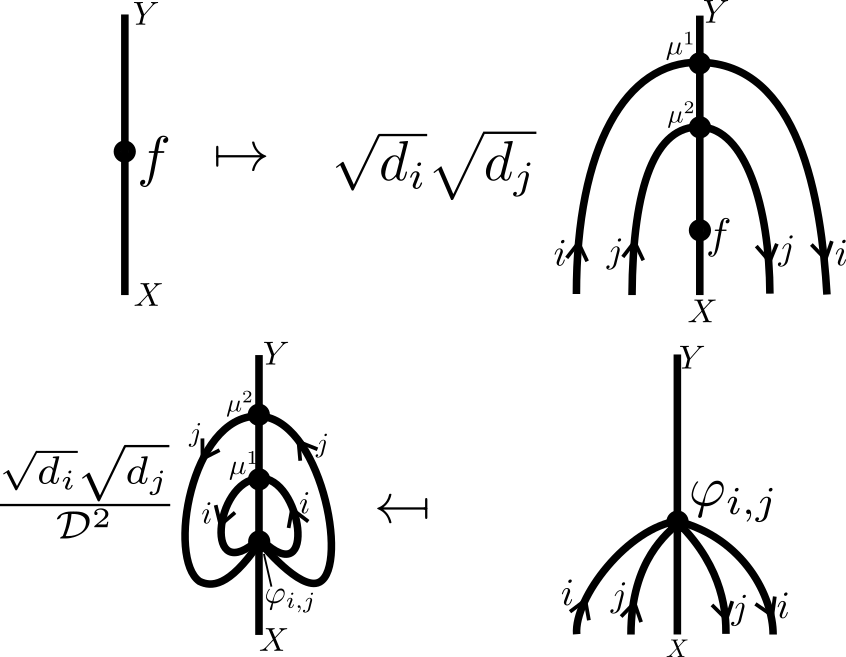}
  \hspace{10pt} & \hspace{10pt}
  \includegraphics[height=150pt]{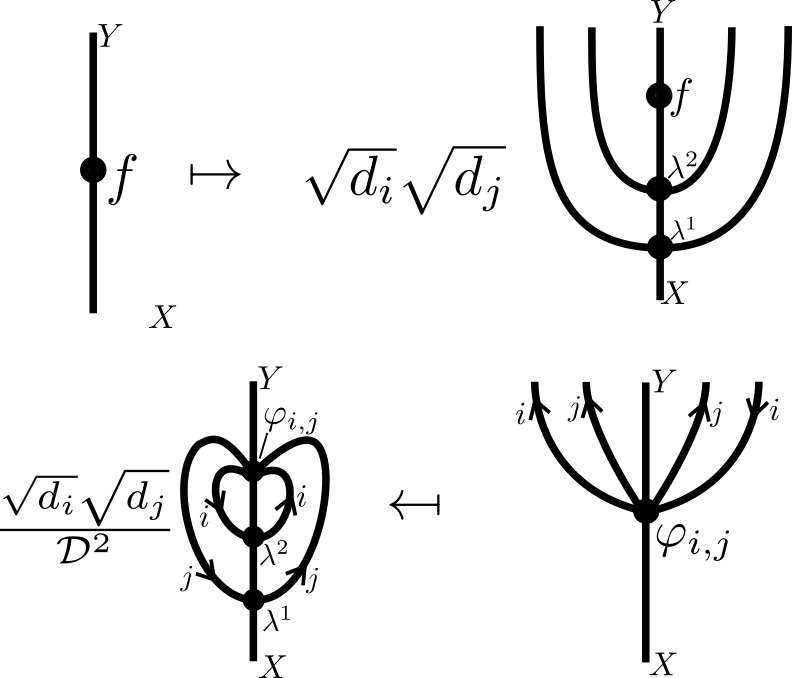}
\end{align*}

\end{proposition}

\begin{remark}
\label{rmk:Gamma_1}
Note that $\Gamma^1$ can also be described as 
the half-braiding given by induction
$\cI: \cC \to \ZC$,
that is,
$(\dirsum_{i,j} X_iX_jXX_j^*X_i^*, \Gamma^1)
\cong \cI(\dirsum_j X_jXX_j^*)$ in $\ZC$.
There is a similar description for $\Gamma^2$
that we will elaborate on later.
See \prpref{prp:intermediate_adjoint_1},
\rmkref{rmk:Gamma_2}.
\end{remark}

\begin{proof}
By \lemref{lem:Gamma_hf_brd}, $\Gamma^1,\Gamma^2$ are half-braidings.
To see that $\cII$ is a well-defined functor,
one easily checks that the two half-braidings $\Gamma^1,\Gamma^2$
satisfy COMM (but see \rmkref{rmk:reason_Gamma_COMM}),
and that on morphisms,
$\cII(f)$ intertwines both $\Gamma^1$ and $\Gamma^2$,
and $\cII$ respects identity and composition.
The maps between the $\Hom$ spaces are also easily checked
to be inverses.
\end{proof}

\begin{remark}
\label{rmk:reason_Gamma_COMM}
It is trivial to check that $\Gamma^1, \Gamma^2$ satisfy COMM,
but more illuminating is the ``reason" that they do.
$\cII X = (\dirsum_{i,j} X_iX_jXX_j^*X_i^*, \Gamma^1, \Gamma^2)$
is the prototype of an object in $\ZZC$.
Indeed, in the skein-theoretic approach to the extended Crane-Yetter TQFT
that we are studying,
such objects $\cII X$ arise naturally in the category
$Z_{\textrm{CY}}(\punctorus)$ as the configuration
consisting of one marked point in $\punctorus$
labelled with $X \in \cC$.
The other objects are obtained by taking all images of idempotents,
i.e. we consider the Karoubian closure.
In the current context, this takes the form of
\prpref{prp:elliptic_direct_summand}.
In other words,
the definition of compatible half-braidings was cooked up 
to capture the essential features of this prototypical object.
\hfill $\triangle$
\end{remark}

The following is the analog of \prpref{prp:direct_summand}:

\begin{proposition}
\label{prp:elliptic_direct_summand}
$(Y, \mu^1, \mu^2) \in \ZZC$ is a direct summand of $\cII Y$.
In particular, $\ZZC$ is the
Karoubian completion of the full image of $\cII$.
\end{proposition}

\begin{proof}
  The morphism below projects $\cII Y$ onto $(Y,\mu^1,\mu^2)$
(see proof of \prpref{prp:direct_summand}).
\[
  \includegraphics[height=80pt]{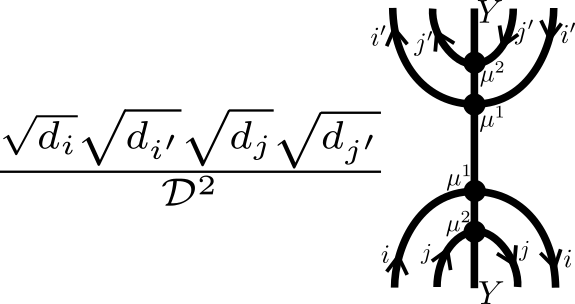}
\]
\end{proof}

So far we have only discussed the relation between $\ZZC$ and $\cC$.
As \rmkref{rmk:Gamma_1} suggests, $\cII$ actually factors through
$\ZC$. Indeed,
observe that $\cFF$ factors through an intermediate forgetful functor
$\cF_1 : \ZZC \to \ZC$, which forgets only the first braiding,
$(X,\lmb^1, \lmb^2) \mapsto (X,\lmb^2)$,
so that $\cFF = \cF \circ \cF_1$.

\begin{proposition}
\label{prp:intermediate_adjoint_1}
The intermediate forgetful functor $\cF_1: \ZZC \to \ZC$ 
which forgets the first braiding,
$\cF_1 : (X, \lmb^1, \lmb^2) \mapsto (X, \lmb^2)$,
has a two-sided adjoint $\cI_1: \ZC \to \ZZC$,
such that $\cII$ factors through it,
i.e. $\cII = \cI_1 \circ \cI$.\\

On objects, $\cI_1$ sends
\[
  (X,\lmb) \mapsto (X_iXX_i^*, \Gamma, \wdtld{\lmb})
\]
where $\Gamma$ was defined in \prpref{prp:adjoint} and
\[
  \includegraphics[height=40pt]{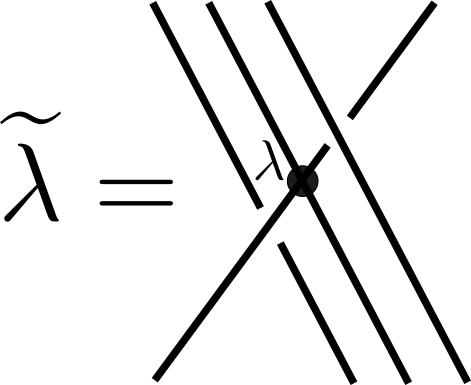}
\]
and on morphisms, $f\in \Hom_\ZC((X,\lmb),(Y,\mu))$,
\[
  \cI_1(f) = \id_{X_i}\tnsr f \tnsr \id_{X_i^*}
\]
which clearly intertwines $\wdtld{\lmb}$ and $\wdtld{\mu}$.\\

The adjunction
\begin{align*}
  \Hom_\ZC((X,\lmb), (Y,\mu^2)) \cong
    \Hom_\ZZC(\cI_1 (X,\lmb), (Y,\mu^1,\mu^2)) \\
  \Hom_\ZC((X,\lmb^2), (Y,\mu)) \cong
    \Hom_\ZZC((X,\lmb^1,\lmb^2), \cI_1 (Y,\mu))
\end{align*}
is given by the same maps as in
\prpref{prp:adjoint},
just with $\lmb$'s and $\mu$'s replaced by $\lmb^1$ and $\mu^1$'s,
respectively.

\end{proposition}

\begin{remark}
\label{rmk:Gamma_2}
In \rmkref{rmk:Gamma_1}, we noted that $\Gamma^1$
can be described as the half-braiding given by $\cI$.
Now, we can describe $\Gamma^2$, the second braiding
of $\cII X$, as $\wdtld{\Gamma}$,
where recall $\Gamma$ is the braiding 
$(X_iXX_i^*, \Gamma) = \cI X$.
Indeed, this had better be the case since
$\cII X = (\iinduct{X}, \Gamma^1, \Gamma^2)$
while $\cI_1(\cI X) = \cI_1(\induct{X}, \Gamma)
= (\iinduct{X}, \Gamma^1, \wdtld{\Gamma})$,
and the proposition claims they are the same.
\end{remark}

\begin{proof}
Similar to \prpref{prp:elliptic_adjoint}.
\end{proof}

The forgetful functor $\cF_1: \ZZC \to \ZC$ was defined to be
forgetting the first half-braiding. There is nothing special about
the first braiding compared to the second
(see \rmkref{rmk:first_second_hfbrd}).
We define $\cF_2: \ZZC \to \ZC$ to be the forgetful functor
forgetting the second half-braiding. One then also has
a two-sided adjoint functor $\cI_2$:\\

\begin{proposition}
\label{prp:intermediate_adjoint_2}
The intermediate forgetful functor $\cF_2: \ZZC \to \ZC$
which forgets the second braiding,
$\cF_2 : (X, \lmb^1, \lmb^2) \mapsto (X, \lmb^1)$,
has a two-sided adjoint $\cI_2: \ZC \to \ZZC$,
such that $\cII = \cI_2 \circ \cI$.\\

On objects, $\cI_2$ sends
\[
  (X,\lmb) \mapsto (X_iXX_i^*, \underset{\sim}{\lmb}, \Gamma)
\]
where
\[
  \includegraphics[height=40pt]{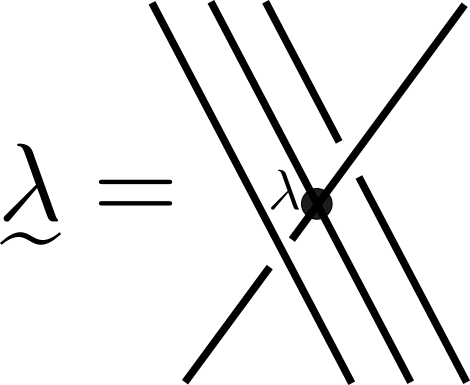}
\]
(note the different braidings used compared to $\wdtld{\lmb}$ in
\prpref{prp:intermediate_adjoint_1})\\
and on morphisms, $f\in \Hom_\ZC((X,\lmb),(Y,\mu))$,
\[
  \cI_2(f) = \id_{X_i}\tnsr f \tnsr \id_{X_i^*}
\]

The adjunction
\begin{align*}
  \Hom_\ZC((X,\lmb), (Y,\mu^1)) \cong
    \Hom_\ZZC(\cI_2 (X,\lmb), (Y,\mu^1,\mu^2)) \\
  \Hom_\ZC((X,\lmb^1), (Y,\mu)) \cong
    \Hom_\ZZC((X,\lmb^1,\lmb^2), \cI_2 (Y,\mu))
\end{align*}
is given by the same maps as in
\prpref{prp:adjoint},
just with $\lmb$'s and $\mu$'s replaced by $\lmb^2$ and $\mu^2$'s,
respectively.

\end{proposition}

\begin{proof}
Similar to \prpref{prp:elliptic_adjoint}.
\end{proof}

\begin{remark}
\label{rmk:sl2z_induct_forget}
In fact, there are infinitely many such pairs of induct-forget pairs,
indexed by elements of $\SLZ$. Indeed, in \secref{sec:sl2z},
we exhibit an $\SLZ$-action on $\ZZC$.
Then precomposing $\cF_1$ and postcomposing $\cI_1$
with the action of a group element
gives another forget-induction adjoint pair.
For example, $\cF_2 = \cF_1 \circ T_s$,
and $\cI_2 \simeq T_s^\inv \circ \cI_1$.
This has a simple explanation in terms of the Crane-Yetter TQFT.
Namely, each forgetful functor corresponds to one way to
cut the once-punctured torus along an embedded arc into an annulus
(the arc necessarily connects puncture to puncture),
while the corresponding
induction functor would be an inclusion of the annulus into
the once-punctured torus that avoids the cut.
This description is inspired by the description given in \ocite{K}
for the forget-induction adjunction between $\cC$ and $\ZC$
in terms of the string-nets interpretation of Turaev-Viro theory.
We will say more in forthcoming work.
\hfill $\triangle$
\end{remark}

\begin{remark}
\label{rmk:first_second_hfbrd}
As pointed out in \defref{def:ZZC},
the COMM relation is not commutative,
so it may appear that the first and second half-braidings
somehow have distinct characteristics.
However, this is merely a manifestation of the
dependence of the equivalence
$\ZZC \to \Zcy(\punctorus)$
on certain ``coordinate" choices (see Introduction).
The action of $s$ in the $\SLZ$-action discussed in
\secref{sec:sl2z} lends further credence to this,
swapping the first and second half-braidings on an object,
but to satisfy COMM, one of them is half-twisted.
\hfill $\triangle$
\end{remark}

In summary, we have the following (2-)commutative diagram:
\[
  \begin{tikzcd}[column sep=large,row sep=large]
    \cC \ar[r, "\cI", shift left=1] \ar[d, "\cI", shift left=1]
      \ar[rd, "\cII", shift left=1]
    & \ZC \ar[d, "\cI_1", shift left=1] \ar[l, "\cF", shift left=1]\\
    \ZC \ar[r, "\cI_2", shift left=1] \ar[u, "\cF", shift left=1]
    & \ZZC \ar[l, "\cF_2", shift left=1] \ar[u, "\cF_1", shift left=1]
      \ar[ul, "\cFF", shift left=1] 
  \end{tikzcd}
\]

\subsection{Tensor Product on $\ZZC$}
\label{sec:tensor_product}

Recall that for a monoidal category $\cC$,
the Drinfeld center $\ZC$ is a braided monoidal category,
with tensor product given by
\[
  (X,\lmb) \tnsr (Y,\mu) := (X\tnsr Y, \lmb \tnsr \mu)
\]
(see \prpref{prp:ZC_modular}).
One might naively try to define a tensor product on $\ZZC$ by
\[
  (X,\lmb^1,\lmb^2) \tnsr (Y,\mu^1,\mu^2)
  = (X\tnsr Y, \lmb^1\tnsr\mu^1, \lmb^2\tnsr\mu^2)
\]
Unfortunately, this doesn't work, because the two half-braidings
$\lmb^1\tnsr\mu^1$ and $\lmb^2\tnsr\mu^2$ do not satisfy COMM
(unless $\cC$ is symmetric, which we discuss in
\secref{sec:symmetric_C}).\\

All is not lost, however. A closer look at the tensor product
on the Drinfeld center reveals that it arises from
considering the pair of pants in Turaev-Viro.
There is a modified version of the pair of pants
for the once-punctured torus (see \rmkref{rmk:pop})
that leads to the following definition:

\begin{definition}
\label{def:reduced_tnsr}
Let $\lmb,\mu$ be half-braidings on $X,Y$, respectively.
Define $X \tnsrproj{\lmb}{\mu} Y$ to be the image of the projection
$Q_{\lmb,\mu} \lcirclearrowright X \tnsr Y$, where
\[
  \includegraphics[height=60pt]{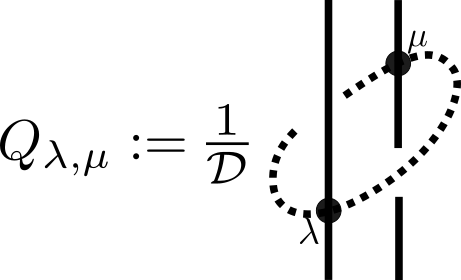}
\]
We call $X \tnsrproj{\lmb}{\mu} Y$ the \emph{reduced tensor product
of $X$ and $Y$ (with respect to $\lmb$ and $\mu$)}.
\hfill $\triangle$
\end{definition}

\begin{remark}
\label{rmk:image_choice}
Strictly speaking, in defining $X\tnsrproj{\lmb}{\mu} Y$
as the image of some projection,
we have to make a choice of some object along
with certain maps relating it to $X\tnsr Y$.
We make such a choice arbitrarily for each $X,Y$,
but identify morphisms out of $X\tnsrproj{\lmb}{\mu} Y$
with morphisms out of $X\tnsr Y$ that are invariant under
precomposing with $Q_{\lmb,\mu}$,
and similarly for morphisms into $X\tnsrproj{\lmb}{\mu} Y$:
\begin{align*}
  \Hom_\cC(X\tnsrproj{\lmb}{\mu} Y, N)
  &\cong \{g\in \Hom_\cC(X\tnsr Y, N)|g = g\circ Q_{\lmb,\mu}\}
  = \Hom_\cC(X\tnsr Y, N) \circ Q_{\lmb,\mu} \\
  \Hom_\cC(M, X\tnsrproj{\lmb}{\mu} Y)
  &\cong \{f\in \Hom_\cC(M, X\tnsr Y)|f = Q_{\lmb,\mu} \circ f\}
  = Q_{\lmb,\mu} \circ \Hom_\cC(M, X\tnsr Y)
\end{align*}
For brevity, when describing a morphism
$g: X\tnsrproj{\lmb}{\mu} Y \to N$,
we will sometimes omit $Q_{\lmb,\mu}$,
or more accurately we implicitly precompose $g$
with $Q_{\lmb,\mu}$.
Likewise for morphisms $f: M \to X\tnsrproj{\lmb}{\mu} Y$.
\hfill $\triangle$
\end{remark}

\begin{lemma}
\label{lem:tnsrproj_hfbrd}
Let $\lmb,\mu$ be half-braidings on $X,Y$, respectively.
For $A\in \cC$, we have the following equality
of morphisms
$A \tnsr (X\tnsrproj{\lmb}{\mu} Y)
\to (X\tnsrproj{\lmb}{\mu} Y) \tnsr A$:
\[
  \includegraphics[height=60pt]{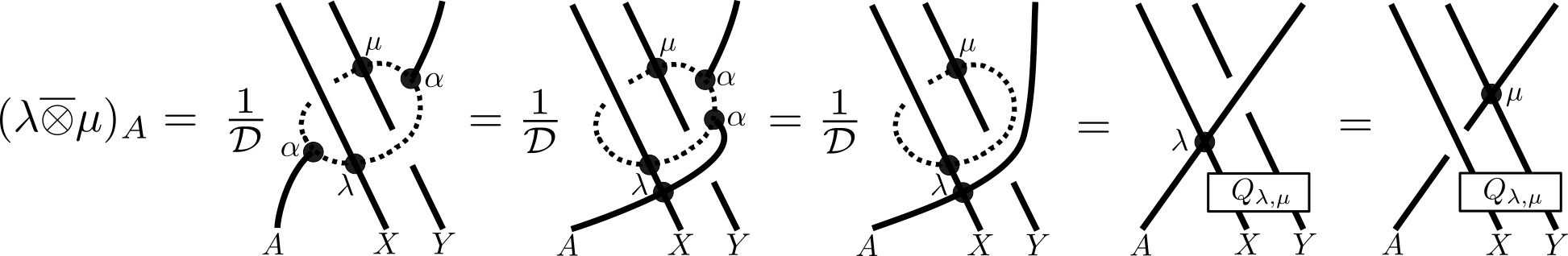}
\]
So defined, $\lmb \tnsrbar \mu$ is a half-braiding
on $X \tnsrproj{\lmb}{\mu} Y$.
\end{lemma}

\begin{proof}
The first and third equality
uses the fact that $\lmb$ is a half-braiding.
The second equality uses \lemref{lem:identity_combine}.
This also easily implies the fact that $\lmb\tnsrbar \mu$ is a half-braiding.
The last equality follows from similar sequence of equalities.\\

The third equality also easily implies that $\lmb\tnsrbar\mu$
is a half-braiding.
\end{proof}

\begin{corollary}
\label{cor:reduced_tnsr_triple}
Let $\lmb,\mu,\zeta$ be half-braidings on $X,Y,Z\in \cC$, respectively.
  Then
\begin{align*}
  (X \tnsrproj{\lmb}{\mu} Y) \tnsrproj{\lmb\tnsrbar\mu}{\zeta} Z
  = \textrm{im}(Q_{\lmb,\mu,\zeta} \lcirclearrowright (X\tnsr Y)\tnsr Z) \\
  X \tnsrproj{\lmb}{\mu\tnsrbar\zeta} (Y \tnsrproj{\mu}{\zeta} Z)
  = \textrm{im}  (Q_{\lmb,\mu,\zeta} \lcirclearrowright X\tnsr (Y\tnsr Z)) \\
\end{align*}
where
$Q_{\lmb,\mu,\zeta} =
 (\id_X \tnsr Q_{\mu,\zeta}) \circ (Q_{\lmb,\mu} \tnsr \id_Z)$.\\

More generally, if we have half-braidings $\lmb_1,\ldots \lmb_k$
on $X_1,\ldots, X_k$, respectively,
then $(\ldots(X_1 \tnsrproj{\lmb^1}{\lmb^2} X_2)\ldots)\tnsrproj{\ldots}{\lmb^k} X_k$
is the image of the projection
\[
  \includegraphics[height=80pt]{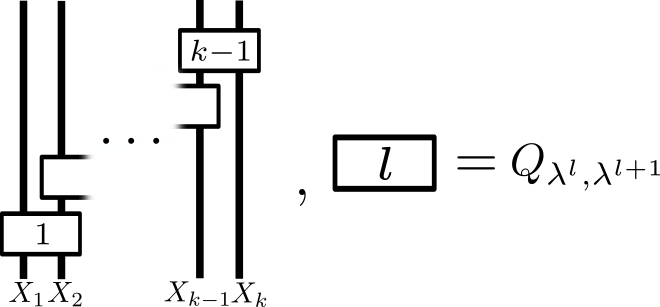}
\]
and similarly for any order of taking reduced tensor products.
\end{corollary}

\begin{defn-prop}
\label{def:ZZC_tnsr}
$\ZZC$ admits the following monoidal structure:\\

For objects $(X,\lmb^1,\lmb^2),(Y,\mu^1,\mu^2) \in \ZZC$,
their tensor product is defined by
\[
  (X,\lmb^1,\lmb^2) \tnsr (Y,\mu^1,\mu^2)
  = (X \tnsrproj{\lmb^1}{\mu^1} Y, \lmb^1 \tnsrbar \mu^1, \lmb^2 \tnsr \mu^2)
\]
where $X \tnsrproj{\lmb^1}{\mu^1} Y$ was defined in \defref{def:reduced_tnsr},
and $\lmb^1 \tnsrbar \mu^1$ was defined in \lemref{lem:tnsrproj_hfbrd}.\\

The associativity constraint $a$ is given by the associativity constraint
of $\cC$ (understood as in \rmkref{rmk:image_choice}).\\

The unit object $\oneel$ is
$\cI_1 (\one, \id_-) = (\induct{},\Gamma,\Omega)$,
where $\Omega = \wdtld{\id_-}$,
with left and right unit constraint given by
\[
  \includegraphics[height=60pt]{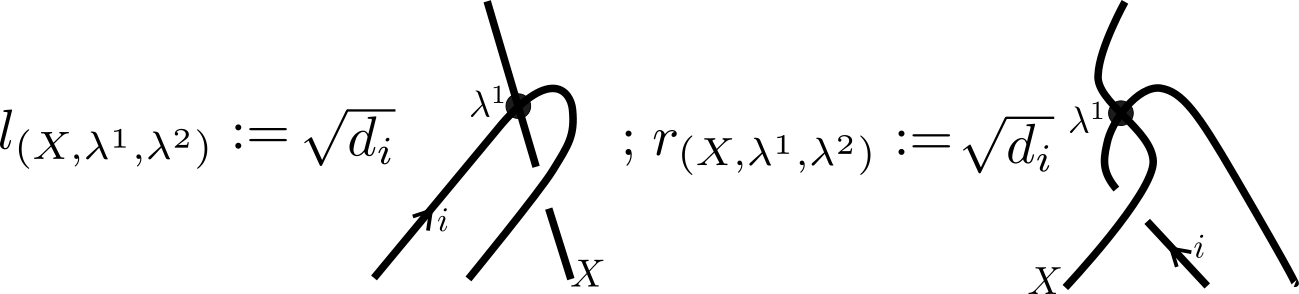}
\]
\hfill $\triangle$
\end{defn-prop}

\begin{proof}
Let $\cX = (X,\lmb^1,\lmb^2),
      \cY = (Y,\mu^1,\mu^2), \cZ = (Z,\zeta^1,\zeta^2)$.
It is easy to see that $\lmb^1\tnsrbar\mu^1$ and $\lmb^2\tnsr\mu^2$
satisfy COMM, so the definition of $\cX \tnsr \cY$ makes sense.\\

By \corref{cor:reduced_tnsr_triple},
the associativity constraint $a_{\cX,\cY,\cZ}$
is just $Q_{\lmb^1,\mu^1,\zeta^1}:
  (X\tnsr Y)\tnsr Z \to X\tnsr (Y\tnsr Z)$,
so is an isomorphism
$(X \tnsrproj{\lmb^1}{\mu^1} Y) \tnsrproj{\lmb^1\tnsrbar\mu^1}{\zeta^1} Z
  \to X \tnsrproj{\lmb^1}{\mu^1\tnsrbar\zeta^1} (Y \tnsrproj{\mu^1}{\zeta^1} Z)$.
The pentagon axiom also follows from \corref{cor:reduced_tnsr_triple}.\\

To see that $l_\cX,r_\cX$ are isomorphisms,
we provide morphisms
\[
  \includegraphics[height=60pt]{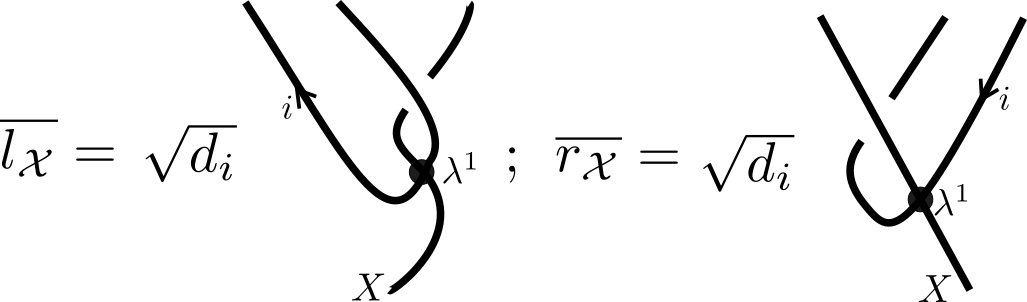}
\]
and check that they are inverses: $l_\cX \circ \overline{l_\cX} = \id_\cX,
  \overline{l_\cX} \circ l_\cX = Q_{\Gamma,\lmb^1}$,
and likewise for $r_\cX$.
For example,

\[
  \includegraphics[height=60pt]{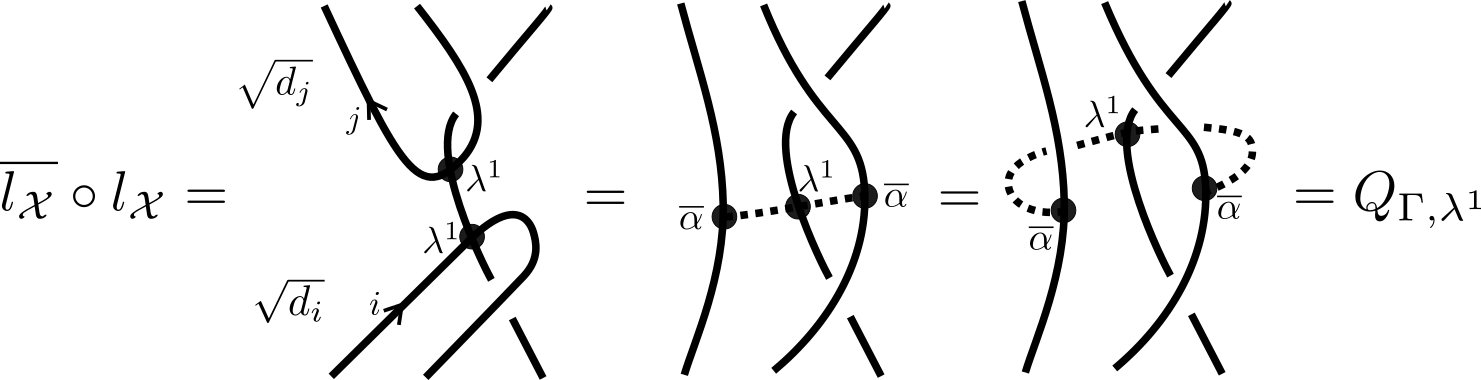}
\]

The triangle axiom for unit constraints is a simple consequence of
\lemref{lem:tnsrproj_hfbrd} (in particular the last equality).
Let us note that the $\Omega$ in $\oneel$
is the same $\Omega$ as in the proof of \lemref{lem:fully_faithful}
but specifically for $X = \one$,
so that $\oneel = i\one$.
\end{proof}

\begin{proposition}
\label{prp:I1_tensor}
For $(X,\lmb),(Y,\mu)\in \ZC$, let
$J_{(X,\lmb),(Y,\mu)} = J_{X,Y} : \cI_1(X,\lmb) \tnsr \cI_1(Y,\mu) \to \cI_1((X,\lmb)\tnsr(Y,\mu))$
be defined by
\[
  \includegraphics[height=60pt]{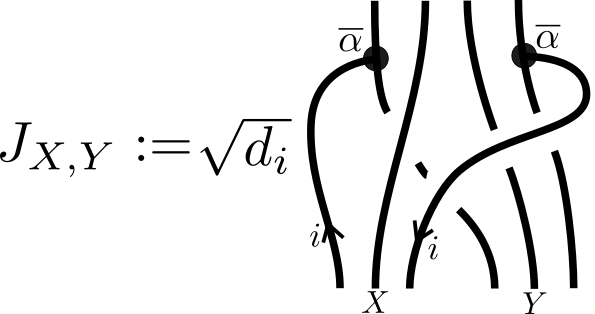}
\]
Then $(\cI_1,J): \ZC \to \ZZC$ is a monoidal functor.
\end{proposition}

\begin{proof}
By construction, $\cI_1$ sends unit to unit.
We check that $J_{X,Y}$ has an inverse, $\overline{J_{X,Y}}$, as follows:
\[
  \includegraphics[height=140pt]{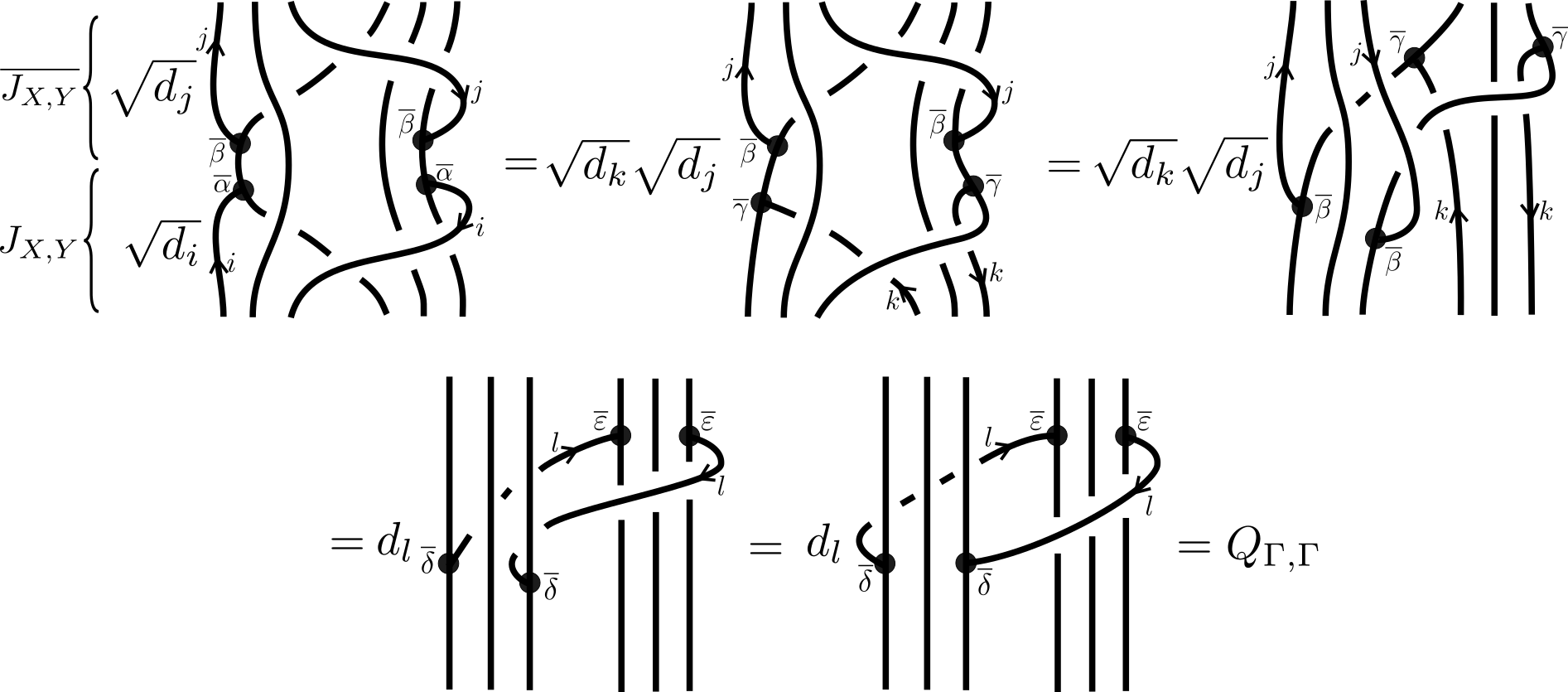}
\]
(using \lemref{lem:Gamma_switch} repeatedly),
and similarly check that $J_{X,Y}\circ \overline{J_{X,Y}} = \id$.
That $J$ satisfies the hexagon axiom follows from
\begin{align*}
  J_{X\tnsr Y,Z} \circ (J_{X,Y} \tnsr \id_Z) =
  \picmid{\includegraphics[height=80pt]{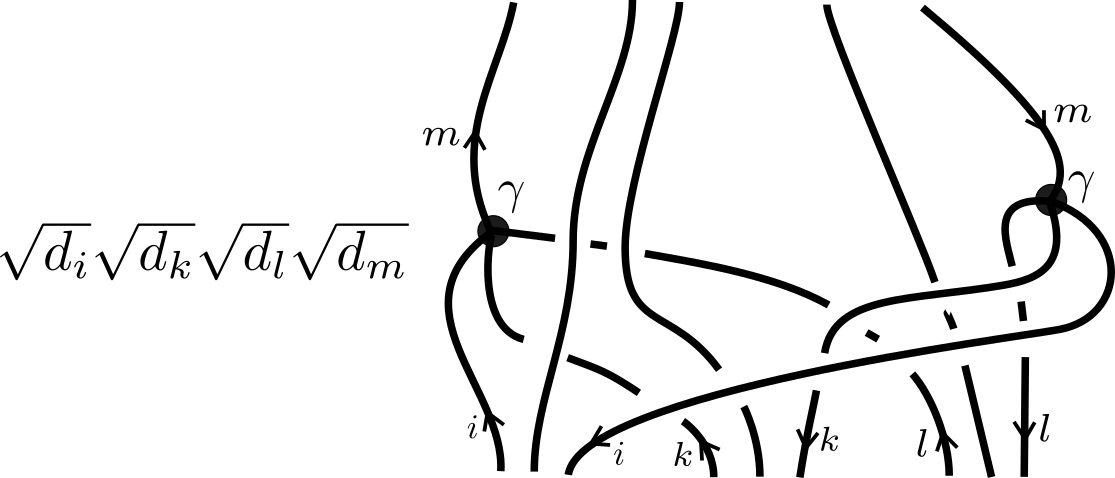}}
  = J_{X,Y\tnsr Z} \circ (\id_X \tnsr J_{Y,Z})
\end{align*}
\end{proof}

We work out in \xmpref{xmp:group_algebra_elliptic_tnsr}
what this tensor product looks like when $\cC = \Rep(G)$,
the category of finite-dimensional representations of a finite group $G$.\\

\begin{remark}
\label{rmk:multitensor}
The presence of the projection in \defref{def:reduced_tnsr} leads
this tensor product to be highly non-commutative,
and also makes $\ZZC$ multi-tensor but generally not tensor.
\hfill $\triangle$\\
\end{remark}

\begin{remark}
\label{rmk:pop}
\defref{def:reduced_tnsr} of the tensor product
looks rather ad hoc, but it actually arises from
understanding the functor that $\Zcy$ associates to the following
cobordism $Y$ from $\punctorus \sqcup \punctorus$ to $\punctorus$:
\[
  \punctorus :
  \picmid{\includegraphics[height=60pt]{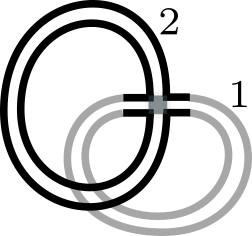}}
  \hspace{20pt}
  Y :
  \picmid{\includegraphics[height=100pt]{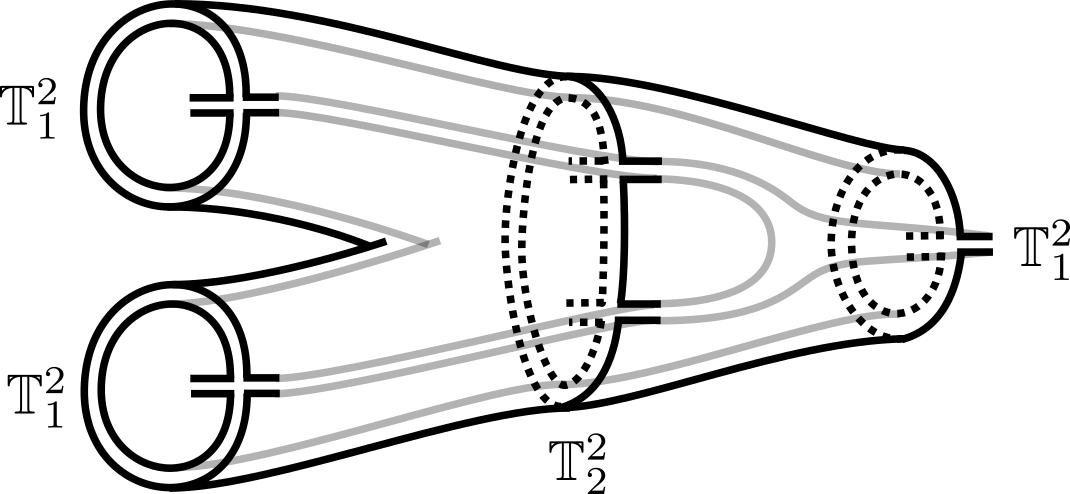}}
\]
In the diagram, $\punctorus$ is drawn as
the ``plumbing" of two annuli,
that is, it is the union of two annuli,
the gray 1-annulus (annulus labelled with a 1)
and the black 2-annulus,
glued along the common gray square in a transverse manner
(the 1-annulus and 2-annulus
do not meet away from the gray square in the middle).
We use the names 1- and 2-annuli suggestively,
to indicate that they are related to the first
and second half-braidings respectively.
While the picture has a boundary,
we should imagine it to be at infinity,
so that we are actually plumbing two open annuli
to get a once-punctured torus
(and not a torus with one boundary component).\\

In the picture for $Y$,
we omit most of the 1-annulus,
just drawing the bit that meets the 2-annulus.
The first (left) half of the cobordism
goes from $\punctorus \sqcup \punctorus$ to $\mathbb{T}_2^2$,
the twice punctured torus,
using a thickened pair of pants on the 2-annuli,
leaving the 1-annuli untouched.
The second (right) half of the cobordism
goes from $\mathbb{T}_2^2$ to $\punctorus$,
leaves the 2-annulus untouched,
and ``stacks" the two first annuli
(see next picture and explanation of $Y_1$).\\

We can also picture $Y$ as the union of the two cobordisms
$Y_2$ and $Y_1$, shown below,
each of which is a cobordism from two annuli to one annulus.
$Y_2$ and $Y_1$ are identified along
the dark gray portions \textbf{Y},
which is a three-dimensional thickened `Y',
thought of as a cobordism from two disks to one disk.
\[
  Y = Y_2 \displaystyle\bigcup_{\textrm{\textbf{Y}}} Y_1
  = \picmid{\includegraphics[width=0.75\linewidth]{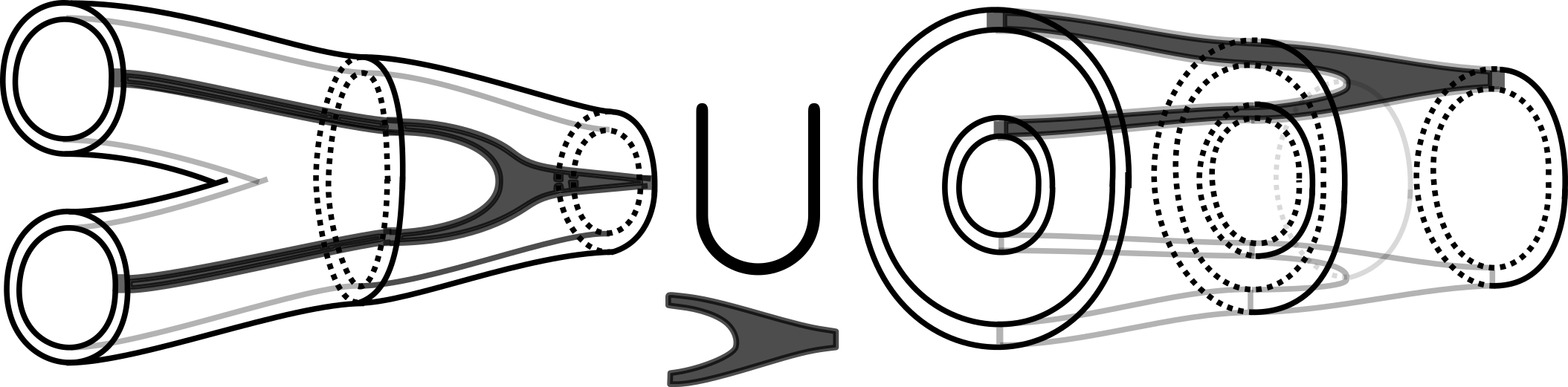}}
\]
The cobordism on the right, $Y_1$,
governs the behaviour of the 1-annuli.
It is obtained by taking 
a two-dimensional thickened `Y' (like the light gray part)
and crossing with $S^1$.
(In the first picture of $Y$,
only a small neighbourhood of \textbf{Y}
in $Y_1$ appears.)
We call this ``stacking"
because the cobordism takes in two annuli,
each thought of as a cylinder going from one circle boundary to the other,
and glues the outgoing boundary of one cylinder to the incoming boundary
of the other.
The cobordism on the left, $Y_2$,
is just a thickened pair of pants,
and governs the behaviour of the 2-annuli.\\

Recall that the tensor product of \defref{def:ZZC_tnsr}
is the usual $\lmb^2 \tnsr \mu^2$ for the second half-braiding,
and the ``reduced tensor product" $\lmb^1 \tnsrbar \mu^1$
for the first half-braiding.
This is reflected in (or more accurately, a consequence of)
the cobordism $Y$:
the 2-annuli are governed by the usual pair of pants $Y_2$,
while the 1-annuli are governed by this ``stacking" cobordism $Y_1$.\\

The cobordism $Y:\punctorus \sqcup \punctorus \to \punctorus$
is not ``commutative" in the way that the usual pair of pants is
- that is, there is no homeomorphism from the cobordism
to itself that flips the two input boundary components.
However, it is still associative,
so leads to a tensor product structure.\\

In the Introduction,
we mention that the category we associate to the annulus
is the Drinfeld center.
The cobordism $Y_2$ leads to the usual tensor product,
while the cobordism $Y_1$ leads to a different
(multitensor) monoidal structure on $\ZC$.
This has been discussed in \ocite{BBJ2},
though not in terms of half-braidings.
This will be the subject of an upcoming paper.
\hfill $\triangle$
\end{remark}

\subsection{$\SLZ$-action on $\ZZC$}
\label{sec:sl2z}

As briefly discussed in the introduction,
$\ZZC$ is the category we associate to a once-punctured torus
in an extended Crane-Yetter TQFT,
so we expect properties of the once-punctured torus to manifest in $\ZZC$.
For example, in \rmkref{rmk:sl2z_induct_forget},
we reveal that the intermediate induction functors $\cI_1,\cI_2$
has a topological interpretation.\\

Here, we will discuss an $\SLZ$-action on $\ZZC$,
since $\SLZ$ is the mapping class group of $\punctorus$.
We define the action by generators and relations,
but do not prove that they are coherent in the appropriate sense.
This is because our proof uses the topology of surfaces
in a crucial way,
so we postpone full proofs to future work
in the context of the extended Crane-Yetter TQFT.\\

Recall that $\SLZ$ is generated by two matrices
\[
  s = 
  \begin{pmatrix}
    0 & -1\\
    1 & 0
  \end{pmatrix}
  \textrm{ and }
  t =
  \begin{pmatrix}
    1 & 1\\
    0 & 1
  \end{pmatrix}
\]
subject to the relations

\begin{align*}
  r_1 &: s^4 = 1 \\
  r_2 &: sts = t^\inv s t^\inv
\end{align*}
so that $\SLZ = \eval{s,t | r_1,r_2}$.\\

Since $\SLZ$ is finitely presented,
we would like to write an action of $\SLZ$ on $\ZZC$
with a finite amount of data,
i.e. describe the action by ``generators and relations".
The general framework for this is spelled out
in the Appendix \secref{sec:group_action},
so here we just provide the relevant data,
and discuss concretely what needs to be proved.\\

To begin, we need to say how $s,t$ act on $\ZZC$.
We define auto-equivalences
$U_s, U_t \in \Aut(\ZZC)$
on objects as follows:
\begin{align*}
  U_s : (X, \lmb^1, \lmb^2)
          \mapsto (X, (\lmb^2)^\dagger, \lmb^1)\\
  U_t : (X, \lmb^1, \lmb^2)
          \mapsto (X, \lmb^1, \lmb^2 \ltimes \lmb^1)
\end{align*}
where
\[
  \includegraphics[height=60pt]{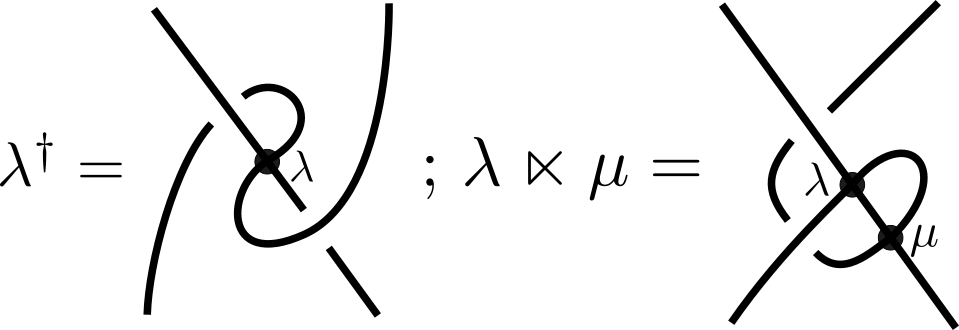}
\]
(Intuitively, $\lmb^\dagger$ is obtained from
$\lmb$ by twisting the $X$ ribbon strand by a half-twist,
and dragging the cross-strand along.)
On morphisms, they will act ``trivially" in the following sense:
it is easy to check that
\[
  \Hom_\ZZC((X,\lmb^1,\lmb^2), (Y,\mu^1,\mu^2))
  = \Hom_\ZZC(U_s(X,\lmb^1,\lmb^2), U_s(Y,\mu^1,\mu^2))
  \subseteq \Hom_\cC(X,Y)
\]
so we specify that $U_s$ acts as identity
on this vector space
(likewise for $U_t$).
$U_s,U_t$ are in fact isomorphisms,
so they have canonical inverses $U_s^\inv, U_t^\inv$.\\

Next, we need to relate the actions of two words
which are equal in $\SLZ$.
To this end, for each of the relations $r_1,r_2$,
we specify a natural isomorphism:
\begin{align*}
  \gamma_1 &: U_s^4 \xrightarrow[]{\sim} \id_\ZZC \\
  \gamma_2 &: U_s U_t U_s \xrightarrow[]{\sim} U_t^\inv U_s U_t^\inv
\end{align*}
These $\gamma$'s are given by
\begin{align*}
  (\gamma_1)_{(X,\lmb^1,\lmb^2)} &= \theta_X
    : U_s^4 (X,\lmb^1,\lmb^2)
      = (X,(\lmb^1)^{\dagger\dagger},(\lmb^2)^{\dagger\dagger})
      \to (X,\lmb^1,\lmb^2)\\
  (\gamma_2)_{(X,\lmb^1,\lmb^2)} &=
    \begin{matrix}
    \includegraphics[height=40pt]{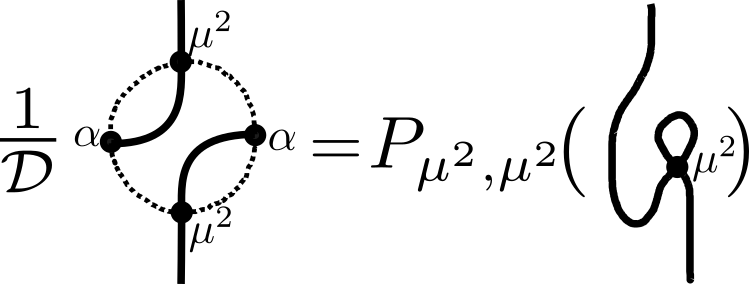}
    \end{matrix}
\end{align*}
where $\mu^2 = (\lmb^2)^\dagger$ is the second half-braiding of
$U_s U_t U_s(X,\lmb^1,\lmb^2)$,
and $\theta$ is the twist or balancing structure on $\cC$.
(Intuitively, $\theta$ is a full twist that turns
$(\lmb^1)^{\dagger\dagger}$ back to $\lmb^1$.)\\

\begin{theorem}
\label{thm:slz_action}
$U_s,U_t,\gamma_1,\gamma_2$
generate an $\SLZ$-action on $\ZZC$.
\end{theorem}

We will prove this in future work,
but let us briefly discuss how one generates the action,
and what is entailed in the proof of the theorem.\\

To a word $w$ in $s,t$ (and inverses $s^\inv, t^\inv$,
we associate the appropriate composition of $U$'s,
which we denote $U_w$.
Each group element $g\in \SLZ$ has many word representatives,
so the associated $U_w$'s of various word representatives
should be in some sense equal.
Being functors, they should not be expected to be equal on
the nose, but related by a natural isomorphism
- these are provided by applying $\gamma_1$ and $\gamma_2$.\\

Now the natural isomorphisms themselves should
be compatible with each other as follows.
Consider two ways to get from $s^8$ to 1:
\[
  s^8 = s^2 \cdot s^4 \cdot s^2
    \xrightarrow[]{s^2 \cdot r_1 \cdot s^2} s^2 \cdot 1 \cdot s^2 = s^4
    \xrightarrow[]{r_1} 1
\]
and
\[
  s^8 = s^4 \cdot s^4
    \xrightarrow[]{s^4 \cdot r_1} s^4 \cdot 1 = s^4
    \xrightarrow[]{r_1} 1
\]
The data in \thmref{thm:slz_action}
gives us two natural isomorphisms $U_s^4 \to \id_\ZZC$:
\[
\begin{tikzcd}
  U_s^8 \ar[r, "U_s^4 \gamma_1"] \ar[d, "U_s^2 \gamma_1 U_s^2"']
    & U_s^4 \ar[d, "\gamma_1"] \\
  U_s^4 \ar[r, "\gamma_1"]
    & \id
\end{tikzcd}
\]
and in order for the group action to be well-defined,
this diagram must commute.\\

In general, for any pair of words $w_1,w_2$,
any path of applying relations from $w_1$ to $w_2$
gives rise to a natural isomorphism,
and all paths must result in the same natural isomorphism.
In the Appendix \secref{sec:group_action},
we show that when this is satisfied,
one obtains an $\SLZ$-action on $\ZZC$
(see \prpref{prp:grp_action_genrel_functor},
\corref{cor:grp_action_genrel_usual}).\\

Once again, our proof of \thmref{thm:slz_action},
that is, checking the above condition,
uses the topology of surfaces in an essential way,
hence we find it best to present the proof in the context
of the extended Crane-Yetter TQFT,
which will appear in future work.\\

\begin{remark}
The action described above does not respect the
monoidal structure defined in \secref{sec:tensor_product}.
Instead, it intertwines the many tensor product structures
one can put on $\ZZC$,
each coming from choosing how to identify $\punctorus$
as the plumbing of two annuli (see \rmkref{rmk:pop}).
\hfill $\triangle$
\end{remark}

\begin{remark}
Later in \thmref{thm:modular}, we prove that when $\cC$ is modular,
$\ZZC \simeq \cC$, so we should have an $\SLZ$-action on $\cC$ as well.
It is not yet clear to us how to write this action on $\cC$ explicitly
without reference to $\ZZC$.
\hfill $\triangle$
\end{remark}

\section{Modular Case}
\label{sec:modular_case}

In this section, we prove \thmref{thm:modular},
which says that when $\cC$ is modular,
$\ZZC \cong \cC$.
As discussed in the introduction,
this is a natural result to expect from
the supposition that the Crane-Yetter TQFT
is a boundary theory.\\

We will need two lemmas:

\begin{lemma}
\label{lem:fully_faithful}
Let $i$ be the composition
\[
    i = \cI_1 \circ \iota : \cC \to \ZZC
\]
where $\iota: \cC \to \ZC$ is the functor $X \mapsto (X,c_{-,X})$,
and $\cI_1$ is the intermediate induction functor
defined in \prpref{prp:intermediate_adjoint_1}.
Then $i$ is fully faithful.
\end{lemma}

\begin{proof}
The functor $i$ sends $X \mapsto (\induct{X}, \Gamma, \Omega)$,
where $\Omega = \wdtld{c_{-,X}}$
(the $\Omega$ in Definition-Proposition \ref{def:ZZC_tnsr}
is for $X = \one$),
and on morphisms, $i$ sends $f:X \to Y$ to
$i(f) = \sum_j \id_{X_j} \tnsr f \tnsr \id_{X_j^*}$.\\

Faithfulness follows without modularity of $\cC$:
indeed, the way $i$ is defined means the following diagram commutes:\\
\[
\begin{tikzcd}
  \Hom_\ZZC (iX, iY) \ar[r, hookrightarrow] & \Hom_\cC (\induct{X}, \induct{Y})\\
  \Hom_\cC (X,Y) \ar[u, "i"] \ar[ur, hookrightarrow]
\end{tikzcd}
\]
so $i$ must be an injection.\\

To show that $i$ is full,
let us study $\Hom_\ZZC (iX, iY)$ as the intersection
$\Hom_\ZC ((\induct{X}, \Gamma), (\induct{Y}, \Gamma))
\cap \Hom_\ZC((\induct{X}, \Omega), (\induct{Y}, \Omega))$.
The first space is just $\Hom_\ZC(\cI X, \cI Y)$,
so it is of the form seen in \corref{cor:hom_IX_IY}.\\

Further intersecting with the second space,
or equivalently finding the image of the first space
under the projection $P_{\Omega,\Omega}$,
\[
  \includegraphics[width=0.9\textwidth]{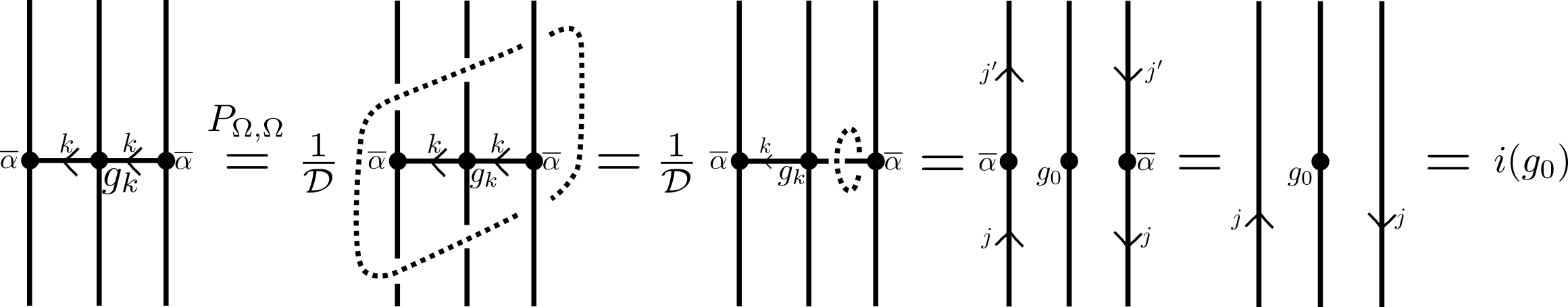}
\]
so all morphisms in $\Hom_\ZZC(iX,iY)$ are indeed of
the form $i(f)$, $f\in \Hom_\cC(X,Y)$,
hence $i$ is full.
This computation depends on the modularity of $\cC$,
specifically in the third equality,
where we use \lemref{lem:charge_conservation}.
Note that it was crucial that in the first step we deduced from
\corref{cor:hom_IX_IY} that the ``struts"
between the middle strand and the outer strands are labelled by
the same thing, so that killing one also kills the other.\\

\end{proof}

\begin{lemma}
\label{lem:I1_kills_braiding}
If $\cC$ is modular,
then for any object $(Y,\mu) \in \ZC$,
\[
  \cI_1(Y,\mu) \cong iY
\]
\end{lemma}

\noindent
Observe that the lemma implies
$\cI_1(Y,\mu) \cong iY = \cI_1(Y, c_{-,Y})$,
so this lemma is saying that when $\cC$ is modular,
$\cI_1$ kills the difference in braiding in $\ZC$.\\

\begin{proof}
By \prpref{prp:modular_factorizable},
we have
$\ZC \simeq_{\tnsr, \text{br}} \cC \boxtimes \cC^\bop$,
so we may write
\[
  (Y,\mu) \overset{\vphi'}{\cong} \dirsum_k (A_k,c_{-,A_k}) \tnsr (B_k,c_{B_k,-}^\inv)
  = (\dirsum_k A_k B_k, \sigma)
\]
where
\[
  \includegraphics[height=2cm]{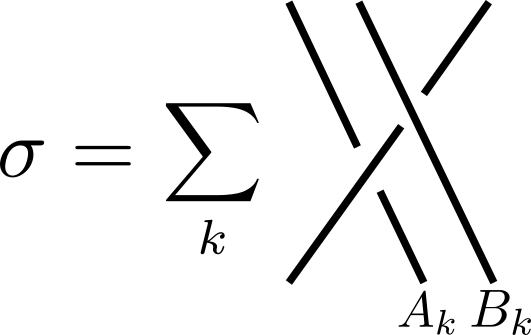}
\]

Let $\vphi = \sum_k \vphi_k: Y \to \dirsum_k B_k A_k$
be an isomorphism in $\cC$
(say $\vphi = c \circ \vphi'$),
and let $\psi = \sum_k \psi_k :\dirsum_k B_k A_k \to Y$
be its inverse.
Consider the morphisms
$\wdtld{\vphi} : \induct{Y} \to \induct{(\dirsum_k A_k B_k)}$
and
$\wdtld{\psi} : \induct{(\dirsum_k A_k B_k)} \to \induct{Y}$
described below:
\[
  \includegraphics[height=90pt]{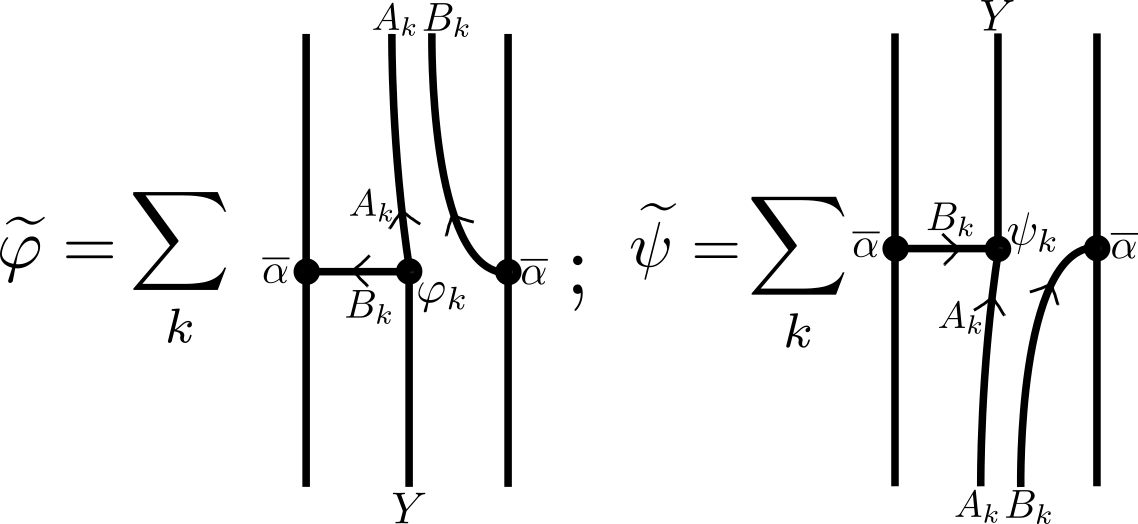}
\]

These are in fact morphisms in $\ZZC$,
that is,
\begin{align*}
  \wdtld{\vphi} &\in \Hom_\ZZC(iY, \cI_1(\dirsum_k A_k B_k, \sigma))\\
  \wdtld{\psi} &\in \Hom_\ZZC(\cI_1(\dirsum_k A_k B_k, \sigma), iY)
\end{align*}

For the first braiding, it follows from properties of $\overline{\alpha}$,
(see proof of \lemref{lem:Gamma_hf_brd}),
while for the second braiding, it is apparent from the diagrams.\\

Finally, we see that $\wdtld{\vphi}$ and $\wdtld{\psi}$ are inverses:
\[
  \includegraphics[height=90pt]{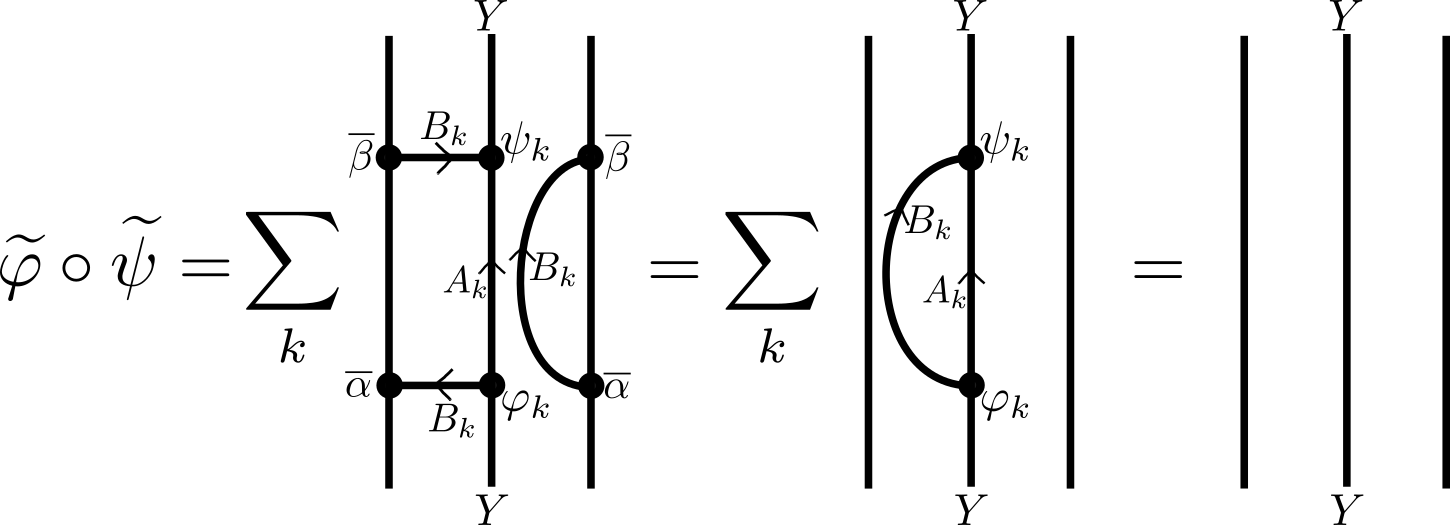}
\]
(the first equality follows from the fact that $\Gamma$ is a half-braiding
  - we just swapped the relative positions of the ``walls"),
and the other way is similar.\\

Thus $\cI_1(Y,\mu) = \cI_1(\dirsum_k A_k B_k, \sigma)
\overset{\wdtld{\psi}}{\cong} iY$.

\end{proof}

Now we can prove:

\begin{theorem}
\label{thm:modular}
  If $\cC$ is modular, then the composition
  \[
    i = \cI_1 \circ \iota : \cC \to \ZZC
  \]
  is an equivalence of abelian categories,
  where $\iota: \cC \to \ZC$ is the functor $X \mapsto (X,c_{-,X})$,
  and $\cI_1$ is the intermediate induction functor
  defined in \prpref{prp:intermediate_adjoint_1}.
\end{theorem}

\begin{proof}
Recall that
$i$ sends $X \mapsto (\induct{X}, \Gamma, \Omega)$,
where $\Omega = \wdtld{c_{-,X}}$,
and on morphisms, $i$ sends $f:X \to Y$ to
$i(f) = \sum_j \id_{X_j} \tnsr f \tnsr \id_{X_j^*}$.\\

By \lemref{lem:fully_faithful}, $i$ is fully faithful.
It remains to show that $i$ is essentially surjective.
Observe that \lemref{lem:I1_kills_braiding} implies that
any object of the form $\cII X$ is isomorphic to some object
$iX'$, $X'\in \cC$ -
just take $(Y,\mu)$ in \lemref{lem:I1_kills_braiding}
to be $(\induct{X}, \Gamma) = \cI X$
(see \prpref{prp:adjoint}), so that
\[
  \cII X = \cI_1(\cI X) = \cI_1(Y,\mu) \cong iY
\]

By \prpref{prp:elliptic_direct_summand},
any object $A$ of $\ZZC$ is a
direct summand of some $\cII X$,
hence by the above observation,
$A$ is a direct summand of some $iX' \cong \cII X$.
But by \lemref{lem:fully_faithful}, $i$ is fully faithful, so if $Q_A$ is a projection
in $\End_\ZZC(iX)$ such that $\im(Q_A) \cong A$, then $Q_A = i(q_A)$ for some
projection $q_A \in \End_\cC(X)$, so $i(\im(q_A)) \cong \im(Q_A) \cong A$.
Hence, $i$ is essentially surjective,
and we are done.
\end{proof}

The last paragraph in the proof above can be phrased
in more abstract but conceptually clearer terms,
encapsulated in the following diagram:
\[
\begin{tikzcd}
  \mathcal{E} \ar[r, hookrightarrow, "\textrm{f.f.}"] \ar[d, hookrightarrow]
    & \mathcal{D} \ar[r, hookrightarrow,  "\textrm{f.f.}"] \ar[d, equal]
    & \ZZC\\
  \overline{\mathcal{E}}^\oplus \ar[r, "\textrm{f.f.}"] \ar[urr, equal]
    & \overline{\mathcal{D}}^\oplus \ar[ur, hookrightarrow]
\end{tikzcd}
\]
where $\mathcal{E}$ is the full image of $\cII: \cC \to \ZZC$,
$\mathcal{D}$ is the essential image of $i$,
and f.f. stands for fully faithful.
Let us explain this diagram and its relation to the proof above.
The observation that any object $\cII X$
is isomorphic to some $iX'$ can be restated as
the fact that the full image $\mathcal{E}$
of $\cII$ is contained in the essential image
$\mathcal{D}$ of $i$.
It follows that the Karoubian completion of $\mathcal{D}$
(which can be thought of as a subcategory of $\ZZC$ since $\ZZC$ is abelian)
contains the Karoubian completion of $\mathcal{E}$.
By \lemref{lem:fully_faithful},
$\mathcal{D}$ is a full subcategory of $\ZZC$,
so the Karoubian completion of $\mathcal{D}$ is $\mathcal{D}$ itself.
By \prpref{prp:elliptic_direct_summand}, any object in $\ZZC$
is a direct summand of some $\cII X$,
hence the Karoubian completion of
$\mathcal{E}$ is the entire $\ZZC$.
So $\ZZC = \overline{\mathcal{E}}^\oplus
  \subseteq \overline{\mathcal{D}}^\oplus
  \subseteq \ZZC$,
hence $\mathcal{D} = \overline{\mathcal{D}}^\oplus = \ZZC$.
In other words, $i$ is essentially surjective.\\

By \prpref{prp:I1_tensor}, $\cI_1$ is a monoidal functor,
and $\iota$ is also monoidal in a natural way.
Thus,

\begin{corollary}
If $\cC$ is modular, then
$(i,j) = (\cI_1,J)\circ \iota : \cC \to \ZZC$ is a monoidal equivalence.
\end{corollary}

\subsection{Connection to Reshetikhin-Turaev Central Charge Anomaly}
\label{sec:modularRT}

One manifestation of the central charge anomaly in 
Reshetikhin-Turaev theory at $\cC$, $Z_{\textrm{RT};\cC}$,
already discussed in Witten's work \ocite{W},
is that one has a projective action of the 
mapping class group of a closed surface $\Sigma$
on $Z_{\textrm{RT},\cC}(\Sigma)$,
and the deviation from being an honest (i.e. not projective)
action is known as the central charge anomaly.
In particular, for the torus $\Sigma = \torus$,
we have $Z_{\textrm{RT},\cC}(\Sigma) = \Hom_\cC(\one, \induct{})$,
and there are morphisms $\xi_s,\xi_t$ in $\End_\cC(\induct{})$
so that post-composing with them gives
an action of $\wdtld{\SLZ}$ on $\Hom_\cC(\one, \induct{})$,
but only factors through a projective action
of $\SLZ$ on $\Hom_\cC(\one, \induct{})$.
One also gets projective actions on $\Hom_\cC(U, \induct{})$
for simple $U$ (see e.g. \ocite{BakK}*{Section 3.1}).\\

We recover this projective representation in $\ZZC$,
in part from the $\SLZ$-action on $\ZZC$ as defined in
\secref{sec:sl2z}.
Consider the $\Hom$ space
\[
  \Hom_\ZZC(i\one, \cII\one)
\]
By \lemref{lem:I1_kills_braiding},
$\cII\one = \cI_1(\cI \one) \cong \cI_1(\iota \induct{}) = i(\induct{})$,
so
\[
  \Hom_\ZZC(i\one, \cII\one)
  \cong \Hom_\ZZC(i\one, i(\induct{}))
  \cong \Hom_\cC(\one, \induct{})
\]
where the second equality follows from \thmref{thm:modular}
(or just \lemref{lem:fully_faithful}).\\

So we want to describe a projective $\SLZ$-action on 
$\Hom_\ZZC(i\one, \cII\one)$. To this end,
we put a projectively-$\SLZ$-equivariant structure
on $i\one$,
and an $\SLZ$-equivariant structure on $\cII\one$.\\

For $\cII\one$, let
\begin{align*}
  \nu_s : U_s(\cII\one) \to \cII\one \\
  \nu_t : U_t(\cII\one) \to \cII\one
\end{align*}
be the isomorphisms
\[
  \picmid{\includegraphics[height=60pt]{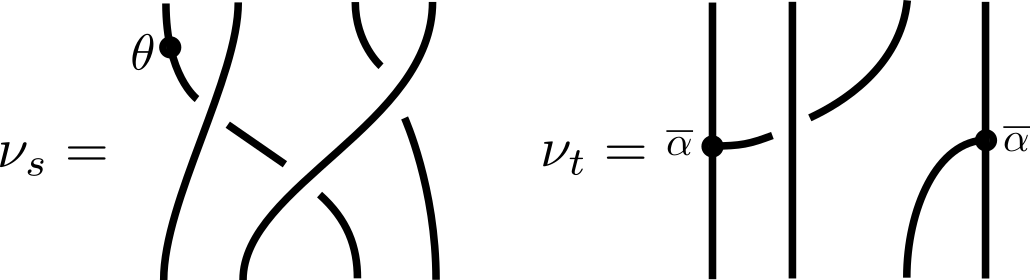}}
\]
For $i\one$, let
\begin{align*}
  \mu_s : U_s(i\one) \to i\one \\
  \mu_t : U_t(i\one) \to i\one
\end{align*}
be the isomorphisms
\[
  \picmid{\includegraphics[height=60pt]{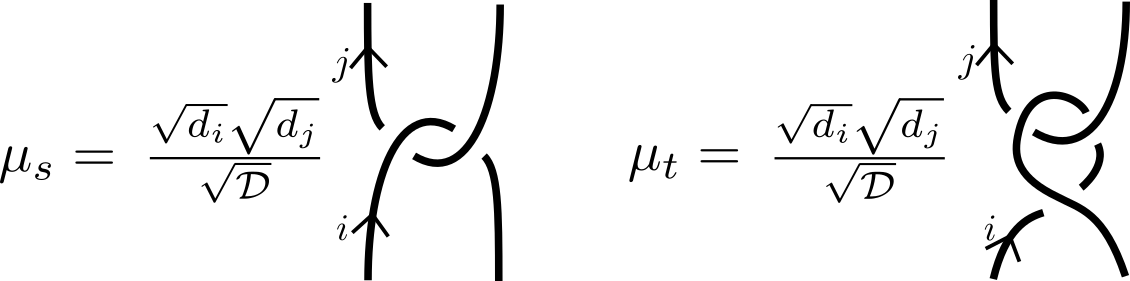}}
\]

\begin{proposition}
\label{prp:eqvrt_obj}
With respect to the $\SLZ$-action on $\ZZC$
defined in \secref{sec:sl2z},
$\nu_s,\nu_t$ define an $\SLZ$-equivariant structure
on $\cII \one$,
and $\mu_s,\mu_t$ define a projectively-$\SLZ$-equivariant structure
on $i\one$.
\end{proposition}

We do not prove this here, since we have not shown the validity
of the $\SLZ$-action on $\ZZC$ in \secref{sec:sl2z}.
Even if we grant that the action is well-defined,
some of the computations are lengthy and belies their
topological origins.\\

So let us grant \prpref{prp:eqvrt_obj} for now.
$\SLZ$ acts on $\Hom_\ZZC(i\one, \cII\one)$ as follows:
for $g\in \SLZ$ and $\psi \in \Hom_\ZZC(i\one, \cII\one)$,
\[
  \rho_g(\psi) = \nu_g \circ \psi \circ \mu_g^\inv:
  i\one
    \xrightarrow[]{\mu_g^\inv} U_g(i\one)
    \xrightarrow[]{U_g(\psi) = \psi} U_g(\cII\one)
    \xrightarrow[]{\nu_g} \cII\one
\]

We use the explicit isomorphism
$\Phi: \cII\one \to i(\induct{})$,
\[
  \includegraphics[height=60pt]{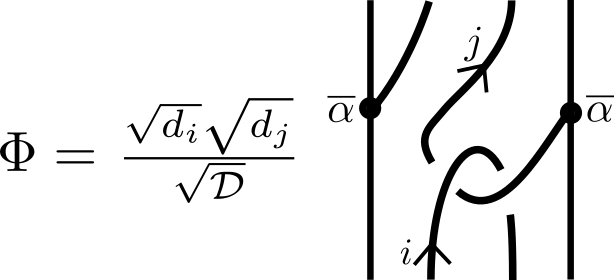}
\]
to identify $\Hom_\ZZC(i\one, \cII \one) \to \Hom_\ZZC(i\one, i(\induct{}))$.\\

The actions of $s,t\in \SLZ$ on $\Hom_\cC(\one, \induct{})$
are given by post-composing with
\begin{align*}
  \xi_s &= \mu_s^\inv \\
  \xi_t &= (\theta^\inv \tnsr \id) \circ \mu_s \circ (\theta^\inv \tnsr \id)
\end{align*}
respectively, where we just think of $\mu_s$ as a morphism in $\cC$,
and $\theta$ is the balancing structure of $\cC$.
(Note that in \ocite{BakK},
the actions of $s,t$ are given by $\xi_s, \theta \tnsr \id$,
respectively.
These actions are related
by twisting by the automorphism
of $\SLZ$ sending $s \mapsto s, t \mapsto t^\inv s^\inv t^\inv$.)\\

Finally, we may state the connection of our work
to the central charge anomaly in Reshetikhin-Turaev theory as follows:

\begin{proposition}
\begin{align*}
  \Hom_\cC(\one, \induct{}) &\to \Hom_\ZZC(i\one, \cII\one) \\
  \vphi &\mapsto \Phi^\inv \circ i(\vphi)
\end{align*}
intertwines the projective $\SLZ$-action
in Reshetikhin-Turaev theory
and the one described above.
\end{proposition}

Once again we do not prove this here
for the same reasons as before,
but we verify it for the action of $s\in \SLZ$.
Since $\{\coev_{X_j}\}_{j\in J}$ forms a basis of
$\Hom_\cC(\one, \induct{})$,
it suffices to consider $\vphi_j := \coev_{X_j}$,
for which we check that
$\Phi \circ \rho_s(\Phi^\inv \circ i(\vphi_j)) = i(\xi_s \circ \vphi_j)$
below
(we leave out coefficients for readability,
and diagrams read left to right):

\begin{align*}
  \Phi \circ \rho_s (\Phi^\inv \circ i(\vphi_j))
  &= \picmid{\includegraphics[height=90pt]{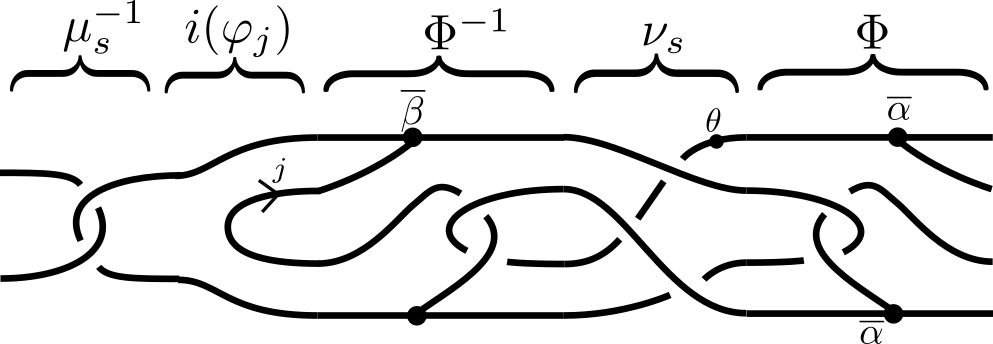}}\\
  &= \picmid{\includegraphics[height=60pt]{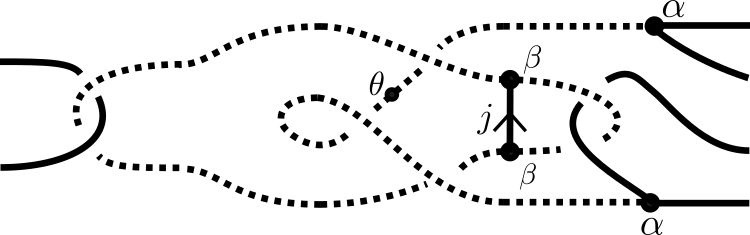}}\\
  &= \picmid{\includegraphics[height=60pt]{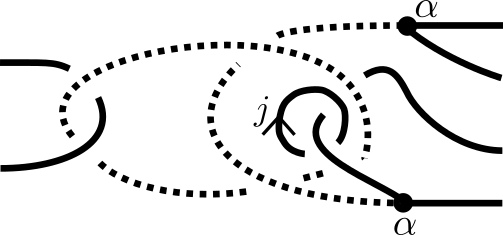}}\\
  &= \picmid{\includegraphics[height=50pt]{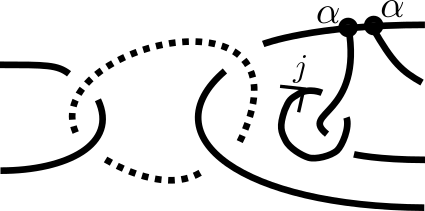}}
   = \picmid{\includegraphics[height=50pt]{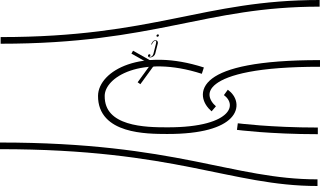}}
   = i(\xi_s \circ \vphi_j)
\end{align*}
The computation for $t$ is similar.\\

To conclude this section,
let us say a few words in relation to the extended Crane-Yetter TQFT.
$\cII\one$ is the ``empty configuration"
in $\Zcy(\punctorus)$, with no marked points or projections,
so it is natural to expect that it has an $\SLZ$-equivariant structure.
$i\one$ is the object $\Zcy(M)$,
where $M$ is the solid torus with a point removed from its interior.
$M$ depends on the choice of an isotopy class of
a simple closed curve on $\torus$,
the one that is made contractible in $M$.
Clearly the action of $\SLZ$ on $\del M \cong \torus$
doesn't extend to $M$,
since it doesn't fix isotopy classes of curves on $\torus$,
so while $U_s i\one \cong Ti\one \cong i\one$,
these isomorphisms are not expected to be
coherent with the action of $\SLZ$ on
$\Zcy(\punctorus) \cong \ZZC$.

\section{$\cC = H-\modl$ and Elliptic Drinfeld Double}
\label{sec:C_Hmod}

In this section, we treat the case when $\cC = H-\modl$
for a finite-dimensional \emph{quasi-triangular}
Hopf algebra $H$
(as in \secref{sec:drfld_symmetric}, once again
we implicitly assume semisimplicty of $H$ to keep
$\cC$ semisimple, but we never really use it).\\

In \secref{sec:drfld_symmetric},
we saw that for a Hopf algebra $H$,
$\cZ(H-\modl) \cong \cD(H)-\modl$,
where $\cD(H)$ is Drinfeld's quantum double
(see \prpref{prp:ZRep}).
Here we will construct a similar algebra $\cDD(H)$
associated to a \emph{quasi-triangular} Hopf algebra $H$,
which we call the
\emph{Elliptic Drinfeld double}.
We then show that
$\cZ^2(H-\modl) \cong \cDD(H)-\modl$
(see \thmref{thm:ZZRep_alg} below).
While $\cD(H)$ is a ribbon algebra, in general, $\cDD(H)$
has no obvious coalgebra structure, but it does
when $H$ is cocommutative.

\begin{remark}
As mentioned in the introduction,
\ocite{BJ} defines a similar algebra which
the authors also call the elliptic double,
and we expect these elliptic doubles to
be Morita equivalent.
\end{remark}

Let us first define $\cDD(H)$ as an algebra:

\begin{defn-prop}
\label{def:elliptic_double}
Let $(H, m, 1, \Delta, \veps, S, \cR, v, u)$ be a finite-dimensional
ribbon Hopf algebra,
where $v,u$ are the pivotal and ribbon elements respectively.
The \emph{Elliptic Drinfeld double} of $H$, denoted $\cDD(H)$,
is defined as the vector space $H \tnsr H_1^\dcop \tnsr H_2^\dcop$
(where $H_1 = H_2 = H$, indexing for clarity)
with the following algebra structure:
\begin{itemize}
\item the three obvious inclusions of
  $H$, $H_1^\dcop$, and $H_2^\dcop$
  into $H \tnsr H_1^\dcop \tnsr H_2^\dcop$ are algebra maps.

\item Each copy of $H^\dcop$ commutes past $H$
  in the same manner as in $\cD(H)$:
  e.g. for $h \in H$ and $f \in H_1^\dcop$,
  \[
    fh =
    \eval{f_3,S^\inv(h_1)}\eval{f_1,h_3}
      h_2 f_2
  \]
  where we use Sweedler's notation 
  $\Delta^2(h) = h_1 \tnsr h_2 \tnsr h_3$
  and $\Delta^2(f) = ((m^*)^\cop)^2(f) = f_3 \tnsr f_2 \tnsr f_1$.
  In other words, the inclusions
  \begin{align*}
    \cD(H) \cong H \tnsr H_1^\dcop
      &\overset{\iota_1}{\subseteq} \cDD(H)\\
    \cD(H) \cong H \tnsr H_2^\dcop 
      &\overset{\iota_2}{\subseteq} \cDD(H)
  \end{align*}
  are algebra maps.

\item The two copies of $H^\dcop$ commute by the following relation:
  writing $\cR = s_a \tnsr t_a$ (suppressing the sum),
    for $f^1 \in H_1^\dcop, f^2 \in H_2^\dcop$,
  \[
    f^2 f^1 =
      \eval{f_1^1, t_a}\eval{f_3^1, s_{a'}}
      \eval{f_1^2,s_a}\eval{f_3^2,t_{a'}}
      f_2^1 f_2^2
  \]

\end{itemize}
\end{defn-prop}

Note that the coproduct on $H^\dcop$ is not used here,
but will be important later when considering
the (symmetric) tensor product structures on $\ZZC$ and $\cDD(H)-\modl$.

\begin{theorem}
\label{thm:ZZRep_alg}
For $\cC = H-\modl$, $\ZZC \cong \cDD(H)-\modl$
as abelian categories.
\end{theorem}

\begin{proof}
The proof will be very similar to $\cZ(H-\modl) \cong \cD(H)-\modl$
(see \prpref{prp:ZRep}).\\

The functor $\cDD(H)-\modl \to \ZZC$ is constructed as follows.
Let $X$ be a left $\cDD(H)$-module. It is in particular an $H$-module,
i.e. an object in $\cC$. The action of $H_1^\dcop$ on $X$
gives us one half-braiding: for $A$ another $H$-module,
\[
  \lambda^1_A = P \circ R_1 : A \tnsr X \to X \tnsr A
\]
where $P$ is the swapping of factors,
and $R_1$ stands for acting by
$R_1 = (\iota_1 \tnsr \iota_1)(R) = \sum h_i \tnsr \iota_1(h_i^*)$,
where recall $\iota_1$ is the first inclusion of algebras
$\cD(H) \cong H \tnsr H_1^\dcop \subseteq \cDD(H)$
(we suppress $\iota_1$ on $H$ because $H \hookrightarrow \cDD(H)$
is unambiguous).
Likewise, we can define a second half-braiding by
\[
  \lambda^2_A = P \circ R_2 : A \tnsr X \to X \tnsr A
\]
where $R_2 = (\iota_2 \tnsr \iota_2)(R)$.\\

We need to show that $\lmb^1,\lmb^2$ satisfy COMM,
and it suffices to check it for $A = B = H$
by the naturality of half-braidings.
This boils down to checking that
\[
  \cR^{12} R_1^{13} R_2^{23} = R_2^{23} R_1^{13} (\cR^\inv)^{21}
  \in H \tnsr H \tnsr \End_\kk(X)
\]
or equivalently,
\[
  \cR^{12} R_1^{13} R_2^{23} \cR^{21} = R_2^{23} R_1^{13}
\]
Here $\cR^{21} = (\cR^\op)^{12}$. Once again we point out that
in COMM, one side has $c_{-,-}$ while the other has $c_{-,-}^\inv$,
hence the appearance of $(\cR^\inv)^\op$. So
\[
  s_a h_i t_{a'} \tnsr t_a h_j s_{a'} \tnsr \iota_1(h_i^*) \iota_2(h_j^*)
  = h_l \tnsr h_k \tnsr \iota_2(h_l^*) \iota_1(h_k^*)
\]
where $\cR = s_a \tnsr t_a$.
For $f^1,f^2 \in H^*$,
applying $f^2 \tnsr f^1 \tnsr \id$, we get
\begin{align*}
  \iota_2(f^2) \iota_1(f^1)
  &= \eval{f^2,s_a h_i t_{a'}} \eval{f^1,t_a h_j s_{a'}}
      \iota_1(h_i^*) \iota_2(h_j^*)\\
  &= \eval{f_1^2,s_a} \eval{f_3^2,t_{a'}}
      \eval{f_1^1,t_a} \eval{f_3^1,s_{a'}}
      \iota_1(f_2^2) \iota_2(f_2^1)
\end{align*}
which is implied by the commutation relation between
$H_1^\dcop$ and $H_2^\dcop$ in $\cDD(H)$.\\

For the other way, let $(X, \lambda^1, \lambda^2)$.
Using the same methods in the proof of \prpref{prp:ZRep},
for each half-braiding $\lambda^1, \lambda^2$,
we cook up two $H^*$-actions on $X$,
so that we have an action of $H * H_1^* * H_2^*$
on $X$, where the $*$ denotes free product of algebras.
To see that this action factors through $\cDD(H)$,
we check that the commutation relations between factors
$H, H_1^\dcop, H_2^\dcop$ of $\cDD(H)$
are satisfied in their actions on $X$.
For commutation relations between $H_1^\dcop$ and $H_2^\dcop$,
it basically follows from the same computations above but in reverse,
while for the other two pairs,
it follows from the proof of \prpref{prp:ZRep}.
\end{proof}

\begin{corollary}
  If $H$ is semisimple, then so is $\cDD(H)$.
\end{corollary}

\begin{remark}
\label{rmk:elliptic_slz}
The elliptic double in \ocite{BJ}
carries an action of $\wdtld{\SLZ}$,
but ours do not.
Via the equivalence in \thmref{thm:ZZRep_alg},
the $\SLZ$-action on $\ZZC$ laid out in \secref{sec:sl2z}
defines an $\SLZ$-action on $\cDD(H) - \modl$,
hence some sort of action on $\cDD(H)$ as well.
However, due to the flexibility of group actions
on categories
(manifested in the extra data of natural isomorphisms
  $\gamma_1, \gamma_2$),
we don't get an honest action on $\cDD(H)$.\\
\indent More precisely, from reconstruction theory,
we have equivalences
\[
\begin{tikzcd}
  \End(F)-\modl \ar[r, equal]
    & \End(F \circ U_s)-\modl \ar[r]
    & \End(F)-\modl \\
    & \ZZC \ar[u] \ar[r, "U_s"]
    & \ZZC \ar[u]
\end{tikzcd}
\]
where $F$ is the forgetful functor to $\Vect$
(forgetting the half-braidings and $H-\modl$ structure.
Since $U_s$ only changes the half-braidings,
$F = F \circ U_s$).
Thus $s\in \SLZ$ acts on the right by
\begin{align*}
  \psi_s: \cDD(H) \cong \End(F) \xrightarrow[]{-\circ U_s} \End(F\circ U_s)
    = \End(F) \cong \cDD(H)
\end{align*}
and similarly for $t$.
Concretely, since $U_s$ changes the half-braidings
$(X, \lmb^1,\lmb^2) \mapsto (X, (\lmb^2)^\dagger, \lmb^1)$,
it changes the resulting action of $H_1^\dcop,H_2^\dcop$ on
$X$, and $\psi_s$ encodes this
(and likewise for $U_t$).\\
\indent However, for the relation $r_1: s^4 = 1$,
we don't get $\psi_s^4 = \id_{\cDD(H)}$,
but we have
\[
\begin{tikzcd}
  \End(F \circ U_s)-\modl
    \ar[rr, bend left=30, ""{name=U, below}, "U_s^4"]
    \ar[rr, bend right=30, ""{name=D}, "\id"']
    \ar[Rightarrow, from=U, to=D, "\gamma_1 = \theta"]
    & & \End(F)-\modl
\end{tikzcd}
\]
So instead we have
\[
  \theta: U_s^4(X,\lmb^1,\lmb^2) \to (X,\lmb^1,\lmb^2)
\]
hence $u$, the ribbon element of $H$, intertwines:
\[
  u \cdot \psi_s^4(x) = x \cdot u
\]
for every $x \in \cDD(H)$.
This suggests an $\wdtld{\SLZ}$-action,
but the other relation $r_2 : sts = t^\inv s t^\inv$
also requires a functorial isomorphism $\gamma_2$.
So instead we have the action of the mapping class
group of a genus 1 surface with one puncture
and one boundary componenet,
which is some extension of $\wdtld{\SLZ}$
by $\pi_1(\punctorus) \cong \ZZ * \ZZ$;
this is not too surprising
since our definition of $\ZZC$
(and hence $\cDD(H)$) comes from fixing a point
and a pair of meridian and longitude (see Introduction),
giving rise to a unnaturalness of the action.
We will investigate this and connections
to the elliptic double in \ocite{BJ}
further in upcoming work.
\hfill $\triangle$\\
\end{remark}

\subsection{$H$ cocommutative, $\cC$ symmetric}
\label{sec:symmetric_C}

In the usual Drinfeld center case, we have a braided tensor
equivalence $Z(H-\modl) \simeq \cD(H)-\modl$
(see \prpref{prp:ZRep}).
The proof presented there implicity uses the 
perspective of reconstruction theory of finite dimensional
Hopf algebras, in that we start with the forgetful functor
$F : Z(H-\modl) \to \Vect$,
and $\cD(H)$ appears as the Hopf algebra
of endomorphisms of $F$
(see also \rmkref{rmk:elliptic_slz}).
We can try to do the same thing with
the tensor product structure on $\ZZC$ sketched in
\secref{sec:tensor_product}.
However, since the tensor product on $\ZZC$ defined
in \secref{sec:tensor_product}
is generally multitensor but not tensor,
the obvious forgetful functor to $\Vect$ cannot be tensor,
so there's no clear way to use reconstruction theory to recover a
coalgebra structure on $\cDD(H)$.\\

If we restrict ourselves to $H$ cocommutative
(so that it is quasi-triangular with $\cR = 1 \tnsr 1$
and $\cC$ is symmetric),
we can consider the naive tensor product
\[
  (X,\lmb^1,\lmb^2) \odot (Y,\mu^1,\mu^2)
  = (X\tnsr Y, \lmb^1\tnsr\mu^1, \lmb^2\tnsr\mu^2)
\]
on $\ZZC$, and here $\lmb^1\tnsr\mu^1, \lmb^2\tnsr\mu^2$
indeed satisfy COMM:
\[
  \includegraphics[height=60pt]{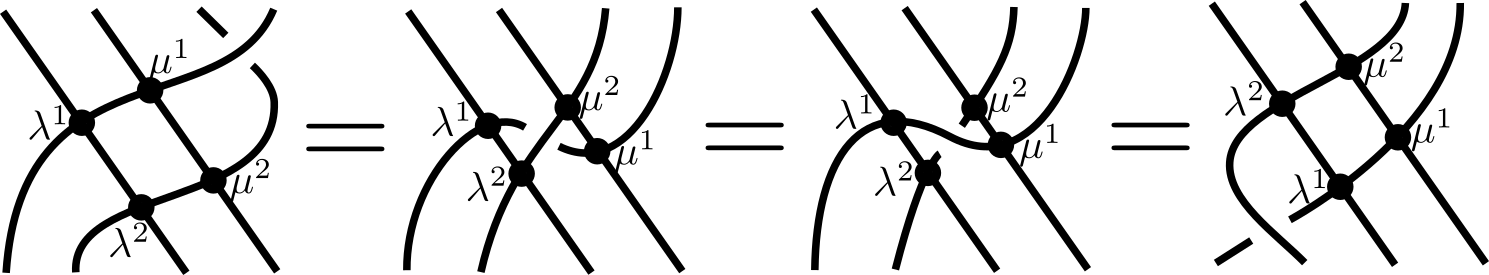}
\]
The associativity constraint will just be the one from $\cC$.
The obvious forgetful functor to $\Vect$ has an obvious tensor structre,
so we can apply reconstruction theory again.
In particular, we may upgrade $\cDD(H)$ to a ribbon Hopf algebra:

\begin{defn-prop}
\label{def:elliptic_double_cocomm}
Let $(H, m, 1, \Delta, \veps, S, v)$ be a finite-dimensional
cocommutative ribbon Hopf algebra
(with $\cR = 1 \tnsr 1$ and $u=1$).
The elliptic Drinfeld double $\cDD(H)$ as defined
in \defref{def:elliptic_double}
admits the following additional structure,
making it a ribbon Hopf algebra:
\begin{itemize}

\item As a coalgebra, it is simply $H \tnsr H_1^\dcop \tnsr H_2^\dcop$,
  i.e. $\Delta(h \tnsr f^1 \tnsr f^2) =
    (h_1 \tnsr f_2^1 \tnsr f_2^2) \tnsr (h_2 \tnsr f_1^1 \tnsr f_2^2)$,

\item The antipode is also given componentwise,
  i.e. 
  \[
    S(h f^1 f^2) = S^\inv(f_2) S^\inv(f_1) S(h)
  \]
  where $f^1\in H_1^\dcop, f^2 \in H_2^\dcop$.

\item $v \in H \hookrightarrow \cDD(H)$ is the pivotal element.

\end{itemize}
\end{defn-prop}

\begin{proof}
Checking compatibility between the various structures
boils down to familiar computations.
\end{proof}

Observe that in this case of $H$ cocommutative,
the actions of $H_1^\dcop$ and $H_2^\dcop$
commute, since $\cR = 1\tnsr 1$, so
\[
  f^2 f^1
  = \eval{f_1^1,1}\eval{f_3^1,1}
    \eval{f_1^2,1}\eval{f_3^2,1}
    f_2^1 f_2^2
  = \veps(f_1^1)\veps(f_3^1)
    \veps(f_1^2)\veps(f_3^2)
    f_2^1 f_2^2
  = f^1 f^2
\]

\begin{remark}
By Deligne's theorem on tensor categories \ocite{De1},\ocite{De2},
any symmetric fusion category is
tensor equivalent to $\Rep(G)$ for some finite group $G$,
and is braided tensor equivalent to it up to a twist
by some central element $z$ of order 2.
Thus, $H-\modl$ covers basically all symmetric fusion categories.
\end{remark}

\begin{theorem}
\label{thm:ZZRep_alg_symm}
When $H$ is cocommutative, the equivalence in \thmref{thm:ZZRep_alg}
is a tensor equivalence.
\end{theorem}

\begin{proof}
Essentially the same as in the proof of \prpref{prp:ZRep}.
\end{proof}

\begin{example}[Group Algebra]
\label{xmp:group_algebra_elliptic}

Recall the setup of \xmpref{xmp:group_algebra}:
$H = \kk[G]$, $H^\dcop = F(G)^\cop$,
so we have
\[
  \cDD(H) = \kk[G] \tnsr F(G_1^\op) \tnsr F(G_2^\op)
    = \kk[G] \tnsr F(G_1^\op \times G_2^\op)
\]
as coalgebras, where of course $G_1 = G_2 = G$
(the second equality is justified because the actions of $F(G_1^\op)$
and $F(G_2^\op)$ commute).
Then the commutation relations read
\[
  \delta_{(g_1,g_2)} h = h \delta_{(h^\inv g_1 h, h^\inv g_2 h)}
\]

Denote $\cDD(G) := \cDD(\kk[G])$.
Similar to \xmpref{xmp:group_algebra},
representations of $\cDD(G)$ can be interpreted as
$G$-equivariant vector bundles over $G\times G$,
where $G$ acts on $G\times G$ by conjugation on each factor.

The diagonal map $G \hookrightarrow G \times G$
is $G$-equivariant, and pulls back $G$-equivariant bundles,
giving us a restriction functor
$\cDD(G)-\modl \to \cD(G)-\modl$.
On the level of algebras, this corresponds to the inclusion
\begin{align*}
  \cD(G) &\hookrightarrow \cDD(G) \\
  g &\mapsto g \\
  \delta_g &\mapsto \delta_{(g,g)}
\end{align*}

This is not a coalgebra map, or equivalently,
the restriction functor is not tensor.
For example, bundles supported on the orbits of
$(g,1)$ and $(1,g)$, respectively,
would each restrict to 0 on the diagonal,
but their tensor product would have a non-trivial
vector space over $(g,g)$.\\

However, the diagonal inclusion $G \hookrightarrow G\times G$
induces a push-forward functor
\[
  \cD(G)-\modl \to \cDD(G)-\modl
\]
and this is tensor;
equivalently, it is easy to verify that the projection
$\cDD(G) \to \cD(G)$ induced from the central idempotent
$\sum_g \delta_{(g,g)} \in \cDD(G)$
is a coalgebra map.
In terms of half-braidings, this functor
\[
  \cZ(G-\modl) \to \cZ^2(G-\modl)
\]
is given by
\[
  (X,\lmb) \to (X,\lmb,\lmb)
\]

\hfill $\triangle$
\end{example}

\begin{example}
\label{xmp:group_algebra_elliptic_tnsr}
We may consider the group algebra example above,
but instead consider the tensor product discussed in
\secref{sec:tensor_product}.
So our objects are still $G$-equivariant bundles
over $G\times G$,
but the tensor product of two bundles
$V = \dirsum_{g_1,g_2} V_{(g_1,g_2)}$ and
$W = \dirsum_{h_1,h_2} W_{(h_1,h_2)}$
is the image of the usual $V \tnsr W$
under the projection $Q_{\lmb^1,\mu^1}$
(see \defref{def:ZZC_tnsr}).\\

Recall from \thmref{thm:ZZRep_alg}
that to interpret a $\cDD(H)$-module $V$ as
an object in $\cZ^\textrm{el}(H-\modl)$,
the first half-braiding is given by
$\lmb^1 = P \circ R_1$,
where $P$ is the usual swapping of factors,
and $R_1 = \sum_j h_j \tnsr \iota_1(h_j^*)$,
and similarly for the second half-braiding.
Here $H=\kk[G]$,
so $R_1 = \sum_g g \tnsr \delta_{(g,*)}$,
where $\delta_{(g,*)} := \iota_1(\delta_g)
  = \sum_h \delta_{g,h}$.\\

Concretely, $Q_{\lmb^1,\mu^1}$ works out to the following.
We write it as a sum 
$Q_{\lmb^1,\mu^1} = \frac{1}{|G|}\sum_j \dim X_j \cdot Q^j$
(recall the dashed line represents a sum over simples,
weighted by $d_j = \dim X_j$, and $\cD = \sum_j d_j^2 = |G|$).\\

For each $j$, $Q^j$ works out to be
(implicitly summing over a basis $\{e_l\}$ of $X_j$)
\begin{align*}
v \tnsr w
  &\mapsto e_l \tnsr e_l^* \tnsr v \tnsr w \\
  &\mapsto \sum_g e_l \tnsr \delta_{(g,*)} \cdot v
                \tnsr g \cdot e_l^* \tnsr w \\
  &\mapsto \sum_g \delta_{(g,*)} \cdot v \tnsr e_l
                \tnsr w \tnsr g \cdot e_l^* \\
  &\mapsto \sum_{g,h} \delta_{(g,*)} \cdot v
                \tnsr \delta_{(h,*)} \cdot w
                \tnsr h \cdot e_l
                \tnsr g \cdot e_l^* \\
  &\mapsto \sum_{g,h} \eval{g \cdot e_l^*, h \cdot e_l}
                \delta_{(g,*)} \cdot v
                \tnsr \delta_{(h,*)} \cdot w \\
  &= \sum_{g,h} \tr_{X_j} (g^\inv h)
                \delta_{(g,*)} \cdot v
                \tnsr \delta_{(h,*)} \cdot w \\
\end{align*}

Then $Q_{\lmb^1,\mu^1}$ is
\begin{align*}
v \tnsr w
  &\mapsto \frac{1}{|G|} \sum_{j,g,h} \dim X_j \tr_{X_j}(h^\inv g)
                      \delta_{(g,*)}\cdot v \tnsr \delta_{(h,*)}\cdot w\\
  &= \frac{1}{|G|} \sum_{g,h} \tr_{\kk[G]}(h^\inv g)
                      \delta_{(g,*)}\cdot v \tnsr \delta_{(h,*)}\cdot w\\
  &= \sum_{g,h} \delta_{g,h}
                      \delta_{(g,*)}\cdot v \tnsr \delta_{(h,*)}\cdot w\\
  &= \sum_g \delta_{(g,*)}\cdot v \tnsr \delta_{(g,*)}\cdot w\\
\end{align*}

In short, $Q_{\lmb^1,\mu^1}$ is projection onto those
$V_{(g_1,g_2)} \tnsr W_{(h_1,h_2)} \subseteq V \tnsr W$
such that $g_1 = h_1$,
so that
\[
  (V \tnsr W)_{(g,h)} = \sum_{h_1 h_2 = h} V_{(g,h_1)} \tnsr W_{(g,h_2)}
\]

Thus, under this tensor product,
$\ZZC$ decomposes as a direct sum of $|G|$ copies
of $\ZC$ as monoidal categories:
\begin{align*}
  \ZZC &\simeq_\tnsr \dirsum_g \ZC \\
  V &\mapsto (\delta_{(g,*)} \cdot V)_g
\end{align*}

\hfill $\triangle$
\end{example}

\section{Concluding remarks, Future directions}
\label{sec:conclusion}

Recall that our motivation for constructing $\ZZC$
is to understand the extended Crane-Yetter TQFT,
and $\ZZC \cong \Zcy(\punctorus)$.
We have a similar construction of the category associated
to each open surface.
For example, the thrice-punctured sphere can also be thought of
as built out of two annuli, except that instead of plumbing
to get the once-punctured torus (see figures in \rmkref{rmk:pop}),
you identify a segment on the boundary of each annulus.
This results in a category with similar looking objects,
$(X,\lmb^1,\lmb^2)$, except that the compatibility relation
between the half-braidings $\lmb^1,\lmb^2$ is different:
instead of COMM, they should satisfy
the following variant, which was mentioned in
\rmkref{rmk:COMM_strands}:
\[
  \includegraphics[height=70pt]{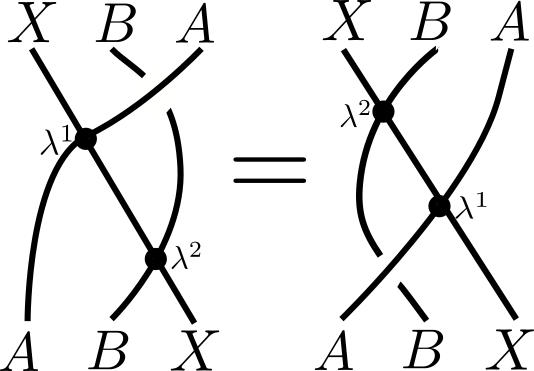}
\]
Note now the braidings used on both sides are
the same, where they were different in COMM.
While both the once-punctured torus and
the thrice-punctured sphere can be obtained from
a disk by attaching two 1-handles to the boundary,
the crucial difference is that
the 1-handles ``link" in the former
but don't in the latter.
In general, for a surface $\Sigma_{g,n}$ of
genus $g$ with $n > 0$ punctures,
upon presenting it as a disk with $2g + n-1$
1-handles attached to the boundary,
the associated category should consist of objects
of the form $(X,\lmb^1,\ldots,\lmb^{2g+n-1})$,
where $\lmb^i$ are half-braidings on $X$,
and pairs of $\lmb^i$'s satisfy some variant of COMM or the
above relation depending on whether the corresponding
1-handles are ``linked" or not.
The idea here is reminiscent of
the way ``gluing patterns" of a surface
are used to compute factorization homology in \ocite{BBJ1}.


\section{Appendix}
\subsection{Useful Lemmas for Computing with String Diagrams}
We record some useful results about string diagrams,
adapted mostly from \ocite{K}, \ocite{BalK}.\\

Let us denote
\[
  \eval{V_1,\ldots,V_n} = \Hom_\cC(\one, V_1\tnsr\ldots\tnsr V_n)
\]
There is a symmetric non-degenerate pairing
\[
\begin{tikzcd}
  \eval{V_1,\ldots,V_n} \tnsr \eval{V_n^*,\ldots,V_1^*}
    \ar[r] \ar[d, "P"]
  & \kk \ar[d, "="] \\
  \eval{V_n^*,\ldots,V_1^*} \tnsr \eval{V_1,\ldots,V_n}
    \ar[r]
  & \kk
\end{tikzcd}
\]
where $P$ is the usual swapping $W \tnsr W' \to W' \tnsr W$
of vector spaces,
and the horizontal arrows are
given by ($V = V_1\tnsr\ldots\tnsr V_n$,
$\vphi \in \eval{V}, f \in \eval{V^*}$)

\begin{align*}
\label{eqn:pairing}
  (\vphi,f) &=
  (\one \cong \one\tnsr\one
    \xrightarrow{\vphi \tnsr f} V \tnsr V^*
    \xrightarrow{\wdtld{ev}_V} \one) \\
  (f,\vphi) &=
  (\one \cong \one\tnsr\one
    \xrightarrow{f \tnsr \vphi} V^* \tnsr V
    \xrightarrow{\ev_V} \one)
\end{align*}

We will use the following convention:
if a figure contains a pair of vertices,
one with outgoing edges labelled $V_1,\ldots,V_n$,
and the other with outgoing edges labelled
$V_n^*,\ldots,V_1^*$,
and the vertices are labelled by the same greek letter,
say $\alpha$,
it will stand for
\[
  \includegraphics[height=60pt]{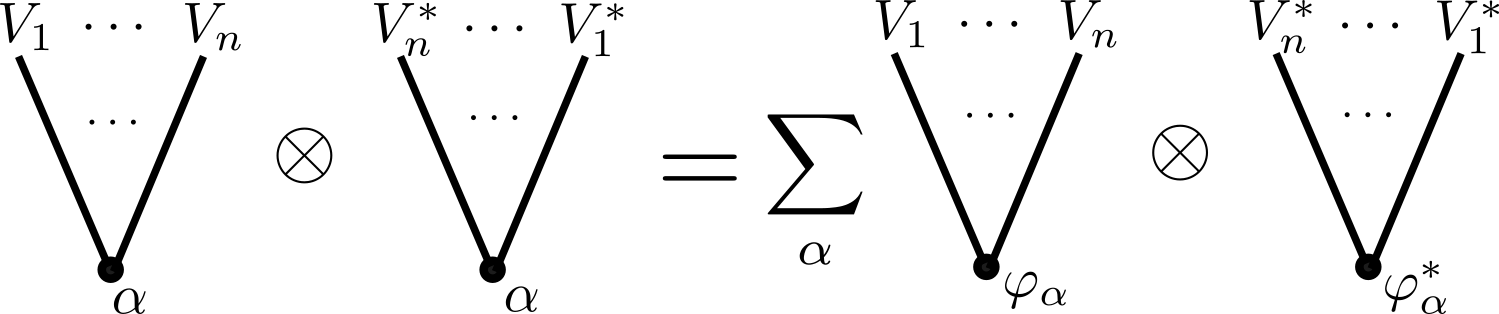}
\]
where $\{\vphi_\alpha\}, \{\vphi_\alpha^*\}$
are a pair of dual bases of 
$\eval{V_1,\ldots,V_n}, \eval{V_n^*,\ldots,V_1^*}$
respectively,
dual respect with respect to the pairing above.\\

We also establish the following convention:
when $\alpha$'s (or any pair of greek letters)
appear with a bar $\overline{\alpha}$,
and two pairs of edges are labelled with small-case latin alphabets
it will stand for the following sum:
\begin{align*}
  \includegraphics[width=0.8\linewidth]{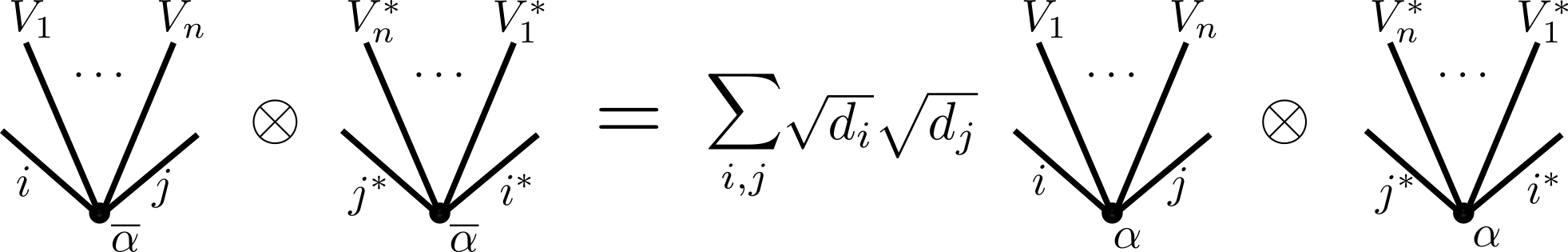}
\end{align*}
The $i,j$ will also often be omitted when the context is clear.\\

Here are some lemmas that are mostly adapted from
\ocite{BalK} and \ocite{K}.
We leave out the proofs, which are standard.

\begin{lemma}
\label{lem:identity_combine}
\begin{align*}
  \includegraphics[height=100pt]{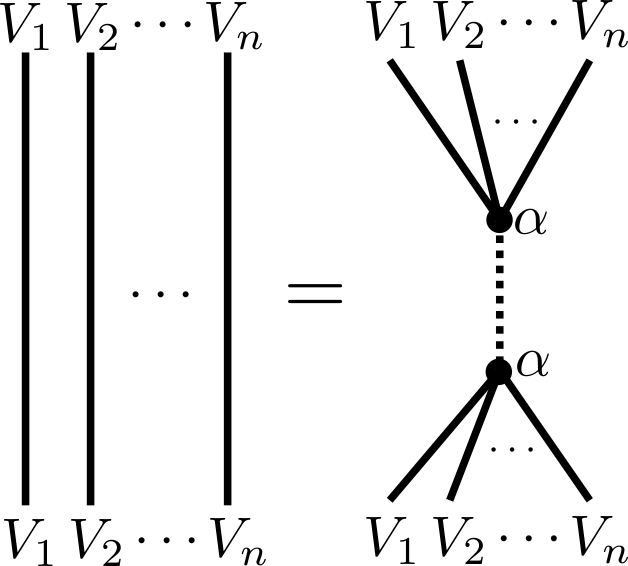}
\end{align*}
(recall the convention of dashed line in the introduction)
\end{lemma}

\begin{lemma}
\label{lem:disks_variant}
(Variant of Lemma 3.6 in \ocite{K})
For $\Phi : V \to W$,
\[
  \includegraphics[width=0.8\textwidth]{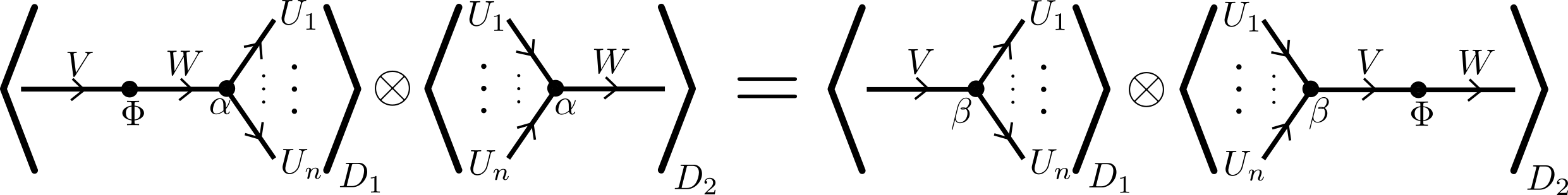}
\]

\end{lemma}

\begin{lemma}
\label{lem:Gamma_hf_brd}
$\Gamma$, as defined in \prpref{prp:adjoint},
is a half-braiding.
\end{lemma}

\begin{proof}
Naturality is immediate from \lemref{lem:disks_variant} above,
and respecting tensor product is checked as follows:
\[
  \includegraphics[height=100pt]{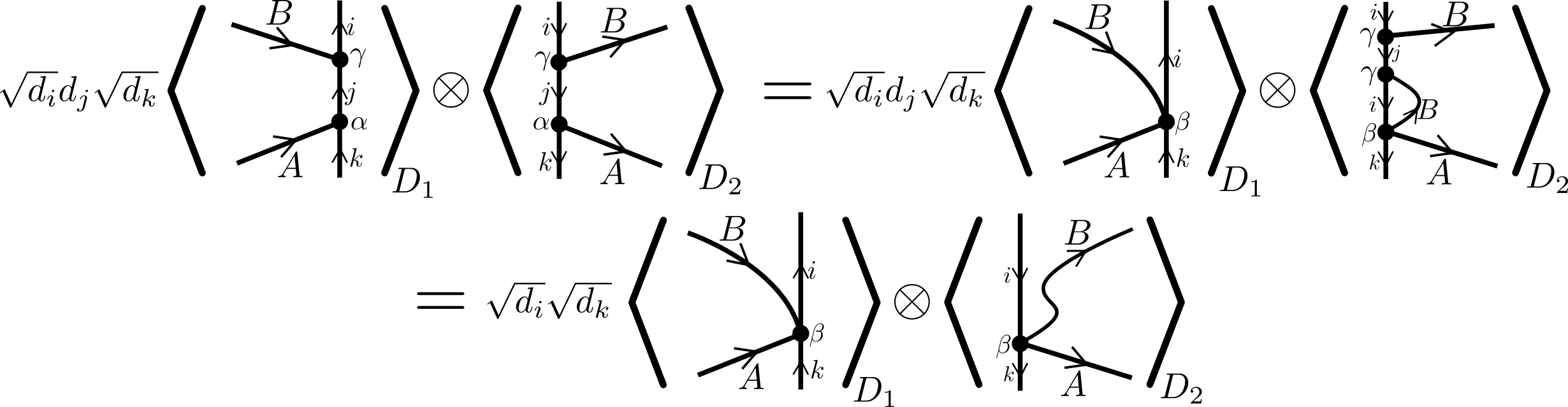}
\]
(Recall our convention of summing over all latin lower case labels.)
The first equality uses \lemref{lem:disks_variant}
with $\Phi = \beta$,
and the second follows from \lemref{lem:identity_combine}.
\end{proof}

The following lemma is often used when the half-braiding $\Gamma$
shows up, allowing us to switch the ``main branch" with the ``side branch"
(see proof of \prpref{prp:I1_tensor}).
\begin{lemma}
\label{lem:Gamma_switch}
\[
  \includegraphics[height=60pt]{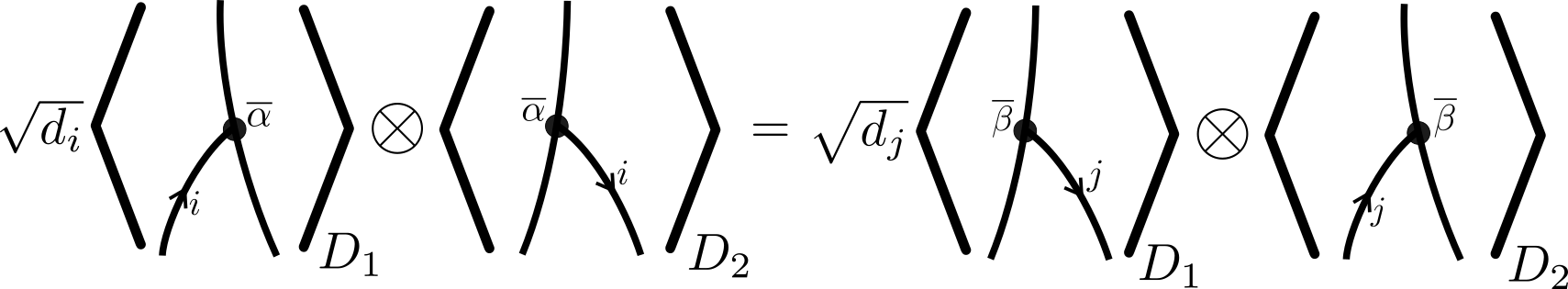}
\]
\end{lemma}

\begin{lemma}
\label{lem:charge_conservation}
(Charge conservation) When $\cC$ is modular,
\[
  \includegraphics[height=60pt]{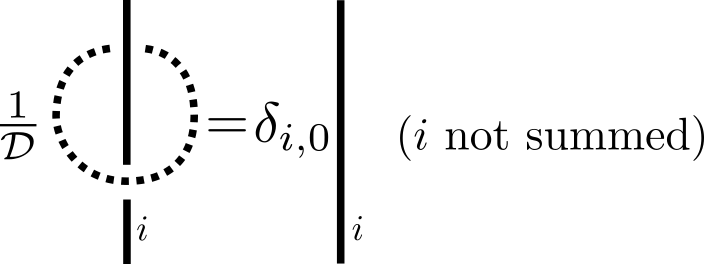}
\]
\end{lemma}

\begin{proof}
  See e.g. \ocite{BakK}*{Cor 3.1.11}.
\end{proof}

\comment{
TODO add sliding property, the version
for $P_{\lmb,\mu}$, and also show $Q_{\lmb,\mu}$.
}

\subsection{Group Actions on Categories by Generators and Relations}
\label{sec:group_action}

Let $G$ be a group acting on a category $\cA$,
in the sense of \ocite{EGNO}*{Section 2.7}.
This consists of an auto-equivalence $T_g : \cA \rcirclearrowleft$
for each $g\in G$,
and natural isomorphisms
$\gamma_{g,h}: T_g \circ T_h \to T_{gh}$
satisfying the cocycle condition
\[
\begin{tikzcd}
  T_gT_hT_k \ar[r, "T_g\gamma_{h,k}"]
    \ar[d, "\gamma_{g,h}T_k"]
  & T_gT_{hk} \ar[d, "\gamma_{g,hk}"]\\
  T_{gh}T_k \ar[r, "\gamma_{gh,k}"]
  & T_{ghk}
\end{tikzcd}
\]
For convenience, we will refer to this as the
\emph{usual} definition of a group action,
and in particular refer to the above as the
\emph{usual cocycle condition}.\\

It is convenient to rephrase this as simply a monoidal functor
\[
  (F,J): \Cat(G) \to_\tnsr \Aut(\cA)
\]
where $\Cat(G)$ is the monoidal category whose objects are
elements of $G$,
with only identity morphisms,
and the monoidal structure is the group operation.
Then in this interpretation,
$T_g = F(g)$,
the $\gamma$'s correspond to the natural isomorphism
$J_{g,h}: F(g) \circ F(h) \to F(gh)$,
and the cocycle condition is just the hexagon axiom
relating the associativity constraints with $J$
(both $\Cat(G)$ and $\Aut(\cA)$ are strict,
so the hexagon is just a square).\\

When $G$ is presented by generators and relations,
say $\SLZ = \eval{s,t|r_1,r_2}$ as in \secref{sec:sl2z},
we would like to describe a $G$-action on $\cA$
by generators and relations as well.
In a typical group action, say on some set $X$,
it suffices to provide an automorphism of $X$ for each generator,
and check that the relations are satisfied.
For an action on a category,
one provides an auto-equivalence for each generator,
a natural isomorphism for each relation,
and check certain equalities between compositions of such
natural isomorphisms;
these are the analogs of $T_g$, $\gamma_{g,h}$,
and the cocycle condition in the usual definition of a group
action on a category given above.
It is the goal of this note to spell this out in more detail.\\

In this note, we fix a group $G$ and a presentation of it,
$G = \eval{g_i|r_j}$.
Since we will be working with unreduced words,
we will include among the relations the trivial ones
$g_i g_i^\inv = 1$, $g_i^\inv g_i = 1$.
All words (henceforth assumed to be unreduced)
will be in the generators $g_i$
(and their inverses $g_j^\inv$).
We will think of a relation $r_j$
as a \emph{move} to transform one word into another;
more precisely, it is given by
a pair of words $v_{j,1}$, $v_{j,2}$,
so that for any words $x,y$,
we may transform $xv_{j,1}y$ into $xv_{j,2}y$
when working in $G$.
When we need to be precise,
we will denote this move by $x r_j y$,
and the inverse move by $x r_j^\inv y$.
(Ambiguities can arise, for example,
``applying $r_j$ to $v_{j,1}v_{j,1}$" could mean
$r_j v_{j,1} : v_{j,1} v_{j,1} \to v_{j,2} v_{j,1}$
or
$v_{j,1} r_j : v_{j,1} v_{j,1} \to v_{j,1} v_{j,2}$.)\\

First, let us give a definition:

\begin{definition}
\label{def:pre_grp_action_genrel}
Let $G=\eval{g_i|r_j}$ be a group presentation as above,
and let $\cA$ be a category.
A \textbf{\emph{pre-$G$-action on $\cA$ given by generators and relations}}
consists of the following data:
\begin{itemize}
  \item For each generator $g_i$,
    \textbf{auto-equivalences} $U_{g_i}, U_{g_i^\inv}: \cA \rcirclearrowleft$.
  \item For each relation $r_j : v_{j,1} = v_{j,2}$,
    a \textbf{natural isomorphism}
    $\gamma_j: U_{v_{j_1}} \to U_{v_{j_2}}$,
    where we write $U_w = U_{a_1}\ldots U_{a_k}$ for a word $w=a_1\ldots a_k$.
\end{itemize}
\hfill $\triangle$
\end{definition}

From $\gamma_j$, we also get, for words $x,y$, a natural isomorphism
$U_x \gamma_j U_y : U_{x v_{j_1} y} \to U_{x v_{j_2} y}$.
We will sometimes abuse notation and denote this natural isomorphism
as $\gamma_j$ too.\\

\begin{definition}
\label{def:grp_action_genrel}
In the set up of \defref{def:pre_grp_action_genrel},
a \textbf{\emph{$G$-action on $\cA$ given by generators and relations}}
is a pre-$G$-action that satisfies the following \textbf{cocycle condition}:
a sequence of moves (i.e. application of relation) 
$w_1 \xrightarrow[]{r_{j_1}} \ldots \xrightarrow[]{r_{j_p}} w_2$
gives rise to a sequence of natural isomorphisms
$U_{w_1} \xrightarrow[]{\gamma_{j_1}} \ldots \xrightarrow[]{\gamma_{j_p}} U_{w_2}$
whose composition is some natural isomorphism
$\gamma: U_{w_1} \to U_{w_2}$;
the cocyle condition says that
any sequence of moves from $w_1$ to $w_2$
results in the same $\gamma$.
\hfill $\triangle$
\end{definition}

\begin{remark}
We do not need to include the trivial relations $g_ig_i^\inv=1$
if we can guarantee that the $U_{g_i}$'s are isomorphisms,
and $U_{g_i^\inv} = U_{g_i}^\inv$.
This was the case in \secref{sec:sl2z}.
\hfill $\triangle$
\end{remark}

Now we justify this definition.
Let us first sketch an approach that is more intuitive.
Consider the following CW-complex $\Xi$:
the 0-skeleton/vertices is the set of unreduced words,
and for words $x,y$ and relation $r_j$,
there is a 1-cell $x r_j y$ going from $x v_{j,1} y$
to $x v_{j,2} y$.
The set of connected components is clearly
in bijective correspondence with $G$.
(We can even impose an H-group structure on $\Xi$
so that the obvious quotient map $\Xi \to G$
is a map of H-groups,
but since we are just giving intuition,
we leave this discussion to the more formal discussions to come.)\\

Given a pre-$G$-action (as in \defref{def:pre_grp_action_genrel}),
we can assign to a vertex $w$ the auto-equivalence $U_w$,
and to a 1-cell $x v_{j_1} y \to x v_{j_2} y$
we can assign the natural isomorphism
$U_x \gamma_j U_y : U_{x v_{j_1} y} \to U_{x v_{j_2} y}$.
Then the cocycle condition in \defref{def:grp_action_genrel}
means that for any loop beginning and ending at a vertex $w$,
the composition of natural isomorphisms
encountered along the loop is just the identity
$\id_{U_w} : U_w \rcirclearrowleft$.\\

So if we have a $G$-action in the sense of \defref{def:grp_action_genrel},
we can get a usual action of $G$
by picking a representative word $w_g$ for each $g\in G$,
and specifying $T_g = U_{w_g}$,
and $\gamma_{g,h}' : T_g T_h = U_{w_g w_h}
\xrightarrow[]{\gamma} U_{w_{gh}} = T_{gh}$,
where $\gamma$ is the appropriate composition of natural
isomorphisms - by the cocyle condition in the sense of
\defref{def:grp_action_genrel},
any choice gives the same natural isomorphism.
The cocycle condition for $T_g, \gamma_{g,h}'$
is automatically satisfied.\\

Now let us give a more precise discussion.
Let $G = \eval{g_i|r_j}$ as above.
Consider the following monoidal category $\wdtld{\cG}$:
its objects are all unreduced words in $g_i$'s.
The morphisms are given by compositions of applications of $r_j$'s;
more precisely, we consider the arrows
$q_{j,x,y}: x v_{j,1} y \to x v_{j,2} y$
for all words $x,y$,
and then a morphism $w_1\to w_2$ in $\wdtld{\cG}$
is just a (possibly empty) composable sequence of such arrows
or their reverse, reduced in the sense that an arrow and its
reverse cancel out.
(If we didn't include the trivial relations among the $r_j$'s,
we may end up with no morphisms between $g_ig_i^\inv$
and the empty word.)
Alternatively,
the set of morphisms from $w_1$ to $w_2$
is the set of homotopy classes of paths
from $w_1$ to $w_2$ in $\Xi$,
the CW-complex considered above.
The morphisms consisting of a single $q_{j,x,y}$
will be called \emph{simple}.
The monoidal structure on $\wdtld{\cG}$
is given by concatenation of words for objects,
and for morphism,
if we have morphisms $f: w_1\to w_2$ and $g: w_1'\to w_2'$,
we set
\[
  f\tnsr f' = f \circ f': w_1w_1'
    \xrightarrow[]{w_1 \cdot f'} w_1w_2'
    \xrightarrow[]{f \cdot w_2'} w_2w_2'
\]
It is easy to see that this forms a well-defined monoidal category.\\

Next consider $\cG$ to be the category with the same objects,
but with fewer morphisms: there is a unique morphism
$w_1\to w_2$ in $\cG$ if there is at least one morphism in $\wdtld{\cG}$,
and there are no morphisms otherwise. In other words,
\[
  \Hom_\cG(w_1,w_2) =
  \begin{cases}
    \{*\} \textrm{ if }\Hom_{\wdtld{\cG}}(w_1,w_2) \neq \emptyset \\
    \emptyset \textrm{ otherwise}
  \end{cases}
\]
The monoidal structure on $\cG$ is also concatenation.
There is an obvious monoidal ``quotient functor"
$\mathcal{Q}: \wdtld{\cG} \to \cG$
which is the identity on objects,
and identifies all morphisms with common source and target.\\

There is a canonical functor $\pi: \cG \to \Cat(G)$
that sends a word to the corresponding element in $G$,
and the morphisms are sent to the identity morphisms.
Clearly this functor is fully faithful and essentially surjective,
and is monoidal (in a unique way),
so is an equivalence of monoidal categories.\\

Thus, we can describe a $G$-action on $\cA$
by giving a monoidal functor
\[
  (F,J) : \cG \to_\tnsr \Aut(\cA)
\]
The point of using $\cG$ instead of $\Cat(G)$
is that we know $\cG$ is somehow built out
of generators and relations.
Precomposing such a monoidal functor with
an inverse to $\pi: \cG \to \Cat(G)$
gives a monoidal functor
$(F',J'): \Cat(G) \to_\tnsr \Aut(\cA)$,
recovering a group action in the usual sense.
Note that any two such inverses are naturally isomorphic
by a unique natural isomorphism.
We say more at the end of this section.\\

Let us see how $(F,J)$ gives us a $G$-action
in the sense of \defref{def:grp_action_genrel}.
$F$ in particular gives us,
for each generator $g_i$, auto-equivalences
\begin{align*}
  U_{g_i} &:= F(g_i) : \cA \to \cA \\
  U_{g_i^\inv} &:= F(g_i^\inv) : \cA \to \cA
\end{align*}
For a word $w=a_1\ldots a_k$,
where $a_l = g_i \textrm{ or } g_i^\inv$,
we write $U_w = U_{a_1}\ldots U_{a_l}$.
$J$ gives us $U_{a_1 a_2} = F(a_1)F(a_2) \to F(a_1 a_2)$,
and similarly, successive applications of $J$'s gives us
\begin{align*}
  U_w = F(a_1) \ldots F(a_k)
  & \xrightarrow[]{F(a_1)\ldots F(a_{k-2}) J_{a_{k-1},a_k}}
    F(a_1) \ldots F(a_{k-2}) F(a_{k-1}a_k) \\
  & \hspace{10pt} \vdots \\
  & \xrightarrow[]{J_{a_1,a_2 \ldots a_k}}
    F(a_1\ldots a_k) = F(w)
\end{align*}
For brevity, we call the composition of these natural isomorphisms $J$ too.
(By the hexagon axiom, the order by which we
group the $a_l$'s together is immaterial.)
Then for each relation $r_j: v_{j,1} = v_{j,2}$,
we have
\[
  \gamma_j : U_{v_{j,1}} \xrightarrow[]{J} F(v_{j_1}) \xrightarrow[]{U(r_j)}
    F(v_{j_2}) \xrightarrow[]{J^\inv} U_{v_{j,2}}
\]

So $U_{g_i},U_{g_i^\inv},\gamma_j$ defines a pre-$G$-action
(as in \defref{def:pre_grp_action_genrel}),
and in fact satisfies the cocycle condition,
so it is a $G$-action (in the sense of \defref{def:grp_action_genrel}).
Indeed, suppose we have a path
$w_1 \xrightarrow[]{r_{j_1}} \ldots \xrightarrow[]{r_{j_p}} w_2$.
The resulting sequence of isomorphism compose to simply
$U_{w_1} \xrightarrow[]{J} F(w_1) \xrightarrow[]{F(r)} F(w_2)
  \xrightarrow[]{J^\inv} U_{w_2}$,
where $r$ is the unique morphism in $\Hom_\cG(w_1,w_2)$.
Since this natural isomorphism is independent of
the path we started with,
we see that we indeed have a $G$-action.\\

Conversely, suppose we are provided with a pre-$G$-action
$U_{g_i},U_{g_i^\inv}, \gamma_j$ (see \defref{def:pre_grp_action_genrel}).
We can easily construct a monoidal functor
\[
  (\wdtld{F}, \wdtld{J}): \wdtld{\cG} \to_\tnsr \Aut(\cA)
\]
as follows:
for a word $w = a_1\ldots a_k$,
we define
\[
  \wdtld{F}(w) := U_{a_1} \ldots U_{a_k}
\]
and make $\wdtld{J} = \id$.
For the simple morphism $q_{j,x,y} : x v_{j,1} y \to x v_{j,2} y$
we define
\[
  \wdtld{F}(q_{j,x,y}) :
    \wdtld{F}(x v_{j,1} y)
    \xrightarrow[]{\wdtld{F}(x) \gamma_j \wdtld{F}(y)}
    \wdtld{F}(x v_{j,2} y)
\]
Since a morphism $q$ in $\wdtld{\cG}$
is a sequence of simple $q_{j,x,y}$ and their reverses,
we take $\wdtld{F}(q)$ to be the composition
of the appropriate $\wdtld{F}(q_{j,x,y})$'s
(these are natural isomorphisms,
so if the sequence uses a reversed arrow,
we associate the inverse natural isomorphism).\\

It is easy to see that this gives a well-defined monoidal functor
$(\wdtld{F}, \wdtld{J}): \wdtld{\cG} \to_\tnsr \Aut(\cA)$.
Then we would get a monoidal functor $(F,J): \cG \to_\tnsr \Aut(\cA)$
if $(\wdtld{F}, \wdtld{J})$ factors through $\mathcal{Q}$:
\[
  \begin{tikzcd}
    \wdtld{\cG} \ar[d, "\mathcal{Q}"'] \ar[dr, "{(\wdtld{F},\wdtld{J})}"]\\
    \cG \ar[r, "{(F,J)}"']
      & \Aut(\cA)
  \end{tikzcd}
\]\\

In concrete terms, factoring through $\mathcal{Q}$
means that for each pair of unreduced words $w_1,w_2$,
any sequence of applications of relations
to get from $w_1$ to $w_2$
(i.e. a morphism in $\wdtld{\cG}$ from $w_1$ to $w_2$),
will result in the same natural isomorphism
$\wdtld{F}(w_1) \to \wdtld{F}(w_2)$.
In other words,
it is equivalent to the statement that $U_{g_i},U_{g_i^\inv},\gamma_j$
satisfy the cocycle condition of \defref{def:grp_action_genrel}.
In summary,

\begin{proposition}
\label{prp:grp_action_genrel_functor}
Given a monoidal functor
$(F,J): \cG \to_\tnsr \Aut(\cA)$,
the values of $F$ on generators and relations,
interpreted appropriately with $J$,
defines a group action in the sense of
\defref{def:grp_action_genrel}.\\

Conversely,
from a group action in the sense of
\defref{def:grp_action_genrel},
one can construct a monoidal functor
$(\wdtld{F},\wdtld{J}): \wdtld{\cG} \to_\tnsr \Aut(\cA)$
that factors through $\mathcal{Q}$,
and hence defines a monoidal functor
$(F,J): \cG \to_\tnsr \Aut(\cA)$.\\

Furthermore, beginning with some
$(F,J):\cG \to_\tnsr \Aut(\cA)$,
applying the first construction
and then the second,
we get a new monoidal functor
$(F',J'):\cG \to_\tnsr \Aut(\cA)$,
and $(F,J),(F',J')$
are naturally isomorphic as monoidal functors.
\end{proposition}

\begin{proof}
It remains to prove the last part,
that $(F,J),(F',J')$ are naturally isomorphic as
monoidal functors.
It is easy to check that $\eta_w = J^\inv$ works,
where recall we abuse notation for $J$ to also mean
successive applications of $J$'s:
for $w = a_1\ldots a_k$,
$F'(w) = U_w = F(a_1)\ldots F(a_k) \to \ldots \to F(a_1\ldots a_k) = F(w)$.

\comment{
Indeed, the fact that $\eta$ is a natural isomorphism
of monoidal functors follows easily from the constructions:

\[
  \begin{tikzcd}
    U_{w_1} \circ U_{w_2} \ar[r, "J"] \ar[d, "\eta \circ \eta"]
    & U_{w_1w_2} \ar[d, "\eta"] \\
    U_{w_1}' \circ U_{w_2}' \ar[r, equal, "J'"]
    & U_{w_1w_2}'
  \end{tikzcd}
\]
}
\end{proof}

\begin{remark}
If we take the trivial presentation
$G = \eval{\bar{g} \textrm{ for } g \in G |
        r_{\bar{g},\bar{h}}: \bar{g}\bar{h} = \overline{gh}}$,
we find that we recover the usual notion of a group action of
$G$ on $\cA$, as first discussed at the beginning of this section.
\hfill $\triangle$
\end{remark}

Finally, let us relate this back to the usual notion
of group action on a category:

\begin{corollary}
\label{cor:grp_action_genrel_usual}
Given a $G$-action on $\cA$ by generators and relations
$U_{g_i},U_{g_i^\inv}, \gamma_j$,
as in \defref{def:grp_action_genrel},
we can obtain a group action in the usual sense
by first applying \prpref{prp:grp_action_genrel_functor}
to get a monoidal functor $(F,J): \cG \to_\tnsr \Aut(\cA)$,
then choosing an inverse to the canonical $\pi: \cG \to \Cat(G)$.\\

More concretely, one chooses, for each $g\in G$,
a word $w_g$, and sets $T_g = F(w_g) = U_{w_g}$,
and for $g,h\in G$,
set $\gamma_{g,h}' = F(r) : T_g T_h = F(w_g)F(w_h) \to F(w_{gh}) = T_{gh}$,
where $r$ is the unique morphism in $\Hom_\cG(w_g w_h, w_{gh})$.
The uniqueness of $r$ also guarantees that the usual
cocycle condition is satisfied by $T_g, \gamma_{g,h}'$.\\

Any two inverses to $\pi$ are naturally isomorphic
by unique natural isomorphism,
so the resulting group actions are equivalent.
\end{corollary}

\comment{
Let us relate this back to the usual notion of
group actions on categories.
As mentioned before, one can get a usual group action
of $G$ on $\cA$ from a generators-and-relations description
by precomposing $(F,J): \cG \to_\tnsr \cA$
by some inverse of $\pi:\cG \to \Cat(G)$.
Concretely, such an inverse amounts to picking a representative
word $w_g$ for each $g\in G$,
so that $T_g = F(w_g)$.
Then we take
$\gamma_{g,h}' = F(r) : T_g T_h = F(w_g)F(w_h) \to F(w_{gh}) = T_{gh}$,
where $r$ is the unique morphism in $\Hom_\cG(w_g w_h, w_{gh})$.
The uniqueness of this morphism $r$ also guarantees
the usual cocycle condition for $T_g, \gamma_{g,h}'$.\\
}

\begin{remark}
\label{rmk:second_order_relations}
In general, it is not so clear how to check that
a pre-$G$-action in the sense of \defref{def:pre_grp_action_genrel}
satisfies the cocycle conditions of \defref{def:grp_action_genrel}
to be a $G$-action.
If we are to rely solely on algebra,
we need to understand the presentation $G=\eval{g_i|r_j}$
a little better,
somehow know the ``relations between relations",
not unlike the second syzygies of a module
as studied in homological algebra.\\

Let us clarify. Recall the CW-complex $\Xi$
defined above as justification for \defref{def:grp_action_genrel}.
Suppose we have checked that two paths $r,r'$ from $w_1$ to $w_2$
in $\Xi$ lead to the same natural isomorphism $U_{w_1} \to U_{w_2}$;
we call this a \emph{second-order relation}.
We attach a 2-cell along the loop $r^\inv r'$.
For any words $x,y$,
we automatically have that the two paths $x r y, x r' y$
from $x w_1 y$ to $x w_2 y$ in $\Xi$
also lead to the same natural isomorphism
$U_{x w_1 y} \to U_{x w_2 y}$,
so we also attach a 2-cell along the loop $(x r y)^\inv (x r' y)$.\\

Now suppose we have found several second-order relations,
so that upon attaching the corresponding 2-cell
and its ``translates" as above for each one,
each connected component of the new CW-complex $\Xi'$ is simply connected.
Then it is easy to see that this implies
that the pre-$G$-action we started with is actually a $G$-action.
Indeed, this means that any loop is contractible via
a sequence of 2-cells;
each time a loop homotopes through a 2-cell,
the second-order relation implies that the corresponding natural
isomorphism doesn't change.
Since the constant loop is associated the identity natural isomorphism,
we find that the original loop is also associated the identity,
hence the cocycle condition is satisfied.

\hfill $\triangle$
\end{remark}

\subsubsection{Equivariant Objects}

Suppose we are given a group $G$ acting on a category $\cA$
in the usual sense, as described at the beginning of the previous
\secref{sec:group_action}.
Then recall that a $G$-equivariant object of $\cA$
is an object $A \in \cA$ with isomorphisms
$\mu_g: T_g(A) \to A$,
and $\mu_-$ is required to be compatible with $T_-$
in the sense that
\[
  \begin{tikzcd}
    T_g(T_h(A)) \ar[r, "T_g(\mu_h)"] \ar[d, "\gamma_{g,h}"]
      & T_g \ar[d, "\mu_g"]\\
    T_{gh}(A) \ar[r, "\mu_{gh}"]
      & A
  \end{tikzcd}
\]
commutes for every pair of $g,h\in G$
(see for example \ocite{EGNO}*{Section 2.7}).\\

If $G$ and its action are presented by
generators and relations as in \defref{def:grp_action_genrel},
we would like to describe a $G$-equivariant object
by making use of the presentation.
In particular, if $G$ is finitely presented,
one would hope to be able to describe a
$G$-equivariant object by a finite amount of data
that satisfy a finite number of equations.
The goal of this subsection is to spell this out in detail.\\

We begin with a $G$-action on $\cA$ that is given by
the following data: for each generator $g_i$,
auto-equivalences $U_{g_i},U_{g_i^\inv}$,
and for each relation $r_j: v_{j,1}=v_{j,2}$,
we have $\gamma_j: U_{v_{j,1}} \to U_{v_{j,2}}$.\\

Suppose we are given an object $A\in \cA$,
and for each generator $g_i$ of $G$,
we are given an isomorphism
\[
  \mu_{g_i} : U_{g_i}(A) \to A
\]

From this, we can define
$\mu_{g_i^\inv} :
  U_{g_i^\inv}(A)
    \xrightarrow[]{U_{g_i^\inv}(\mu_{g_i}^\inv)}
      U_{g_i^\inv}(U_{g_i}(A))
    \xrightarrow[]{\gamma_{l_i}} A$
where $\gamma_{l_i}$ is the natural isomorphism
corresponding to the relation $g_i^\inv g_i = 1$.
We can then construct $\mu_w: U_w(A) \to A$
for a word $w = a_1\ldots a_k$ as the composition
\begin{align*}
  U_w(A) = U_{a_1}\ldots U_{a_k}(A)
    &\xrightarrow[]{U_{a_1}\ldots U_{a_{k-1}}(\mu_{a_k})}
      U_{a_1}\ldots U_{a_{k-1}}(A) \\
    & \hspace{10pt} \vdots \\
    & \xrightarrow[]{\mu_{a_1}} A
\end{align*}

Now just by construction, we have that,
for words $w_1,w_2$,
\[
  \begin{tikzcd}
    U_{w_1}(U_{w_2}(A)) \ar[r, "U_{w_1}(\mu_{w_2})"]
      \ar[d, equal]
    & U_{w_1}(A) \ar[d, "\mu_{w_1}"] \\
    U_{w_1w_2} \ar[r, "\mu_{w_1w_2}"]
    & A
  \end{tikzcd}
\]
so it seems that there is nothing to prove.\\

Thinking back to the usual definition of a $G$-equivariant
object as in the beginning of this section,
one should have exactly one $\mu_g: U_g(A) \to A$
for each $g$, so we should expect that if two words
$w_1,w_2$ are equal in $G$,
then $\mu_{w_1}$ should be equal to $\mu_{w_2}$.
This cannot literally be true, being morphisms from
different objects.
We may however impose that for each relation $r_j: v_{j,1} = v_{j,2}$,
\[
  R_j: \hspace{10pt}
  \begin{tikzcd}
    U_{v_{j,1}}(A) \ar[rr, "(\gamma_j)_A"] \ar[rd, "\mu_{v_{j,1}}"']
    & & U_{v_{j,2}}(A) \ar[ld, "\mu_{v_{j,2}}"] \\
    & A
  \end{tikzcd}
\]
This implies a similar commutative diagram
for $x r_j y$ relating $x v_{j,1} y$ to $x v_{j,2} y$:
\[
  \begin{tikzcd}
    U_{x v_{j,1} y}(A)
      \ar[r, equal]
      \ar[rrddd, "\mu_{x v_{j,1} y}"', bend right]
    & U_x U_{v_{j,1}} U_y(A)
      \ar[rr, "(U_x \gamma_j U_y)_A"]
      \ar[d, "U_x U_{v_{j,1}} (\mu_y)"']
    & & U_x U_{v_{j,2}} U_y(A)
      \ar[r, equal]
      \ar[d, "U_x U_{v_{j,2}} (\mu_y)"]
    & U_{x v_{j,2} y}(A)
      \ar[llddd, "\mu_{x v_{j,2} y}", bend left]
    \\
    & U_x U_{v_{j,1}}(A)
      \ar[rr, "(U_x \gamma_j)_A"]
      \ar[rd, "U_x (\mu_{v_{j,1}})"']
    & & U_x U_{v_{j,2}}(A)
      \ar[ld, "U_x (\mu_{v_{j,2}})"]
    \\
    & & U_x(A)
      \ar[d, "\mu_x"]
    \\
    & & A
  \end{tikzcd}
\]

\begin{definition}
\label{def:eqvrt_genrel}
Let an action of $G = \eval{g_i|r_j}$ on a category $\cA$
be given by generators and relations in the sense of
\defref{def:grp_action_genrel}.
Then a generators-and-relations presentation of a $G$-equivariant
object structure on $A$
consists of an isomorphism
$\mu_{g_i}: U_{g_i}(A) \to A$
for each $g_i$,
subject to the equation $R_j$ above
for each relation $r_j$.
\hfill $\triangle$\\
\end{definition}

Note that, in contrast to checking the cocycle condition
of \defref{def:grp_action_genrel},
there is no need to use ``second order relations"
(as discussed at the end of the previous section).\\

Finally,
let us relate this to the usual notion of group action on categories
and equivariant objects.
Recall that in \corref{cor:grp_action_genrel_usual},
we see that to go from a generators-and-relations
description of a group action to the usual one,
we make a choice of word $w_g$ for each $g\in G$,
and set $T_g = U_{w_g}$,
and $\gamma_{g,h}': T_g T_h \to T_{gh}$ is determined from $\gamma_j$'s.
Then we can obtain a usual $G$-equivariant structure on $A$
from a generators-and-relations description as in \defref{def:eqvrt_genrel}
as follows:
Take $\nu_g = \mu_{w_g}: T_g(A) = U_{w_g}(A) \to A$.
The compatibility of $\nu_-$ with the group action $T_-,\gamma'$
(i.e. whether $\nu_-$ defines a $G$-equivariant structure on $A$)
is equivalent to the commutativity of the diagram
\[
  \begin{tikzcd}
    U_{w_g}(U_{w_h}(A))
      \ar[r, "U_{w_g}(\mu_{w_h})"]
      \ar[d, "\gamma"]
    & U_{w_g}(A)
      \ar[d, "\mu_{w_h}"] \\
    U_{w_{gh}}(A) \ar[r, "\mu_{w_{gh}}"]
    & A
  \end{tikzcd}
\]
where $\gamma$ is the composition of some $\gamma_j$'s
corresponding to a path from $w_g w_h$ to $w_{gh}$.
We see that this diagram is commutative
precisely because we have imposed the equations $R_j$
on $\mu_j$.

\end{document}